\documentclass[reqno, a4paper, 11pt]{amsart}

\usepackage{amsmath,amssymb,latexsym, amsfonts}
\usepackage{verbatim}
\usepackage{times}
\usepackage{hyperref}
\usepackage{latexsym}
\usepackage{graphicx}
\usepackage[2emode]{psfrag}
\usepackage{mathrsfs}
\usepackage{stmaryrd}
\usepackage{amsthm}
\usepackage{amsmath}
\usepackage[all]{xy}
\usepackage{enumerate}
\usepackage[latin1]{inputenc}
\usepackage{lscape}
\usepackage{a4wide}
\usepackage[usenames]{color}
\usepackage{cancel}
\usepackage[normalem]{ulem}
\usepackage{txfonts}
\usepackage{wasysym}
\usepackage{capt-of}

\usepackage{rotating}

\usepackage{cancel}
\usepackage[normalem]{ulem}

\numberwithin{equation}{section}

\theoremstyle{plain}

\newtheorem{theorem}{Theorem}[section]
\newtheorem{thm}[theorem]{Theorem}
\newtheorem{proposition}[theorem]{Proposition}
\newtheorem{prop}[theorem]{Proposition}
\newtheorem{lemma}[theorem]{Lemma}

\newtheorem{corollary}[theorem]{Corollary}

\newtheorem{lemdfn}[theorem]{Lemma and Definition}

\theoremstyle{definition}
\newtheorem{definition}[theorem]{Definition}

\theoremstyle{remark}
\newtheorem{example}[theorem]{Example}

\newtheorem{rem}[theorem]{Remark}

\newcommand{\ahha}{{\scriptscriptstyle{A}}}
\newcommand{\behhe}{{\scriptscriptstyle{B}}}
\newcommand{\cehhe}{{\scriptscriptstyle{C}}}

\newcommand{\emme}{{\scriptscriptstyle{M}}}
\newcommand{\enne}{{\scriptscriptstyle{N}}}
\newcommand{\erre}{{\scriptscriptstyle{R}}}
\newcommand{\esse}{{\scriptscriptstyle{S}}}
\newcommand{\tehhe}{{\scriptscriptstyle{T}}}
\newcommand{\pehhe}{{\scriptscriptstyle{P}}}
\newcommand{\quhhu}{{\scriptscriptstyle{Q}}}

\newcommand{\cakka}{{\scriptscriptstyle{\cH}}}
\newcommand{\ckppa}{{\scriptscriptstyle{\cK}}}

\newcommand{\ga}{\alpha} 
\newcommand{\gb}{\beta}  
\newcommand{\gd}{\delta} 
\newcommand{\gD}{\Delta}

\newcommand{\gve}{\varepsilon}

\newcommand{\go}{\omega}

\newcommand{\gs}{\sigma}

\newcommand{\gvt}{\vartheta} 

\newcommand{\cF}{{\mathcal F}}
\newcommand{\cG}{{\mathcal G}}
\newcommand{\cH}{{\mathcal H}}
\newcommand{\cI}{{\mathcal I}}
\newcommand{\cJ}{{\mathcal J}}
\newcommand{\cK}{{\mathcal K}}

\newcommand{\cM}{{\mathcal M}}

\newcommand{\id}{{\rm id}}

\newcommand{\pr}{{\rm pr} \,}

\newcommand{\coinv}{\scriptscriptstyle{{\rm coinv}}}
\newcommand{\co}{\scriptscriptstyle{{\rm co}}}

\newcommand{\due}[3]{{}_{{#2 \!}} {#1}_{{\! #3}}\,}    

\newcommand{\pl}{\partial}

\newcommand{\rmref}[1]{{\rm (}\ref{#1}{\rm )}}

\newcommand{{\Hl}}{{H^{\ell}}} 
\newcommand{{\mHop}}{{m_{H^{\rm op}}}} 
\newcommand{{\Hop}}{{H^{\rm op}}} 
\newcommand{{\mUop}}{{m_{U^{\rm op}}}} 
\newcommand{{\Uop}}{{U^{\rm op}}}
\newcommand{{\mVop}}{{m_{V^{\rm op}}}} 
\newcommand{{\Vop}}{{V^{\rm op}}}  
\newcommand{{\Ae}}{{A^{\rm e}}}
\newcommand{{\Be}}{{B^{\rm e}}}
\newcommand{{\Aop}}{{A^{\rm op}}}
\newcommand{{\Aope}}{({A^{\rm op}})^{\rm e}}
\newcommand{{\Aopl}}{{A^{\rm op}_\pl}}

\newcommand{{\Bop}}{{B^{\rm op}}}
\newcommand{{\Bope}}{({B^{\rm op}})^{\rm e}}
\newcommand{{\Bpl}}{{B_\pl}}

\newcommand{{\Reh}}{{R^{\rm e}}}
\newcommand{{\Se}}{{S^{\rm e}}}
\newcommand{{\Rop}}{{R^{o}}}
\newcommand{{\Sop}}{{S^{o}}}

\newcommand{{\Pe}}{{P^{\rm e}}}
\newcommand{{\Qe}}{{Q^{\rm e}}}

\newcommand{{\op}}{{{o}}}
\newcommand{{\coop}}{{{\rm coop}}}
\newcommand{{\sop}}{{*^{\rm op}}}

\newcommand{\lact}{{\,\raise1pt\hbox{$\scriptscriptstyle{\rhd}$} \, }}                  %
\newcommand{\ract}{{\,\raise1pt\hbox{$\scriptscriptstyle{\lhd}$} \, }}                  
\newcommand{\blact}{{\,\raise1pt\hbox{$\scriptscriptstyle{\blacktriangleright}$} \, }}  %
\newcommand{\bract}{{\,\raise1pt\hbox{$\scriptscriptstyle{\blacktriangleleft}$} \, }}   %

\newcommand{{\gog}}{{G \rightrightarrows G_0}}

\newcommand{{\rra}}{\rightrightarrows}

\newcommand{{\lra}}{ \longrightarrow  }
\newcommand{{\lla}}{ \longleftarrow }
\newcommand{{\lma}}{ \longmapsto  }

\def\kasten#1{\mathop{\mkern0.5\thinmuskip
\vbox{\hrule
      \hbox{\vrule
            \hskip#1
            \vrule height#1 width 0pt
            \vrule}%
      \hrule}%
\mkern0.5\thinmuskip}}

\newcommand{\bx}{{\kasten{6pt}}}

\newcommand{{\bull}}{{\scriptscriptstyle{\bullet}}}

\newcommand{{\qqquad}}{{\quad\quad\quad}}

\input{xy}
 \xyoption{all}

\newcommand{\tensor}[1]{\otimes_{\scriptscriptstyle{#1}}}
\newcommand{\comdtensor}[1]{\otimes^{\scriptscriptstyle{#1}}}
\newcommand{\cotensor}[1]{\, \square_{\scriptscriptstyle{#1}} \,}

\newcommand{\td}[1]{\tilde{#1}}
\newcommand{\bara}[1]{\overline{#1}}

\newcommand{\fk}[1]{\mathfrak{#1}}
\newcommand{\Sf}[1]{\mathsf{#1}}

\newcommand{\B}[1]{\boldsymbol{#1}}

\newcommand{\rcomod}[1]{\mathsf{Comod}_{#1}}
\newcommand{\lcomod}[1]{{}_{#1}\mathsf{Comod}}
\newcommand{\bicomod}[2]{{}_{#1}\mathsf{Bicomod}_{#2}}
\newcommand{\rmod}[1]{\mathsf{Mod}_{#1}}

\newcommand{\lcoaction}[2]{\boldsymbol{\lambda}^{\scriptscriptstyle{#1}}_{\scriptscriptstyle{#2}}}
\newcommand{\rcoaction}[2]{\boldsymbol{\rho}^{\scriptscriptstyle{#1}}_{\scriptscriptstyle{#2}}}
\newcommand{\can}[2]{\mathsf{can}_{\scriptscriptstyle{#1},\, \scriptscriptstyle{#2}}}
\newcommand{\shopf}[1]{{}_{\scriptscriptstyle{\Sf{s}}}\mathcal{#1}}
\newcommand{\thopf}[1]{\mathcal{#1}_{\scriptscriptstyle{\Sf{t}}}}

\newcommand{\Spectre}[1]{\mathsf{Spec}(#1)}
\newcommand{\lPB}[2]{\mathsf{PB}^{\scriptscriptstyle{\ell}}(#1,#2)}
\newcommand{\rPB}[2]{\mathsf{PB}^{\scriptscriptstyle{r}}(#1,#2)}
\newcommand{\bPB}[2]{\mathsf{PB}^{\scriptscriptstyle{b}}(#1,#2)}

\newcommand{\Ff}{{\mathscr F}}
\newcommand{\Gg}{{\mathscr G}}
\newcommand{\Hh}{{\mathscr H}}

\newcommand{\Kk}{{\mathscr K}}

\newcommand{\Oo}{{\mathscr O}}

\newcommand{\Uu}{{\mathscr U}}

\newcommand{\Sscript}[1]{\scriptscriptstyle{#1}}
\newcommand{\Go}{\Gg_{\Sscript{0}}}
\newcommand{\Ga}{\Gg_{\Sscript{1}}}
\newcommand{\Ho}{\Hh_{\Sscript{0}}}
\newcommand{\Ha}{\Hh_{\Sscript{1}}}
\newcommand{\Ko}{\Kk_{\Sscript{0}}}
\newcommand{\Ka}{\Kk_{\Sscript{1}}}
\newcommand{\GSets}{\Gg\text{-}{\sf Sets}}
\newcommand{\HSets}{\Hh\text{-}{\sf Sets}}
\newcommand{\timesG}{\underset{\Sscript{\Go}}{\times}}
\newcommand{\Orb}[2]{\mathscr{O}rb_{\Sscript{\mathscr{#1}}}(#2)}

\begin{document}
\allowdisplaybreaks

\title[Morita theory for Hopf algebroids, principal bibundles, and weak equivalences.]{Morita theory for Hopf algebroids, principal bibundles, and weak equivalences} 

\author{Laiachi El Kaoutit}
\author{Niels Kowalzig}

\address{L.E.K.: 
Universidad de Granada, Departamento de \'Algebra and IEMath-Granada, Facultad de Educaci\'on, Econom\'ia  
y Tecnolog\'ia. Campus De Ceuta.
Cortadura del Valle s/n. E-51001 Ceuta, Espa\~na}
\email{kaoutit@ugr.es}

\address{N.K.: Dipartimento di Matematica, Universit\`a degli Studi di Roma La
Sapienza, P.le Aldo Moro 5, 00185 Roma, Italia}
\email{kowalzig@mat.uniroma1.it}

\thanks{L.~El Kaoutit was partially supported by grants MTM2013-41992P
from the Spanish Ministerio de Educaci\'{o}n y Ciencia and  P11-FQM-7156  from the Junta de Andaluc\'ia and acknowledges  hospitality and travel support granted by the INdAM.  The research of N.~Kowalzig was funded by an INdAM-COFUND Marie Curie grant.}

\vspace{-2cm}
\begin{abstract}
We show that  two flat commutative Hopf algebroids are Morita equivalent if and only if they are weakly equivalent and if and only if there exists a principal bibundle connecting them. This gives a positive answer to a conjecture due to Hovey and Strickland. 
We also prove that principal (left) bundles lead to a bicategory together with  a $2$-functor from flat Hopf algebroids to trivial principal bundles.  This turns out to be the universal solution for $2$-functors which send weak equivalences to invertible $1$-cells.  
Our approach can be seen as an algebraic counterpart to  Lie groupoid Morita theory. 
\end{abstract}

\maketitle

\vspace{-1.2cm}

\begin{footnotesize}
\tableofcontents
\end{footnotesize}

\section{Introduction}

\subsection{Aims and objectives}

The two fundamental concepts around which this article is orbiting are those of {\em weak equivalence} and {\em Morita equivalence}. Recall from, {\em e.g.}, \cite[\S5]{MoeMrc:LGSAC} 
that two Lie groupoids $\Gg$ and $\Gg'$ are called weakly equivalent if there exist weak equivalences $\phi: \Hh \to \Gg$ and $\phi': \Hh \to \Gg'$ for
some third Lie groupoid $\Hh$ (see again {\em op.~cit.}~ for the precise definition of a weak equivalence $\phi$). For instance, the groupoids associated to two  atlases of a manifold (or two transverse atlases of a foliated manifold)  are weakly equivalent; each groupoid associated to a principal bundle of  a Lie group $G$ and base manifold $\cM$ is weakly equivalent to the unit Lie groupoid  $\Uu(\cM)$.

As a definition of Morita equivalence of two (Lie) groupoids might serve reversing the (classical) Morita theorem, that is, the requirement that their categories of representations (quasi-coherent $\Gg$-sheaves of $\Bbbk$-modules) are equivalent as symmetric monoidal categories. This leads to a quite general idea of equivalence which can be applied to any mathematical object that allows for the notion of ``representation'', or, more generally, (co)modules. 

That the two notions of weak equivalence and Morita equivalence are essentially the same and also imply the presence of a principal bibundle (in an appropriate sense) is a well-known fact for (Lie) groupoids (in fact, the terminology varies and often coincides, which adds somewhat to the confusion), see \cite{MuhRenWil:EAIFGCSA, Hae:GDHEC, Mrc:FOTBATAHSM}. Note, however, that in the first of these references the respective concept of principal bundle slightly differs from the latter two. 
Taking Lie groupoids as objects, one constructs,  together with the isomorphism classes of principal bundles (as morphisms, sometimes called {\em Hilsum-Skandalis maps}) and equipped with the tensor product, a category,  sometimes called the {\em Morita category}. Moreover, there is a functor from the category of Lie groupoids to this Morita category which transforms weak equivalences to isomorphisms  that establishes a universal solution for functors having this property.

Roughly speaking, \emph{commutative} Hopf algebroids can be seen as  presheaves of groupoids on affine schemes: the datum of a {\em flat} Hopf algebroid is equivalent to the datum of a certain stack with a specific presentation \cite{Naumann:07, FalCha:DOAV}. 
In this perspective, one can establish an equivalence between (right) comodules over a Hopf algebroid and quasi-coherent sheaves with a groupoid action \cite[Thm.~2.2]{Hovey:02}.

Hopf algebroids
were introduced 
in algebraic topology (see, {\em e.g.}, \cite{Rav:CCASHGOS}) as a {\em cogroupoid} kind of object,
which motivates the following definitions taken from \cite[Def.\ 6.1]{HovStr:CALEHT} resp.\ \cite{Hovey:02}.
For the necessary ingredients and notation used therein we refer to the main text.

\begin{definition}
\label{quellala1}
\label{quellala2} Let $(A,\cH)$ and $(B,\cK)$ be two flat Hopf algebroids.
\begin{enumerate}
\item
A morphism $(A,\cH) \to (B,\cK)$ 
 is said to be a \emph{weak equivalence} if and only if
the respective induction functor
$\rcomod{\cH} \to \rcomod{\cK}$ establishes an equivalence of categories. The Hopf algebroids  $(A,\cH)$ and $(B,\cK)$ are said to be \emph{weakly equivalent} if there is a diagram 
\vspace{-.1cm}
$$
\xymatrix@R=8pt{ &   (C,\cJ)  & \\ (A,\cH) \ar@{->}
[ur] &  & \ar@{->}[ul] (B,\cK) }
\vspace{-.1cm}
$$ 
of weak equivalences of Hopf algebroids. 
\item
Two flat Hopf algebroids are said to be {\em Morita equivalent} if their categories of (right) comodules are equivalent as symmetric monoidal categories.
\end{enumerate}
\end{definition}
For instance, the existence of  a weak equivalence implies Morita equivalence since induction functors are always symmetric monoidal functors.

In the context of Hopf {\em algebras}, the second part in the above definition appeared in \cite[Def.~3.2.3]{Schau:HGABGE} baptised {\em monoidal Morita-Takeuchi} equivalence therein but also before in \cite[Def.~5.6]{Schau:HBGE}, where such a property was called {\em monoidal co-Morita} equivalence.   
Let us also mention that a Morita theory for certain cocommutative Hopf algebroids  (so-called {\em \'etale Hopf algebroids}) was developped in \cite{Mrc:THAOFOEGATPME} using a different notion of bundles (called {\em principal bimodules}). Furthermore, the idea of describing Morita theory in the language of bicategories was explained, for example, in \cite{Lan:BOOAAPM} for various contexts, such as rings, $C^*$-algebras, von Neumann algebras, Lie groupoids, symplectic groupoids, and Poisson manifolds.


\subsection{Main results}
Transferring the above statements from Lie groupoids to the case of commutative Hopf algebroids will be the main task (and result)
of this article, summarised as follows:

{\renewcommand{\thetheorem}{{A}}

\begin{thm}
\label{processocivile}
Let $(A,\cH)$ and $(B,\cK)$ be two flat Hopf algebroids. The following are equivalent:
\begin{enumerate}[(1)]
\item $(A,\cH)$ and $(B,\cK)$ are Morita equivalent.
\item There is a principal bibundle connecting $(A,\cH)$ and $(B,\cK)$.
\item $(A,\cH)$ and $(B,\cK)$ are weakly equivalent.
\end{enumerate}
\end{thm}
}
 
One might be tempted to think that these results can be obtained by simply dualising the usual techniques in the groupoid case (which we recall in \S\ref{sec:PSets}, Theorem \ref{thm:AG}) but things turn out to be more intricate: one of the main obstacles in mimicking the groupoid case is the construction of orbit spaces which correspond to quotients of affine schemes, which is a subtle concept with its own challenges.  
In contrast to that, our arguments make large use of cotensor products of comodule algebras in correspondence to these quotients of affine schemes, which might seem technical at first sight but proves useful in this context.

The subsequent picture shows all implications between $(1),  (2)$, and $(3)$ in the above theorem that we will explore in the main text: 
\begin{small}
\begin{equation*}
\xymatrix{ &  & (1)  \ar@/_0.8pc/@{=>}_-{{\scriptstyle{\text{Proposition } \ref{prop:llegaras}}}}[lldd]   & &  \\ & & & &  \\ (2) \ar@/_0.8pc/@{=>}_-{\scriptstyle{\text{Theorem } \ref{aromanflower}}}[rruu] \ar@{<=>}_-{\scriptstyle{\text{Proposition } \ref{prop:PHPKP}}}[rrrr]  & & &  & (3)
\ar@/_1pc/@{=>}_-{{\scriptstyle{\text{trivial }}}}[lluu] }
\end{equation*}
\begin{center}
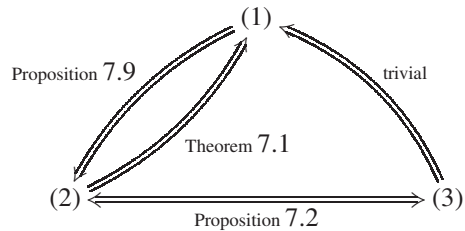

\vspace{0.05cm}
\captionof{figure}{\footnotesize Paths in the proof of Theorem \ref{processocivile}}\label{ahaaha}
\end{center}
\end{small}
\vspace{-.5cm}

In particular, the step $(1) \Rightarrow (3)$ in the above Theorem \ref{processocivile} was conjectured in \cite[Conj.~6.3]{HovStr:CALEHT}: more precisely, Hovey and Strickland conjectured that in case the category of $\cH$-comodules is equivalent to the one of comodules over $\cK$, then the two Hopf algebroids $(A,\cH)$ and $(B, \cK)$ 
are connected by a chain of weak equivalences, and we show that this chain can be taken to be of length $2$.

By a  chain of weak equivalences of length $n \geq 2$ we mean a zig-zag of weak equivalences in the sense of \cite[Def.~7.9.1]{Hirschhorn}, up to the equivalence transformations  given in \cite[\S14.4]{Hirschhorn}. The key here is Proposition \ref{prop:1out2}, which shows  that any  zig-zag of weak equivalences of the form $\xymatrix{ \bullet & \ar@{->}^-{}[r] \bullet \ar@{->}^-{}[l] & \bullet }$
can be completed to a diagram of weak equivalences  having the form 
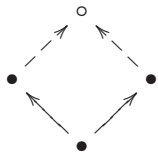
\begin{small}
$$
\xymatrix@R=15pt@C=15pt{  & \circ & \\ \bullet \ar@{-->}^-{}[ru]   & & \ar@{-->}^-{}[lu] \bullet \\  & \ar@{->}^-{}[lu] \bullet \ar@{->}^-{}[ru] & 
}
$$
\end{small}
which is commutative up to a $2$-isomorphism (a property dual to condition (BF3) in \cite[p.~254]{Pro:EASABOF}).

In this way, any chain of weak equivalences (in the above sense) between two flat Hopf algebroids $(A,\cH)$ and $(B,\cK)$ can be transformed
to one of the form 
\begin{footnotesize}
$$
\mathfrak{Z}_{\scriptstyle{k+2}}:\;\;
\xymatrix@R=8pt@C=8pt{ 
& \scriptstyle{(D_1,\,\cI_1)}   &  &  \scriptstyle{(D_2,\,\cI_2)} & & \scriptscriptstyle{\cdots\cdots\cdots\cdots\cdots\cdots} & &   \scriptstyle{(D_{k},\,\cI_{k})}  & &\scriptstyle{(D_{k+1},\,\cI_{k+1})}  & \\
 \scriptstyle{(A,\cH)} \ar@{->}^-{}[ru] & & \ar@{->}^-{}[ru]  \scriptstyle{(C_1,\,\cJ_1)} \ar@{->}^-{}[lu] &  & \ar@{->}^-{}[lu] \scriptstyle{(C_{2},\,\cJ_{2})}   & \scriptscriptstyle{\cdots\cdots\cdots\cdots\cdots\cdots} &  \scriptstyle{(C_{k-1},\,\cJ_{k-1})} \ar@{->}^-{}[ru] & & \ar@{->}^-{}[ru] \scriptstyle{(C_{k},\,\cJ_{k})}  \ar@{->}^-{}[lu]& &  \ar@{->}^-{}[lu] \scriptstyle{(B,\cK)}
}
$$
\end{footnotesize}
of length $2(k+1)$, which, in turn, can be completed to the following isosceles triangle 
\begin{small}
$$
\xymatrix@R=10pt@C=12pt{  & & & & & \scriptstyle{(C_{k1},\,\cJ_{k1})} & & & & &   \\ 
& & & & \ar@{->}^-{}[ru]  \scriptstyle{(C_{(k-1)1},\,\cJ_{(k-1)1})} & & \ar@{->}^-{}[lu]  \scriptstyle{(C_{(k-1)2},\,\cJ_{(k-1)2})} & & & & \\ 
& & &   \ar@{.}^-{}[ru] \scriptstyle{(C_{11},\,\cJ_{11})} & & \ar@{.}^-{}[ru] \scriptscriptstyle{\!...\!} \ar@{.}^-{}[lu] & &  \ar@{.}^-{}[lu] \scriptstyle{(C_{1k},\,\cJ_{1k})}  & & & \\
& & \scriptstyle{(D_1,\,\cI_1)} \ar@{->}^-{}[ru] &  &  \ar@{->}^-{}[lu]\scriptstyle{(D_2,\,\cI_2)} \ar@{.}^-{}[ru] &  & \ar@{.}^-{}[lu] \scriptstyle{(D_{k},\,\cI_{k})}  \ar@{->}^-{}[ru]  & & \scriptstyle{(D_{k+1},\,\cI_{k+1})}  \ar@{->}^-{}[lu] & & \\
& \scriptstyle{(A,\,\cH)}  \ar@{->}^-{}[ru]  &  &   \ar@{->}^-{}[lu] \scriptstyle{(C_1,\,\cJ_1)} \ar@{->}^-{}[ru]    &   &  \ar@{.}^-{}[lu] \scriptscriptstyle{\!...\!} \ar@{.}^-{}[ru] & &  \ar@{->}^-{}[lu] \scriptstyle{(C_k,\,\cJ_k)} \ar@{->}^-{}[ru] & &\scriptstyle{(B,\,\cK)} \ar@{->}^-{}[lu] & 
}
$$
\end{small}
of $(k+2)$ vertices on each side. Such a triangle is obtained by constructing ${k(k+1)}/{2}$ new flat Hopf algebroids being essentially two-sided translation Hopf algebroids built  from trivial principal bundles.

The notion of {\em (quantum) principal bundle} that appears as a crucial ingredient  in Theorem \ref{processocivile} is a relatively straightforward extension of the corresponding concept for Hopf algebras as introduced in \cite{BrzMaj:QGGTOQS}, see also \cite{Brz:TMIQPP}. In \cite[\S3.2.4]{Schau:HGABGE}, again in the realm of Hopf algebras, these objects were called {\em bi-Galois objects} and the corresponding implications $(1) \Leftrightarrow (2)$ of Theorem \ref{processocivile} were shown. 
As a matter of fact, in many examples constructing bi-Galois objects or principal bundles has turned out to be a practicable way to establish monoidal equivalences between comodule categories; as a concrete illustration, see, for example, \cite{Mas:CDAGOFSCHAOFD, Bic:TRCOTQGOANDBF}.  Analogous objects in sheaf theory 
are known under the name of \emph{(bi)torsors}, see \cite{DemGab:GATIGAGGC}.

In fact, we gather flat Hopf algebroids and principal bundles along with their morphisms in a bicategory. More precisely, in Proposition \ref{piazzadellorologio} we prove that the data given by
\begin{center}
\begin{enumerate}[\quad \raisebox{1pt}{$\scriptstyle{\bullet}$}]
\item
flat Hopf algebroids (as $0$-cells),
\item
left principal bundles (as $1$-cells), 
\item
as well as morphisms of left principal bundles (as $2$-cells)
\end{enumerate}
\end{center}
define a bicategory, denoted by $\mathsf{PB}^{\scriptscriptstyle{\ell}}$. 
The bicategories of analogously constructed right resp.\ two-sided principal bundles (or bibundles) are denoted by $\mathsf{PB}^{\scriptscriptstyle{r}}$ and $\mathsf{PB}^{\scriptscriptstyle{b}}$, respectively.  As in classical situations, 
for two $0$-cells $(A,\cH)$ and $(B,\cK)$, the category $\lPB{\cH}{\cK}$ turns out to be a groupoid. 
This leads to the structure of a \emph{bigroupoid} on the bicategory $\mathsf{PB}^{\scriptscriptstyle{b}}$, and hence to a \emph{categorical group} (or \emph{bigroup}) structure on each category $\bPB{\cH}{\cH}$, see, for instance, \cite{Noo:NOTGTGACM}. 

Applying Theorem \ref{processocivile} above to a single flat Hopf algebroid yields the following result:  

{\renewcommand{\thetheorem}{{B}}
\begin{theorem}
\label{thm:B}
Let $(A,\cH)$ be a flat Hopf algebroid and denote by $\mathscr{U}(\cH)$ its associated  principal unit bibundle. 
Then the category $\big(\Sf{Aut}^{\scriptscriptstyle{\otimes}}(A,\cH), \B{\circ},\id_{\scriptscriptstyle{\rcomod{\cH}}}\big)$  of symmetric monoidal auto-equivalences of right $\cH$-comodules with morphisms given by natural tensor transformations forms a categorical group, and the functors 
\begin{equation*}
\begin{array}{rclrcl}
\big(\Sf{Aut}^{\scriptscriptstyle{\otimes}}(A,\cH), \B{\circ},\id_{\scriptscriptstyle{\rcomod{\cH}}}\big) 
&\longrightarrow& 
\big(\bPB{\cH}{\cH},\cotensor{\cH}, \mathscr{U}(\cH) \big), &
 \Ff &\longmapsto& \Ff(\cH) 
\\  
\big(\bPB{\cH}{\cH},\cotensor{\cH}, \mathscr{U}(\cH) \big)   &\longrightarrow& 
\big(\Sf{Aut}^{\scriptscriptstyle{\otimes}}(A,\cH), \B{\circ},\id_{\scriptscriptstyle{\rcomod{\cH}}}\big), & (P,\alpha,\beta) &\longmapsto& -\cotensor{\cH}P   
\end{array}
\end{equation*}
establish a monoidal equivalence of categorical groups.  
\end{theorem}
}

Moreover, it turns out that there is a $2$-functor 
$$
\mathscr{P}: 2\text{-}\mathsf{HAlgd} \longrightarrow \mathsf{PB}^{\ell \, \co}
$$
from the $2$-category of flat Hopf algebroids to the conjugate of $\mathsf{PB}^{\scriptscriptstyle{\ell}}$, which sends any $1$-cell  $\B{\phi}: (A,\cH) \to (B, \cK)$ to its associated trivial left principal bundle  $\mathscr{P}(\B{\phi})=\cH\tensor{\scriptscriptstyle{\phi}}B$, that is, the pull-back of the unit bundle $\mathscr{U}(\cH)$. 
A $1$-cell $\B{\phi}$ in $2\text{-}\mathsf{HAlgd}$ is a weak equivalence if and only if $\mathscr{P}(\B{\phi})$ is an invertible $1$-cell in $\mathsf{PB}^{\ell \, \co}$, \emph{i.e.},  is part of an internal equivalence. We then present the pair $(\mathsf{PB}^{\scriptscriptstyle{\ell}}, \mathscr{P})$ as the universal solution with respect to this property:

{\renewcommand{\thetheorem}{{C}}

\begin{thm}\label{thm:Ramadan}
Let $\mathscr{F}: 2\text{-}\Sf{HAlgd} \to \mathscr{B}$ be a $2$-functor which sends weak equivalences to invertible $1$-cells. Then,  up to isomorphism (of $2$-functors), there is a unique $2$-functor $\td{\mathscr{F}}$ such that the diagram 
\begin{equation*}
\xymatrix@R=20pt{ 2\text{-}\Sf{HAlgd} \ar@{->}_-{\mathscr{F}}[rrd]   \ar@{->}^-{\mathscr{P}}[rr] & & \mathsf{PB}^{\ell \, \co} \ar@{.>}^-{\td{\mathscr{F}}}[d]  \\ & & \mathscr{B} }
\end{equation*}
commutes up to an isomorphism of $2$-functors. 
\end{thm}
}

We finally want to mention that this universality leads to a kind of calculus of fractions in 
the $2$-category $2\text{-}\Sf{HAlgd}$ with respect to weak equivalences in a  sense  ``dual"  to the approach in \cite{Pro:EASABOF}.

\bigskip

\noindent \textbf{Acknowledgements.}
It is a pleasure to thank  Alessandro Ardizzoni, Federica Galluzzi, and Fabio Gavarini for stimulating discussions and useful comments. We are also grateful to the referee for careful reading and useful comments.

\section{Abstract groupoids and principal bisets revisited}
\label{sec:PSets}
In this section we expose some basic results on abstract groupoids which are going to serve as a sort of motivation for the forthcoming sections dealing with flat Hopf algebroids. The exposition we follow here is parallel to \cite{MoeMrc:LGSAC} dealing with Lie 
groupoids, as well as to  \cite{Kaoutit:2015}.

\subsection{Principal bisets and orbit sets}\label{ssec:PB}
A \emph{groupoid} (or \emph{abstract groupoid}) is a small category where each morphism is an isomorphism. That is, a pair of sets $\mathscr{G}:=(\Gg_{\Sscript{1}}, \Gg_{\Sscript{0}})$ 
with a diagram 
$\xymatrix@C=35pt{\Gg_{\Sscript{1}}\ar@<0.70ex>@{->}|-{\scriptscriptstyle{\sf{s}}}[r] \ar@<-0.70ex>@{->}|-{\scriptscriptstyle{\sf{t}}}[r] & \ar@{->}|-{ \scriptscriptstyle{\iota}}[l] \Gg_{\Sscript{0}}}$,
where $\Sf{s}$ and $\Sf{t}$ are the source resp.\ the target of a given arrow, and where $\iota$ assigns to each object its identity arrow; together with an associative and unital multiplication  $\Gg_{\Sscript{2}}:= \Gg_{\Sscript{1}}\, \due  \times {\Sscript{\Sf{s}}} {\, \Sscript{\Sf{t}}} \, \Gg_{\Sscript{1}} \to \Gg_{\Sscript{1}}$  as well as a map $\Gg_{\Sscript{1}} \to \Gg_{\Sscript{1}}$, which associates to each arrow its inverse.

Recall that for a groupoid $\Gg$ one can define its set of orbits as follows: for any $a \in \Gg_{0}$, one considers either the set
$$
\Oo_{\scriptscriptstyle{a}} \,= \, \Sf{t}\big( \Sf{s}^{-1}(a)\big),
$$ or $\Oo_{\scriptscriptstyle{a}} = \Sf{s}\big( \Sf{t}^{-1}(a)\big)$. 
An equivalence relation on $\Gg_{0}$ is now defined by setting $a \sim b $ if and only if $\Oo_{\scriptscriptstyle{a}}=\Oo_{\scriptscriptstyle{b}}$. The \emph{set of orbits} of $\Gg$ is the quotient set $\Gg_{\Sscript{0}}/\sim$, which is often denoted by $\Gg_{\Sscript{0}}/\Gg$. In other words this is the set of all connected components of $\Gg$.

A more general situation arises when a groupoid acts on a set, which we will refer to as groupoid-set. Specifically, 
recall that a  \emph{left $\Gg$-action} of a groupoid $\Gg$ on a set $X$ consists of two maps $\alpha: X \to \Gg_{\Sscript{0}}$ (\emph{the structure map}) and   $\lambda: \Gg_{\Sscript{1}}\, \due \times {\scriptscriptstyle{\Sf{s}}} {\, \scriptscriptstyle{\alpha}} X \to X, \ (g,x) \mapsto gx$ (\emph{the action map}), satisfying
$$ 
\alpha(gx)\,=\, \Sf{t}(g),\quad  \iota_{\scriptscriptstyle{\alpha(x)}} x\,=\, x,\quad g'(gx)\,=\, (g'g)x. 
$$
The pair $(X,\alpha)$ is called \emph{a left $\Gg$-set}. In this way, one can define the \emph{left translation groupoid} $\Gg \lJoin X$ with $\Gg_{\Sscript{1}}\, \due \times {\scriptscriptstyle{\Sf{s}}} {\, \scriptscriptstyle{\alpha}} X$ as set of arrows and $X$ as set of objects. This is  the so-called \emph{semi-direct product groupoid}, see \cite[p.~163]{MoeMrc:LGSAC}. The \emph{orbit set} $X/\Gg$ of the left $\Gg$-set $(X,\alpha)$ is by definition the orbit set of the translation groupoid  $\Gg \lJoin X$.   For a given object $x \in X$, the equivalence class, that is, the orbit of $x$, will be denoted by $\mathscr{O}rb_{\Sscript{\Gg}}(x)$.

Morphisms between left $\Gg$-sets (or \emph{$\Gg$-equivariant maps}) are defined in the obvious way, and the category so-obtained is denoted by $\GSets$ and called {\em left groupoid-sets}. The category ${\sf Sets}\text{-}\Gg$ of right groupoid-sets is similarly defined. These categories are in fact symmetric monoidal categories, and one can observe that $\GSets$ is isomorphic to ${\sf Sets}\text{-}\Gg$. Explicitly, the \emph{tensor product} of two objects $(X,\alpha)$ and $(X', \alpha')$ in $\GSets$ is given by the object 
$$
(X,\alpha) \timesG (X',\alpha'): =\big(X \, \due \times {\scriptscriptstyle{\alpha}} {\, \scriptscriptstyle{\alpha'}} X', \alpha\alpha'\big), 
$$
where $\alpha\alpha': X \, \due \times {\scriptscriptstyle{\alpha}} {\, \scriptscriptstyle{\alpha'}} X' \to \Go$, $(x,x') \mapsto  \alpha(x)=\alpha'(x')$. The identity object is the left $\Gg$-set $(\Go,1_{\Sscript{\Go}})$ with the action $\Gg_{\Sscript{1}}\, \due \times {\scriptscriptstyle{\Sf{s}}} {\, \scriptscriptstyle{\alpha}} \Go \to \Go$, $(g,a) \mapsto g.\, a=t(g)$. The isomorphism of categories between left $\Gg$-sets and right $\Gg$-sets is obviously constructed by using the inverse map $\Ga \to \Ga$, $g \mapsto g^{-1}$.
Moreover, the forgetful functor $\Oo:  \GSets \to {\sf Sets}_{\Sscript{/\Go}}$, where the latter denotes the category of objects over $\Go$ (the comma category), admits a left adjoint functor $\Ga \, \due \times {\scriptscriptstyle{\Sf{s}}} {\, \scriptscriptstyle{\bullet}} -: {\sf Sets}_{\Sscript{/\Go}} \to \GSets$, which is defined on objects as follows. If $(M,\gamma)$ is an object in ${\sf Sets}_{\Sscript{/\Go}}$, then $(\Ga \, \due \times {\scriptscriptstyle{\Sf{s}}} {\, \scriptscriptstyle{\gamma}} M, {\sf{t}} \circ pr_{\Sscript{1}})$ is a left $\Gg$-set with action given by the multiplication on the first component. 


Consider a left $\Gg$-set $(X,\alpha)$ and let $x \in X$. Then clearly the pair $(\Orb{G}{x}, \alpha_{\Sscript{x}})$, where $\alpha_{\Sscript{x}}$ is the restriction of $\alpha$, inherits from $(X,\alpha)$ the structure of a left $\Gg$-set with $\Gg$-equivariant monomorphism  $\tau_{\Sscript{x}}: (\Orb{G}{x}, \alpha_{\Sscript{x}}) \hookrightarrow (X,\alpha)$, the canonical injection. It turns out that the disjoint union 
\begin{equation}\label{Eq:union}
(X,\alpha)\, =\, \biguplus_{x \, \in \, {\rm rep}(X/\Gg)} (\Orb{G}{x}, \alpha_{\Sscript{x}}),
\end{equation}
where ${\rm rep}(X/\Gg)$ is a set of representatives of the equivalence classes, coincides with the coproduct  of the discrete system $\{(\Orb{G}{x}, \alpha_{\Sscript{x}}), \tau_{\Sscript{x}})\}_{\Sscript{x \, \in \, {\rm rep}(X/\Gg)}}$ in the category of left $\Gg$-sets.

Let $\Gg$ and $\Hh$ be two groupoids and $(X,\alpha, \beta)$ a triple consisting of a set $X$ and two maps $\alpha : X \to \Gg_{\Sscript{0}}$ and $\beta: X \to \Hh_{\Sscript{0}}$.  The following definitions are abstract formulations of those given in \cite{MoeMrc:LGSAC} for topological  and Lie groupoids.

\begin{definition}\label{def:biset}
The triple $(X,\alpha, \beta)$ is said to be an \emph{$(\Gg,\Hh)$-biset} if there is a left $\Gg$-action $\lambda: \Gg_{\Sscript{1}}\, \due \times {\Sscript{\Sf{s}}} {\, \Sscript{\alpha}} \,  X \to X$ and right $\Hh$-action $\rho: X\, \due \times {\Sscript{\beta}} {\, \Sscript{\Sf{t}}} \,  \Hh_{\Sscript{1}} \to X$ such that
\begin{enumerate} 
\item 
For any $x \in X$, $h \in \Hh_{\Sscript{1}}$, and $g \in \Gg_{\Sscript{1}}$ with $\alpha(x)=\Sf{s}(g)$ as well as $\beta(x)=\Sf{t}(h)$, we have
$$ \beta(gx) =\beta(x)\; \text{ and }\; \alpha(xh)=\alpha(x).$$
\item 
For any $ x \in X$, $h \in \Hh_{\Sscript{1}}$, and $ g \in \Gg_{\Sscript{1}}$ with  $\alpha(x)=\Sf{s}(g)$ as well as $\beta(x)=\Sf{t}(h)$, we have 
$g(xh)\,=\, (gx)h$.
\end{enumerate}
\end{definition}
Given a $(\Gg,\Hh)$-biset $(X,\alpha,\beta)$, we denote by $(X^{\Sscript{op}},\beta,\alpha)$ the so-called \emph{opposite biset} of $(X,\alpha,\beta)$, that is, the $(\Hh,\Gg)$-biset whose underlying set is $X$ and whose actions are interchanged: 
$h x^{\Sscript{op}} = (x h^{-1})^{\Sscript{op}}$ and $x^{\Sscript{op}} g=(g^{-1} x)^{\Sscript{op}}$,  whenever the action between parentheses is permitted. 

\begin{rem}\label{rem:tensorG}
For a left resp.\ right $\Gg$-set $(X,\alpha)$ and $(Y, \vartheta)$ over the same groupoid $\Gg$, the fibred product $Y \, \due \times {\scriptscriptstyle{\vartheta}} {\, \scriptscriptstyle{\alpha}} X$ 
carries a left $\Gg$-action given by  $g  (x,y):= (xg^{-1}, gy)$, and one can consider its orbit space, {\em i.e.}, the orbit of the left translation groupoid $\Gg \lJoin \big(Y \, \due \times {\scriptscriptstyle{\vartheta}} {\, \scriptscriptstyle{\alpha}} X\big)$, denoted by $Y\tensor{\Gg}X$ in \cite[p.\ 166]{MoeMrc:LGSAC}. This product can be termed as the \emph{tensor product over the groupoid $\Gg$}. The universal property of this  tensor product is summarised in the following coequaliser: 
\begin{equation}\label{Eq:CoEq}
\xymatrix{ Y \, \due \times {\scriptscriptstyle{\vartheta}} {\, \scriptscriptstyle{\sf{t}}} \Ga \, \due \times {\scriptscriptstyle{\sf{s}}} {\, \scriptscriptstyle{\alpha}}  X   \ar@<0.70ex>@{->}^-{\rho \times 1_{\Sscript{X}}}[rr] \ar@<-0.70ex>@{->}_-{1_{\Sscript{Y}} \times \lambda }[rr]  & & Y \, \due \times {\scriptscriptstyle{\vartheta}} {\, \scriptscriptstyle{\alpha}} X      \ar@{->>}^-{}[rr]  &&  Y\tensor{\Gg}X.    }
\end{equation}
Obviously, there are natural isomorphisms $\Gg \tensor{\Gg} X \cong X$ and $Y\tensor{\Gg}\Gg \cong Y$ in the categories of left $\Gg$-sets and that of right $\Gg$-sets, respectively. 
Moreover, taking another two groupoids $\Hh$ and $\Kk$ and assuming $Y$ to be (the underlying set) of an $(\Hh,\Gg)$-biset along $\varsigma: Y \to \Ho$, and $X$ that of a $(\Gg,\Kk)$-biset along $\beta:X \to \Ko$. Then  $Y\tensor{\Gg} X$ inherits, in a canonical way, the structure of an $(\Hh,\Kk)$-biset along the maps $\bara{\varsigma}: Y\tensor{\Gg} X \to \Ho$, $y\tensor{\Gg}x \mapsto \varsigma(y)$ and $\bara{\beta}: Y\tensor{\Gg} X \to \Ko$, $y\tensor{\Gg}x \mapsto \beta(x)$. 
\end{rem}

The \emph{two-sided translation groupoid} associated to a given $(\Gg, \Hh)$-biset $(X,\alpha, \beta)$ is defined to be the groupoid $\Gg \lJoin X \rJoin \Hh$ whose set of objects is $X$ and whose set of arrows is given by 
$$
\Gg_{\Sscript{1}}\, \due \times {\Sscript{\Sf{s}}} {\, \Sscript{\alpha}} \, X \, \due \times {\Sscript{\beta}}{\, \Sscript{\Sf{s}}} \, \Hh_{\Sscript{1}}\,=\, \big\{ (g,x,h) \, \in \,  \Gg_{\Sscript{1}}\times X \times \Hh_{\Sscript{1}} \mid \,\, \Sf{s}(h)= \beta(x),\, \Sf{s}(g)=\alpha(x) \big\}.
$$
Its structure maps are as follows. Source and target read as  
$$
\Sf{s}(g,x,h)=x,\quad \Sf{t}(g,x,h)=gxh^{-1}\;\;  \text{ and }\; \iota_{\Sscript{x}}=(\iota_{\Sscript{\alpha(x)}}, x,  \iota_{\Sscript{\beta(x)}}), 
$$
whereas multiplication and inverse are given by 
$$
(g,x,h) (g',x',h')\,=\,(gg',x',hh'),\quad  (g,x,h)^{-1}=(g^{-1}, gxh^{-1}, h^{-1}).
$$

Associated to a given $(\Gg, \Hh)$-biset $(X,\alpha,\beta)$, there are two canonical morphisms of groupoids: 
\begin{eqnarray}
\Sigma: \Gg \lJoin X \rJoin \Hh \longrightarrow \Hh, & & \big((g,x,h),  y \big) \longmapsto \big( h,\beta(y) \big),         \label{Eq:t}  \\
\Theta: \Gg \lJoin X \rJoin \Hh \longrightarrow \Gg, & &   \big( (g,x,h),  y \big) \longmapsto \big( g,\alpha(y) \big). \label{Eq:s}
\end{eqnarray}

The following concept (and its analogue notion of principal bibundles for flat Hopf algebroids in Definition \ref{def:PB}) will be the crucial ingredient when it comes to defining equivalences:

\begin{definition}\label{def:pbset}
Let  $(X,\alpha,\beta)$ be a $(\Gg,\Hh)$-biset. We say that $(X,\alpha,\beta)$ is a \emph{left principal $(\Gg,\Hh)$-biset} (or \emph{left principal $(\Gg,\Hh)$-bundle}) if it satisfies the following conditions:
\begin{enumerate}[(P-1)]
\item $\beta:  X \to \Hh_{\Sscript{0}}$ is surjective;
\item the canonical map 
\begin{equation}\label{Eq:can}
\nabla^{\Sscript{l}}: \Gg_{\Sscript{1}}\, \due \times {\Sscript{\Sf{s}}} {\, \Sscript{\alpha}} \, X \longrightarrow X\,  \due \times {\Sscript{\beta}} {\, \Sscript{\beta}} \, X  , \quad (g,x) \longmapsto (gx,x) 
\end{equation}
is bijective. 
\end{enumerate}
\end{definition}
Condition (P-2) allows us to define $\delta^{\Sscript{l}}: = pr_{\Sscript{1}} \circ (\nabla^{\Sscript{l}})^{-1}:  X\,  \due \times {\Sscript{\beta}} {\, \Sscript{\beta}} \, X \to \Gg_{\Sscript{1}}$. This map clearly satisfies:  
\begin{eqnarray}
 \Sf{s}\big(\delta^{\Sscript{l}}(x,x')\big) &=& \alpha(x') \label{Eq:d1} \\
\delta^{\Sscript{l}}(x,x')x'&=& x,\quad \text{ for any} \,  x,x' \in X\, \text{ with }\, \beta(x)=\beta(x'); \label{Eq:d2} \\ 
\delta^{\Sscript{l}}(gx,x) &=& g, \quad \text{for } g \in \Gg_{\Sscript{1}},\, x \in X\,\, \text{ with } \Sf{s}(g)=\alpha(x). \label{Eq:d3}
\end{eqnarray}
Equation \eqref{Eq:d3} shows that the action is in fact free, that is, $g x=x$ only when $g=\iota_{\Sscript{\alpha(x)}}$. Left principal bisets can now be characterised as follows: a $(\Gg,\Hh)$-biset is left principal if and only if $\Ho$ is, up to a bijection, the left orbit set $X/\Gg$ and the left action is free. 

Right principal bisets are defined in an obvious manner and the corresponding map from above will be denoted by $\delta^{\Sscript{r}}$. The following result will turn out to be useful in the sequel.

\begin{lemma}
\label{lema:NatIso}
Let $(Y,\varsigma, \vartheta)$ be a right principal $(\Hh,\Gg)$-biset and let $(X,\alpha)$ be any left $\Gg$-set. Then there is a natural isomorphism 
$$
Y\, \due \times {\Sscript{\varsigma}} {\, \Sscript{\bara{\varsigma}}} \, \big(  Y \tensor{\Gg}X \big) \longrightarrow  Y\, \due \times {\Sscript{\vartheta}} {\, \Sscript{\alpha}} \, X, \quad (y,y'\tensor{\Gg}x) \longmapsto \big(y,\delta^{\Sscript{r}}(y,y')x\big) 
$$
whose inverse is 
$$
Y\, \due \times {\Sscript{\vartheta}} {\, \Sscript{\alpha}} \, X \longrightarrow
Y\, \due \times {\Sscript{\varsigma}} {\, \Sscript{\bara{\varsigma}}} \, \big(  Y \tensor{\Gg}X \big) , \quad (y,x) \longmapsto (y,y\tensor{\Gg}x).
$$
\end{lemma}
\begin{proof}
Straightforward. 
\end{proof}

A $(\Gg,\Hh)$-biset $(X,\alpha,\beta)$ is said to be a \emph{principal biset} (or \emph{principal $(\Gg,\Hh)$-bibundle}) if it is simultaneously a left and a right principal biset. Thus both $\alpha$ and $\beta$ are surjective and the canonical maps 
\begin{equation}
\label{Eq:nablas}
\nabla^{\Sscript{l}}: \Gg_{\Sscript{1}}\, \due \times {\Sscript{\Sf{s}}} {\, \Sscript{\alpha}} \, X \longrightarrow X\,  \due \times {\Sscript{\beta}} {\, \Sscript{\beta}} \, X, \, (g,x) \longmapsto (gx,x),  \quad
\nabla^{\Sscript{r}}: X\, \due \times {\Sscript{\beta}} {\, \Sscript{\Sf{t}}} \, \Hh_{\Sscript{1}} \longrightarrow X\,  \due \times {\Sscript{\alpha}} {\, \Sscript{\alpha}} \, X,\, (x,h) \longmapsto (x,xh)
\end{equation}
are both bijective. It is clear that $(\Gg_{\Sscript{1}},\Sf{t}, \Sf{s})$ with the canonical action is a principal $(\Gg, \Gg)$-set, and that the pull-back of any principal groupoid-set is also a principal groupoid-set. 

\subsection{Natural isomorphisms and functors between groupoid-sets}
\label{ssec:GS}
Let $(X,\alpha, \beta)$ be a triple consisting of a left $\Gg$-set $(X,\alpha)$ and a map $\beta: X \to \Ko$ such that $\beta(gx)=\beta(x)$, for every $(g,x) \in \Ga \, \, \due \times {\Sscript{\Sf{s}}} {\, \Sscript{\alpha}} \, X$. Triples like that form a category (of \emph{left $\Gg$-sets over $\Ko$}), which we denote by $\GSets_{/\Sscript{\Ko}}$. Clearly, when $\Ko$ is the object set of a groupoid $\Kk$, then the category of $(\Gg,\Kk)$-bisets is a full subcategory of $\GSets_{/\Sscript{\Ko}}$. In particular, if $\Kk=(\Ko,\Ko)$ is a trivial groupoid, then both categories coincide. 

For a functor $\Phi: \GSets \to \HSets$ (which we always assume to transform the empty set to the empty set and which most of the times we just denote by $\Phi(X)$ for the image of a left $\Gg$-set $(X, \alpha)$), we want to next discuss conditions under which $\Phi$ descends to a functor from $\GSets_{/\Sscript{\Ko}}$ to $\HSets_{/\Sscript{\Ko}}$. 

\begin{lemma}\label{lema:Phi}
Let $\Phi$ and $(X,\alpha, \beta)$ be as above. 
\begin{enumerate}
\item 
Assume that $\Phi$ preserves monomorphisms and coproducts. Then there is a functor $\Phi'$ which makes the following diagram commutative:
$$
\xymatrix@R=15pt{ \GSets  \ar@{->}^-{\Phi}[rr] & & \HSets  \\  \GSets_{/\Sscript{\Ko}}  \ar@{-->}^-{\Phi'}[rr]  \ar@{->}^-{}[u]& & \ar@{->}^-{}[u] \HSets_{/\Sscript{\Ko}}, }
$$
where the vertical functors are the forgetful ones. 
\item 
Assume that $\Phi(\Go)=\Ho$. Then, for any left $\Gg$-set $(X,\alpha)$, the structure map of the left $\Hh$-set $\Phi(X) = \Phi(X, \alpha)$ is given by $\Phi(\alpha)$.
\end{enumerate}
\end{lemma}
\begin{proof}
Part $(i)$: for an object $(X,\alpha,\beta) \in \GSets_{/\Sscript{\Ko}}$, using the decomposition (or stratification) of equation \eqref{Eq:union}, we obtain a map:
\begin{equation}\label{Eq:betaphi}
\xymatrix{ \beta^{\Sscript{\Phi}}: \Phi(X) \,=\, \underset{x \, \in \, rep(X/\Gg)}{\biguplus} 
\Phi\big(\Orb{G}{x}\big) 
\ar@{->}^-{}[r] &   X \ar@{->}^-{\beta}[r]   & \Ko. }
\end{equation}
The triple $(X^{\Sscript{\Phi}}, \alpha^{\Sscript{\Phi}}, \beta^{\Sscript{\Phi}})$, where $\Phi(X,\alpha):=(X^{\Sscript{\Phi}}, \alpha^{\Sscript{\Phi}})$,  is easily shown to be an object in the category $\HSets_{ \Sscript{/\Ko}}$ since $\Phi$ preserves  monomorphisms. This gives the construction of $\Phi'$ on the objects class; the compatibility of $\Phi'$ with the arrows of $\GSets_{ \Sscript{/\Ko}}$ is immediate.  

Part $(ii)$: we set as before $\Phi(X,\alpha)=(X^{\Sscript{\Phi}}, \alpha^{\Sscript{\Phi}})$, the associated left $\Hh$-set. Since the map $\alpha: (X,\alpha) \to (\Go,1_{\Go})$ is a left $\Gg$-equivariant, its image $\Phi(\alpha)$ gives the structure map of the left $\Hh$-set $(X^{\Sscript{\Phi}}, \alpha^{\Sscript{\Phi}})$, that is, we have $\alpha^{\Sscript{\Phi}} = \Phi(\alpha)$. 
\end{proof}

Consider now an object $(X,\alpha, \beta) $ in $\GSets_{ \Sscript{/\Ko}}$ and a functor as in Lemma \ref{lema:Phi}. We then get two functors: the first one is $\Phi \circ (X\, \due \times {\Sscript{\beta}} {\, \Sscript{\bullet}} \, -): {\sf Sets}_{\Sscript{/\Ko}} \to \HSets_{ \Sscript{/\Ko}}$ and the second $\Phi(X,\alpha) \, \due \times {\Sscript{\beta^{\Phi}}} {\, \Sscript{\bullet}} \, -: {\sf Sets}_{\Sscript{/\Ko}} \to \HSets_{ \Sscript{/\Ko}}$. The subsequent technical lemma shows a natural isomorphism between these two functors. 

\begin{lemma}\label{lema:Upsilon}
Let $\Phi: \GSets \to \HSets$ be as in Lemma \ref{lema:Phi}. Then, for any object $(X,\alpha,\beta) $ in the category $\GSets_{\Sscript{/\Ko}}$,  there is a natural isomorphism 
$$
\Upsilon: \Phi\big(X\, \due \times {\Sscript{\beta}} {\, \Sscript{\gamma}} \, M, \alpha \circ pr_{\Sscript{1}}\big)  \, \cong \, \big(\Phi(X) \, \due \times {\Sscript{\beta^{\Phi}}} {\, \Sscript{\gamma}} \, M, \alpha^{\Sscript{\Phi}} \circ pr_{\Sscript{1}}\big)
$$
for every $(M,\gamma)$ in ${\sf Sets}_{\Sscript{/\Ko}}$. Furthermore if there is a morphism $f: (X,\alpha,\beta) \to (X',\alpha',\beta')$ in the category $\GSets_{\Sscript{/\Ko}}$, then there is a commutative diagram:
$$
\xymatrix{  \Phi\big(X\, \due \times {\Sscript{\beta}} {\, \Sscript{\gamma}} \, M, \alpha \circ pr_{\Sscript{1}}\big)   \ar@{->}^-{\Upsilon}[rr]  \ar@{->}_-{\Phi(f \times 1_M)}[d] &  &  \big(\Phi(X) \, \due \times {\Sscript{\beta^{\Phi}}} {\, \Sscript{\gamma}} \, M, \alpha^{\Sscript{\Phi}} \circ pr_{\Sscript{1}}\big) \ar@{->}^-{\Phi(f) \times 1_M}[d]
 \\ \Phi\big(X'\, \due \times {\Sscript{\beta'}} {\, \Sscript{\gamma}} \, M, \alpha '\circ pr_{\Sscript{1}}\big) \ar@{->}^-{\Upsilon'}[rr] & & \big( \Phi(X') \, \due \times {\Sscript{\beta'{}^{\Phi}}} {\, \Sscript{\gamma}} \, M, \alpha'{}^{\Sscript{\Phi}} \circ pr_{\Sscript{1}}\big).
 }
$$
\end{lemma}

An important consequence of the previous lemma is: 

\begin{proposition}\label{prop:Upsilon}
Let $\Phi: \GSets \to \HSets$ be an equivalence of categories. Then we have
\begin{enumerate}
\item 
For any $(\Gg,\Kk)$-biset $(X,\alpha,\beta)$ the triple $(X^{\Sscript{\Phi}}, \alpha^{\Sscript{\Phi}}, \beta^{\Sscript{\Phi}})$ is an $(\Hh,\Kk)$-biset, where $X^{\Sscript{\Phi}}$ denotes the underlying set of $\Phi(X)$. 
\item 
There is a natural isomorphism $\Phi \,\cong\, \Phi(\Ga) \tensor{\Gg}-: \GSets \to \HSets$. 
\end{enumerate} 
\end{proposition}

\begin{proof}
Part $(i)$: let  $(X,\alpha,\beta)$ be a $(\Gg,\Kk)$-biset. 
Using Lemma \ref{lema:Upsilon}, we have a commutative diagram
$$
\xymatrix@R=15pt{ \big( \Phi(X) \, \due \times {\Sscript{\beta^{\Phi}}} {\, \Sscript{\sf{t}}} \, \Ka, \alpha^{\Sscript{\Phi}} \circ pr_{\Sscript{1}}\big) \ar@{->}^-{\Upsilon^{-1}}[rd]   \ar@{-->}^-{}[rr] &  & \Phi(X)  \\  & \Phi\big(  X\, \due \times {\Sscript{\beta}} {\, \Sscript{\sf{t}}} \, \Ka, \alpha \circ \pr_{\Sscript{1}} \big)  \ar@{->}^-{\Phi(\varrho)}[ru] &   }
$$
The horizontal map leads to a well-defined right $\Kk$-action on the set $(X^{\Sscript{\Phi}}, \beta^{\Sscript{\Phi}})$. Moreover, since each stratum in the stratification \eqref{Eq:union} of the left $\Gg$-set $(X,\alpha)$ is invariant under the right $\Kk$-action, the triple $(X^{\Sscript{\Phi}}, \alpha^{\Sscript{\Phi}}, \beta^{\Sscript{\Phi}})$ fulfils the conditions of Definition \ref{def:biset} for the groupoids $\Hh$ and $\Kk$. Thus, $(X^{\Sscript{\Phi}}, \alpha^{\Sscript{\Phi}}, \beta^{\Sscript{\Phi}})$ is actually an $(\Hh,\Kk)$-biset. 

Part $(ii)$: by the previous part, the image of $(\Ga,\Sf{t})$ under $\Phi$ is an $(\Hh, \Gg)$-biset since $(\Ga,\Sf{t},\Sf{s})$ is a $(\Gg,\Gg)$-biset. Now, using Remark \ref{rem:tensorG}, we know that the functor $\Phi(\Ga) \tensor{\Gg}-: \GSets \to \HSets$ is well-defined. The claimed natural isomorphism is then derived from the commutative diagram 
$$
\xymatrix{    
\Phi(\Ga) \, \due \times {\scriptscriptstyle{\sf{s}^{\Phi}}} {\, \scriptscriptstyle{\sf{t}}} \Ga \, \due \times {\scriptscriptstyle{\sf{s}}} {\, \scriptscriptstyle{\alpha}}  X   \ar@<0.70ex>@{->}^-{}[rr] \ar@<-0.70ex>@{->}_-{}[rr] \ar@{->}_-{\Upsilon^{-1} }[d]  & & \Phi(\Ga) \, \due \times {\scriptscriptstyle{\sf{s}^{\Phi}}} {\, \scriptscriptstyle{\alpha}} X      \ar@{->>}^-{}[rr]  \ar@{->}_-{\Upsilon^{-1} }[d] &&   \Phi(\Ga) \tensor{\Gg}X \ar@{->}_-{\cong}[d]  \\ 
\Phi\big( \Ga \, \due \times {\scriptscriptstyle{\sf{s}}} {\, \scriptscriptstyle{\sf{t}}} \Ga  \, \due \times {\scriptscriptstyle{\sf{s}}} {\, \scriptscriptstyle{\alpha}}  X \big)   \ar@<0.70ex>@{->}^-{}[rr] \ar@<-0.70ex>@{->}_-{}[rr]  & & \Phi\big(  \Ga\, \due \times {\scriptscriptstyle{\sf{s}}} {\, \scriptscriptstyle{\alpha}} X \big)      \ar@{->>}^-{}[rr]  &&  \Phi\big(  \Gg\tensor{\Gg}X \big) \,\cong \, \Phi(X)
}
$$
as $\Phi$ preserves coequalisers. 
\end{proof}

\subsection{Monoidal equivalence between groupoid-sets versus principal bisets}
\label{ssec:MPB}
Let $\phi: \Hh \to \Gg$ be a morphism of groupoids. Then the induced morphism $\phi^*:   \GSets \to  \HSets $ which sends any left $\Gg$-set $(X,\alpha)$ to the left $\Hh$-set $$
\phi^*(X,\alpha):=( \Hh_{\Sscript{0}}\, \due \times {\Sscript{\phi_{\Sscript{0}}}} {\, \Sscript{\alpha}} \, X ,  \alpha \circ \pr_{\Sscript{2}}= \phi_{\Sscript{0}} \circ pr_{\Sscript{1}})
$$ 
with action $hx=\phi_{\Sscript{1}}(h)x$, is clearly a symmetric monoidal functor. The morphism $\phi$ is said to be a \emph{weak equivalence} if the functor between the underlying categories induces an equivalence of categories, {\em i.e.}, if $\phi$ is a full, faithful, and essentially surjective functor.  
In this way, it is clear that any weak equivalence induces an equivalence of categories between the categories of left groupoid-sets. \

Next, we want to discuss the converse, meaning that any monoidal symmetric equivalence between $\GSets$ and $\HSets$ can be reconstructed (although in a noncanonical way) from some weak equivalence.  

Recall that two groupoids $\Gg$ and $\Hh$ are said to be \emph{weakly equivalent} if there is a third groupoid $\Kk$ and a diagram 
$$
\xymatrix@R=7pt{ & \ar@{->}_-{\Sscript{}}[ld] \Kk \ar@{->}^-{\Sscript{}}[rd] &  \\  \Hh  & &   \Gg}
$$ 
of weak equivalences. One can choose an inverse of one of the morphisms in this diagram in order to construct a weak equivalence connecting $\Hh$ and $\Gg$. This is almost impossible in the case of topological and/or Lie groupoids and also for flat Hopf algebroids as we will see in the forthcoming sections. However, we have the following lemma analogous to the case of Lie groupoids \cite{MoeMrc:LGSAC}, and we will later show in \S\ref{ssec:lpb-we} its analogue for flat Hopf algebroids.

\begin{lemma}\cite[Proposition 2.13]{Kaoutit:2015}
\label{lema:WE}
Let $\Gg$ and $\Hh$ be two groupoids and let $(X,\alpha, \beta)$ be a principal $(\Gg, \Hh)$-biset. 
Then the canonical morphisms of groupoids 
$$
\xymatrix@R=7pt{ & \ar@{->}_-{\Sscript{\Theta}}[ld] \Gg \lJoin X \rJoin \Hh \ar@{->}^-{\Sscript{\Sigma}}[rd] &  \\  \Gg  & &   \Hh}
$$
are weak equivalences, where $\Theta$ and $\Sigma$ are as in \eqref{Eq:t} resp.\ \eqref{Eq:s}. In particular, $\Gg$ and $\Hh$ are weakly equivalent. 
\end{lemma}

The main motivation behind Theorem \ref{processocivile} in the Introduction is the following characterisation of weak equivalences between groupoids and principal bisets (see \cite[Corollary 3.11]{MoeMrc:LGSAC} for the implication $(iii) \Rightarrow (ii)$, where groupoid-sets are replaced by sheaves of \'etale spaces).

\begin{theorem}
\label{thm:AG}
Let $ \Gg$ and $\Hh$ be  two groupoids. Then the following are equivalent:
\begin{enumerate}
\item $\Gg$ and $\Hh$ are weakly equivalent.
\item  There is a symmetric monoidal equivalence of the categories $\GSets$ and $\HSets$.
\item There is a principal $(\Hh,\Gg)$-biset. 
\end{enumerate}
\end{theorem}
\begin{proof}
The proof of $(i) \Rightarrow (ii)$ is immediate. The implication $(iii) \Rightarrow (i)$ follows from Lemma \ref{lema:WE}. 

As for the implication $(ii) \Rightarrow (iii)$, let $\Phi: \GSets \to \HSets$ be such an equivalence of categories and denote by $\Psi$ its inverse functor. We set $(P,\varsigma,\vartheta)$ as the image of the principal $(\Gg,\Gg)$-biset $(\Ga, \sf{t}, \sf{s})$ by the functor $\Phi$ from which we know by Proposition \ref{prop:Upsilon}(i) that it is an $(\Hh,\Gg)$-biset. 
Now using the monoidal properties of $\Phi$, we have from one hand that $\varsigma=\Phi(t)$ by Lemma \ref{lema:Phi}(i), which is a surjective map, and from the other hand we have a chain of isomorphisms
$$
P\,\due \times {\Sscript{\vartheta}} {\, \Sscript{\sf{t}}} \,  \Ga \,\cong \, \Phi\big(  \Ga\,\due \times {\Sscript{\sf{s}}} {\, \Sscript{\sf{t}}} \,  \Ga \big) \longrightarrow \Phi\big(  \Ga\,\due \times {\Sscript{\sf{t}}} {\, \Sscript{\sf{t}}} \,  \Ga \big) \,\cong \, P  \, \due \times {\Sscript{\varsigma}} {\, \Sscript{\varsigma}} \, P,
$$
which turns out to be the canonical map $\nabla^{\Sscript{r}}$ for $P$. Therefore, $(P,\varsigma,\vartheta)$ is a right principal $(\Hh,\Gg)$-biset. 

Similarly, if we denote by $(Q,\mu,\nu)$ the image of the principal $(\Hh,\Hh)$-biset $(\Ha, \sf{t},\sf{s})$ under the functor $\Psi$, we get a right principal $(\Gg,\Hh)$-biset. 
 To conclude, one needs to check that there is an isomorphism $(P^{\Sscript{op}},\vartheta,\varsigma)  \to (Q,\mu,\nu)$ of $(\Gg,\Hh)$-bisets, where $(P^{\Sscript{op}},\vartheta,\varsigma)$ is the biset opposite to $(P,\varsigma, \vartheta)$.

To this end, we first apply Lemma \ref{lema:NatIso} to $(P,\varsigma,\vartheta)$ and $(Q,\mu)$ in order to obtain the isomorphism 
$$
\gamma: P\, \due \times {\Sscript{\varsigma}} {\, \Sscript{\bara{\varsigma}}} \, \big(  P \tensor{\Gg}Q \big) \longrightarrow  P\, \due \times {\Sscript{\vartheta}} {\, \Sscript{\mu}} \, Q, \quad (p,p'\tensor{\Gg}q) \longmapsto \big(p,\delta^{\Sscript{r}}(p,p')q\big).
$$ 
Second, we use the isomorphism  $\chi: \Ha \to P \tensor{\Gg}Q$ of $(\Hh,\Hh)$-bisets given by the natural isomorphism of Proposition  \ref{prop:Upsilon}(ii) applied to $\Phi$, in order to construct the desired isomorphism 
$$
P^{\Sscript{op}} \longrightarrow Q, \quad p \longmapsto pr_{\Sscript{2}}\big( \gamma(p,\chi(\iota_{\Sscript{\varsigma(p)}}) \big)
$$
of $(\Gg,\Hh)$-bisets.
\end{proof}

\section{Hopf algebroids and comodule algebras}\label{sec:HalgdComodalg}
All algebras are considered to be commutative $\Bbbk$-algebras, where $\Bbbk$ is a commutative ground ring.  The $\Bbbk$-module of all algebra maps from $R$ to $C$ will be denoted by $R(C):={\rm Alg}_{\Bbbk}\big(R,C\big)$.

\subsection{Hopf algebroids}
\label{ssec:halgd}
Recall from, {\em e.g.}, \cite{Rav:CCASHGOS} 
that a {\em commutative} Hopf algebroid is a pair $(A,\cH)$ of two commutative $\Bbbk$-algebras together with a diagram 
$
\xymatrix@C=25pt{A\ar@<1ex>@{->}|-{\scriptscriptstyle{\sf{s}}}[r] \ar@<-1ex>@{->}|-{\scriptscriptstyle{\sf{t}}}[r] & \ar@{->}|-{ \scriptscriptstyle{\varepsilon}}[l]\cH}
$ 
of algebra maps, 
a structure  $({}_{\scriptscriptstyle{\Sf{s}}}\cH_{\scriptscriptstyle{\Sf{t}}}, \Delta, \varepsilon)$ of an $A$-coring with underlying $A$-bimodule  ${}_{\scriptscriptstyle{A}}\cH_{\scriptscriptstyle{A}}=
{}_{\scriptscriptstyle{\Sf{s}}}\cH_{\scriptscriptstyle{\Sf{t}}}$, along with an isomorphism 
$\mathscr{S}: {}_{\scriptscriptstyle{\Sf{s}}}H_{\scriptscriptstyle{\Sf{t}}} \to {}_{\scriptscriptstyle{\Sf{t}}}H_{\scriptscriptstyle{\Sf{s}}}$ of $A$-corings that fulfils $\mathscr{S}^2 = \id$, 
where the codomain is the opposite $A$-coring of   ${}_{\scriptscriptstyle{\Sf{s}}}H_{\scriptscriptstyle{\Sf{t}}}$. The map $\mathscr{S}$ is  called the \emph{antipode} of $\cH$. All the previous maps are asked to be compatible in the following way:
\begin{eqnarray}
\varepsilon \circ \Sf{s} \,=\, \id_\ahha,  & &   \varepsilon \circ \Sf{t}\,=\, \id_\ahha,  \\
\Delta(1_\cakka)\,=\, 1_{\cakka}\tensor{A} 1_{\cakka}, & &  \varepsilon(1_{\cakka})\,=\, 1_\ahha, 
\label{Eq:es}\\ 
\Delta(uv)\,=\, u_{(1)}v_{(1)}\tensor{A} u_{(2)}v_{(2)}, & & \varepsilon(uv)\,=\, \varepsilon(u)\varepsilon(v),\label{Eq:demul}\\ 
\label{maxxi}
\Sf{t}(\varepsilon(u))\,=\, \mathscr{S}(u_{(1)})u_{(2)}, & & \Sf{s}(\varepsilon(u))\,=\, u_{(1)}\mathscr{S}(u_{(2)}),\label{Eq:penultima}\\
\mathscr{S}(uv)\,=\, \mathscr{S}(u) \mathscr{S}(v),  & & \mathscr{S}(1_{\cakka})\,=\, 1_{\cakka},\label{Eq:ultima}
\end{eqnarray}
for every $a \in A$, $u, v \in \cH$, where we used Sweedler's notation for the comultiplication. 

As all Hopf algebroids in this article are commutative and flat over the base ring, they are also faithfully flat since both the source and target are (left) split morphisms of modules over the base ring. 

\smallskip

A \emph{morphism} 
$\B{\phi}: (A,\cH) \to (B,\cK)$ {\em of Hopf algebroids} 
consists of a pair $\B{\phi}=(\phi_{\scriptscriptstyle{0}},\phi_{\scriptscriptstyle{1}})$ 
of algebra maps  
$\phi_{\scriptscriptstyle{0}}: A \to B$ and $\phi_{\scriptscriptstyle{1}}: \cH \to \cK$ that are
compatible 
with the structure maps of both $\cH$ and $\cK$ in a canonical way. 
That is, the equalities 
\begin{eqnarray}
\phi_{\scriptscriptstyle{1}} \circ \Sf{s} \,\,=\,\,  \Sf{s} \circ \phi_{\scriptscriptstyle{0}}, & &
\phi_{\scriptscriptstyle{1}} \circ \Sf{t} \,\,=\,\,  \Sf{t} \circ \phi_{\scriptscriptstyle{0}},   \\
\Delta \circ \phi_{\scriptscriptstyle{1}} \,\,=\,\,  \chi \circ (\phi_{\scriptscriptstyle{1}} \tensor{A}\phi_{\scriptscriptstyle{1}})  \circ \Delta, && \varepsilon \circ \phi_{\scriptscriptstyle{1}}\,\,=\,\, \phi_{\scriptscriptstyle{0}} \circ \varepsilon, \\ 
\mathscr{S} \circ \phi_{\scriptscriptstyle{1}} \,\,=\,\, \phi_{\scriptscriptstyle{1}} \circ \mathscr{S}, 
\end{eqnarray}
hold, where $\chi$ is the obvious map $\chi: \cK\tensor{A}\cK \to \cK\tensor{B}\cK$,   and  where no distinction between the structure maps of $\cH$ and $\cK$ was made. 

\begin{example}[{\em Scalar extension Hopf algebroid}]
For a Hopf algebroid $(A,\cH)$ and an algebra map $\phi_{\scriptscriptstyle{0}}: A \to B$, we can consider the so-called \emph{scalar extension}  Hopf algebroid $(B, B\tensor{A}\cH\tensor{A}B)$  in a canonical way such that 
$(\phi_{\scriptscriptstyle{0}}, \phi_{\scriptscriptstyle{1}}): (A,\cH) \to (B, B\tensor{A}\cH\tensor{A}B)$, where $\phi_{\scriptscriptstyle{1}}(u)= 1_\behhe \tensor{A} u \tensor{A} 1_\behhe$, becomes a morphism of Hopf algebroids. In this way,  any morphism 
$\B{\phi}: (A,\cH) \to (B,\cK)$ of Hopf algebroids factors through the following morphism 
\begin{equation}\label{Eq:PHI}
{\Phi}: (B, B\tensor{A}\cH\tensor{A}B) \to (B,\cK), \quad b\tensor{A}u\tensor{A}b' \mapsto \Sf{s}(b)\phi_{\scriptscriptstyle{1}}(u)\Sf{t}(b')
\end{equation}
of Hopf algebroids.
\end{example}
\begin{rem}
Notice that the scalar extension Hopf algebroid $(B, B\tensor{A}\cH\tensor{A}B)$ is not necessarily flat. This happens, for instance, if $\phi_{\scriptscriptstyle{0}}$ is a flat extension or if $B$ is {\em Landweber exact} over $(A,\cH)$ in the sense of \cite[Def.\ 2.1, Corollary 2.3]{HovStr:CALEHT}, which means that either the extension $A \to \cH\tensor{A}B$, $a \mapsto \Sf{s}(a)\tensor{A}1_{\scriptscriptstyle{B}}$ or $A \to B\tensor{A} \cH$, $a \mapsto 1_{\scriptscriptstyle{B}} \tensor{A} \Sf{t}(a)$ is flat, see also Remark \ref{rem:HS}. Another important situation is when $\cH$ is assumed to be flat as an $A\tensor{}A$-module ({\em i.e.}, the extension $\Sf{s}\tensor{}\Sf{t}$ is flat). This happens, for instance, when $\cH$ is geometrically transitive Hopf algebroid in the sense of Deligne and Brugui\`eres \cite{Deligne:1990, Bruguieres:1994}, see also \cite{Kaoutit:2015}.
\end{rem}
\subsection{Comodules, bicomodules and cotensor product}
\label{ssec:bicomod}
This section gathers some standard material on comodules over commutative Hopf algebroids which will be needed in the sequel, see, {\em e.g.}, again \cite{Rav:CCASHGOS} for more information.

A right $\cH$-comodule over a Hopf algebroid $(A, \cH)$ is a pair $(M,\rcoaction{\cH}{M})$, where $M$ is an $A$-module and $\rcoaction{\cH}{M}: M \to M\tensor{A} \shopf{H}$, $m \mapsto m_{(0)} \otimes_\ahha m_{(1)}$ is an $A$-linear map, written in the usual Sweedler notation,
and which satisfies the usual coassociativity and counitary properties. Here, the $A$-module structure on $M\tensor{A} \shopf{H}$ with respect to which the coaction is $A$-linear is defined by $(m\tensor{A}u) \bract a :=  m\tensor{A} u\Sf{t}(a)$. 
When the context is clear, we shall also drop sub- and superscripts on $\rcoaction{\cH}{M}$ that are sometimes needed to distinguish various coactions.

{\em  Morphisms of right $\cH$-comodules} are defined in an obvious way, and the category of right $\cH$-comodules will be denoted by $\rcomod{\cH}$, whereas a morphism between two right $\cH$-comodules $M$ and $N$ will be denoted as  $\rcomod{\cH}(M,N)$.
The category $\rcomod{\cH}$ is symmetric monoidal, where the coaction on the tensor product is given by the codiagonal coaction, that is, 
\begin{equation}
\label{Eq:tensor}
\rcoaction{\cH}{M \tensor{A}N}: M\tensor{A}N \to (M\tensor{A}N)\tensor{A}\shopf{H},\quad m\tensor{A}n \mapsto (m_{(0)}\tensor{A}n_{(0)})\tensor{A} m_{(1)}n_{(1)}.
\end{equation}
The identity object is given by $(A,\Sf{t})$ and the symmetry is given by the natural transformation obtained from the tensor flip. 

\begin{rem}
\label{mondragone}
There are situations where the  tensor product $M\tensor{A}N$ of the underlying modules of two right $\cH$-comodules can be endowed with more than one comodule structure. For distinction, we will from now on denote by 
$M\otimes^\ahha N$ the tensor product in $\rcomod{\cH}$ endowed then with the coaction of equation \eqref{Eq:tensor}.
\end{rem}

To each right $\cH$-comodule $(M,\rcoaction{}{})$ one can define the $\Bbbk$-vector space of \emph{coinvariants}:
$$
M^{\coinv_{\cH}}=\big\{ m \in M \mid \rcoaction{}{}(m)=m\tensor{A}1_{\scriptscriptstyle{\cH}} \big\}.
$$ 
This, in fact, establishes a functor which is naturally isomorphic to the functor $\rcomod{\cH}\big(A,- \big)$, that is, we have a natural isomorphism of $\Bbbk$-vector spaces:
$$
\rcomod{\cH}\big(A,M \big) \,\cong\,  M^{\coinv_{\cH}}.
$$

 Analogously, one can define the category $\lcomod{\cH}$ of left comodules, and both categories are isomorphic via the antipode. Explicitly, one can endow a left $\cH$-comodule $(M,\lcoaction{\cH}{M})$ with a right $\cH$-comodule structure, denoted by $M^{\op}$, 
\begin{equation}
\label{Eq:leftright}
\rcoaction{\cH}{M^{\op}}: M^{\op} \to M^{\op}\tensor{A}\shopf{H},\quad m \mapsto m_{(0)}\tensor{A}\mathscr{S}(m_{(-1)}),
\end{equation}
and referred to as the \emph{opposite comodule} of $M$.
Since we always have $\mathscr{S}^2=id$ for commutative Hopf algebroids, this correspondence obviously establishes  an isomorphism of symmetric monoidal categories. 

For an arbitrary algebra $R$ and a right comodule $(N,\rcoaction{}{})$ whose underlying module is also an $(A,R)$-bimodule such that $\rcoaction{}{}$ is left $R$-linear, {\em i.e.}, $\rcoaction{\cH}{M}(rn)=rn_ {(0)} \tensor{A} n_{(1)}$, for $r \in R$, $n \in N$, one can define a functor 
\begin{equation}
\label{Eq:RM}
- \tensor{R}N : \rmod{R} \to  \rcomod{\cH}, \quad X \mapsto (X\tensor{R}N, X\tensor{R}\rcoaction{}{}). 
\end{equation}

For two Hopf algebroids $(A,\cH)$ and $(B,\cK)$, the \emph{category of $(\cH, \cK)$-bicomodules} has 
triples $(P, \lcoaction{\cH}{P}, \rcoaction{\cK}{P})$ as objects, where $P = \due P \ahha \behhe$ is an $(A,B)$-bimodule 
such that $(P , \lcoaction{\cH}{P})$ is a left comodule with a right $B$-linear coaction $\lcoaction{\cH}{P}$, while $(P, \rcoaction{\cK}{P})$ is right comodule with a left $A$-linear coaction $\rcoaction{\cK}{P}$, and both coactions are compatible in the sense that
\begin{equation}\label{Eq:bicomod}
(\cH\tensor{A}\rcoaction{\cK}{P}) \circ \lcoaction{\cH}{P} \,\,=\,\,  
(\lcoaction{\cH}{P}\tensor{B}\cK) \circ \rcoaction{\cK}{P}.
\end{equation}
In other words, $\lcoaction{\cH}{P}$ is a morphism of right $\cK$-comodules, and $\rcoaction{\cK}{P}$ of left $\cH$-comodules, where the codomains of both maps are comodules according to the functor of equation \eqref{Eq:RM}.  {\em Morphisms} of bicomodules are defined in a canonical way; denote by  $\bicomod{\cH}{\cK}$ the category of $(\cH,\cK)$-bicomodules.

Next, we recall the definition of the cotensor product. Let $(M, \rcoaction{}{})$ be a right $\cH$-comodule  and $(N,\lcoaction{}{})$ a  left $\cH$-comodule. The \emph{cotensor product  bifunctor} is defined as the equaliser 
$$
\xymatrix{ 0 \ar@{->}[r]  & M\cotensor{\cH}N \ar@{->}^-{}[r] & M\tensor{A}N 
\ar@<0.5ex>@{->}^-{\scriptstyle{\rcoaction{}{}\tensor{A}N}}[rr] \ar@<-0.5ex>@{->}_-{\scriptstyle{M\tensor{A}\lcoaction{}{}}}[rr]  & & M\tensor{A}\cH\tensor{A} N,   }
$$
which is a bifunctor from the product category $\rcomod{\cH}\times \lcomod{\cH}$  to $\rmod{A}$. If we further assume that $(N,\rcoaction{}{},\lcoaction{}{})$ is also an $(\cH, \cK)$-bicomodule, 
the  cotensor product lands in the category of right $\cK$-comodules since our Hopf algebroids are flat. 
This way, it is possible to define the bifunctor
\begin{equation}
\label{Eq:bicotensor}
-\cotensor{\cH}-: \bicomod{\cJ}{\cH}\times \bicomod{\cH}{\cK} \to \bicomod{\cJ}{\cK}.
\end{equation}

One easily checks that $\cH\cotensor{\cH}N \cong N$ and $A\cotensor{\cH}N \cong N^{\scriptscriptstyle{\coinv_{\cH}}} $ for every right $\cH$-comodule $N$. 

The associativity of the cotensor products is not always guaranteed unless one makes more assumptions on the comodules involved. For example, since all our Hopf algebroids are assumed to be flat, if $M$ is a flat $A$-module  along with a flat $B$-module $N'$, one has
$$
M \cotensor{\cH} (N \cotensor{\cK} N') \simeq (M \cotensor{\cH} N) \cotensor{\cK} N'.
$$
Compare, for example, \cite[\S\S22.5--22.6]{BrzWis:CAC} 
for more situations in which this associativity holds true.

\smallskip

Given a morphism 
$\B{\phi}=(\phi_{\scriptscriptstyle{0}},\phi_{{1}}) :(A,\cH) \to (B,\cK)$ of Hopf algebroids, there is a functor 
\begin{equation}\label{Eq:indfunct}
\B{\phi}_{*}:= - \tensor{\scriptscriptstyle{\phi}}B: \rcomod{\cH} \longrightarrow \rcomod{\cK},
\end{equation} 
called \emph{the induction functor}, which is defined on objects by sending any right comodule $(M,\rcoaction{\cH}{M})$ to a right comodule $(M\tensor{\scriptscriptstyle{\phi}}B, \rcoaction{\cK}{M\tensor{\scriptscriptstyle{\phi}}B})$ with underlying $B$-module $M\tensor{A}B$ and coaction 
$$
\rcoaction{\cK}{M\tensor{\scriptscriptstyle{\phi}}B}: M\tensor{\scriptscriptstyle{\phi}}B \to (M\tensor{\scriptscriptstyle{\phi}}B)\tensor{B} \cK, \quad m\tensor{A}b \mapsto (m_{(0)}\tensor{A}1_{\scriptscriptstyle{B}}) \tensor{B} \phi_{\scriptscriptstyle{1}}(m_{(1)}) \Sf{t}(b).
$$
The image of $\cH$ with the induction functor is, in fact, an $(\cH,\cK)$-bicomodule. 
In a similar way, we have the induction functor 
$$
{}_{*}\B{\phi}:=B\tensor{\scriptscriptstyle{\phi}}-: \lcomod{\cH} \to \lcomod{\cK},
$$
between left comodules,
and $B\tensor{\scriptscriptstyle{\phi}}\cH$ is now an $(\cK,\cH)$-bicomodule. 
The induction functor has a right adjoint
given by 
\begin{equation}
\label{Eq:coind}
-\cotensor{\cK}(B\tensor{\scriptscriptstyle{\phi}}\cH): \rcomod{\cK} \to \rcomod{\cH},
\end{equation}
called the \emph{coinduction functor}.

\subsection{Comodule algebras}
\label{ssec:comoalg}

Parallel to subsection \ref{ssec:PB}, 
we next want to give the analogue notion of groupoid-sets in the Hopf algebroids context. To this end, recall first that a \emph{left $\cH$-comodule algebra} for a Hopf algebroid $(A,\cH)$ is a commutative  monoid in the symmetric monoidal category $\lcomod{\cH}$. That is, a pair $(R,\sigma)$ consisting of a commutative  $A$-algebra $\sigma: A\to R$ which is also a left $\cH$-comodule with coaction $\lcoaction{\cH}{R}: R \to  \cH\tensor{A}R$, 
satisfying for all $x, y \in R$
\begin{equation}\label{Eq:comodalg}
\lcoaction{\cH}{R}(xy) \,=\, x_{(-1)}y_{(-1)} \tensor{A} x_{(0)}y_{(0)} 
\quad \mbox{and} \quad \lcoaction{\cH}{R}(1_\erre)\,=\, 1_{\cakka}\tensor{A}1_{\erre}.
\end{equation} 
In others words, the coaction $\lcoaction{\cH}{R}$ is an $A$-algebra map, where $\cH\tensor{A}R$ is seen as an $A$-algebra via $A \to \cH \otimes_\ahha R,\ a  \mapsto \Sf{s}(a)\tensor{A} 1_\erre$. 
A \emph{morphism of left $\cH$-comodule algebras} is an $A$-algebra map which is also a left $\cH$-comodule morphism. 
{\em Right} $\cH$-comodule algebras are analogously defined. 

Note that for a left $\cH$-comodule algebra $(R,\sigma)$ the $\Bbbk$-vector subspace 
$$
R^{\coinv_{\cH}}=\{x \in R \mid  \lcoaction{\cH}{R}(x)=1_{\cakka}\tensor{A}x \}
$$ 
of $\cH$-coinvariant elements 
is a $\Bbbk$-subalgebra of $R$ that does not necessarily contain the image $\sigma(A)$, unless one makes more assumptions; for instance, if the source and the target maps are equal.  
A trivial example of a comodule algebra is the base algebra $A$ of a Hopf algebroid $(A, \cH)$ itself.

Assume now that $\gamma: B \to R$ is another algebra map such that $\lcoaction{\cH}{R}$ is right  $B$-linear, that is,  
$$
\lcoaction{\cH}{R}(x\,\gamma(b))\,=\, x_{(-1)}\tensor{A}x_{(0)} \gamma(b), 
$$ 
for every $x \in R$ and $b \in B$. One can easily see that $\gamma(B) \subseteq R^{\coinv_{\cH}}$. 
In
this situation, the  \emph{canonical map} 
\begin{equation}
\label{Eq:canhr}
\can{\cH}{R}: R\tensor{B}R \to \cH\tensor{A}R, \quad x\tensor{B}y \mapsto x_{(-1)}\tensor{A}x_{(0)} y
\end{equation} 
is a $B$-algebra map, where $\cH\tensor{A}R$ is a $B$-algebra via $\gamma$ in the second factor. 
The canonical map is also left $\cH$-colinear, when $R\tensor{B}R$ is seen as a left comodule via the coaction $\lcoaction{\cH}{R}\tensor{B}R$. 

We have the following well-known properties:
\begin{lemma}
\label{lemma:comdalg}
Assume that $R$ carries a left $\cH$-comodule algebra structure with underlying algebra map $\gs: A \to R$ and that $\gamma: B \to R$ is a morphism of algebras. 
\begin{enumerate}
\item The pair $(R, \cH\tensor{A}R)$ is a Hopf algebroid with the following structure maps:
\begin{equation*}
\begin{array}{rclrcl}
\mathsf{s} &:=&  \lcoaction{\cH}{R}, & \mathsf{t} &:=& 1_{\cakka}\tensor{A} -, \\
\varepsilon(u \otimes_\ahha r) &:=& \varepsilon_\cakka(u)r, & \Delta(u \otimes_\ahha r) &:=& (u_{(1)} \otimes_\ahha 1_\erre) \otimes_\erre (u_{(2)} \otimes_\ahha r), \\
&&& \mathscr{S} (u \otimes_\ahha r) &:=& \mathscr{S}_\cakka(u)r_{(-1)}\tensor{A}r_{(0)}.
\end{array}
\end{equation*}
\item The map 
$
(\sigma, -\tensor{A}1_{\scriptscriptstyle{R}}):  (A,\cH) \to (R,\cH\tensor{A}R)
$
is a morphism of Hopf algebroids.
\item If $\lcoaction{\cH}{R}$ is right $B$-linear, where $R$ is seen as an $(A,B)$-bimodule, then the canonical map of Eq.~\eqref{Eq:canhr} is a morphism of Hopf algebroids as well as a morphism of left $\cH$-comodules.
\item If $R$ is an $(\cH,\cK)$-bicomodule, then the canonical map 
$$
\can{\cH}{R}: (R {\otimes^{\scriptscriptstyle{B}}} R,\rcoaction{\cK}{R{\otimes^{\scriptscriptstyle{B}}}R}) \to (\cH\tensor{A}R,\cH\tensor{A}\rcoaction{\cK}{R})
$$
is also a morphism of right $\cK$-comodules. 
\end{enumerate}
\end{lemma}
\begin{proof}
These are routine computations.  
\end{proof}

In analogy to groupoid terminology as in \S\ref{ssec:PB}, the Hopf algebroid $(R,\cH\tensor{A}R)$ of Lemma \ref{lemma:comdalg} is termed the \emph{left translation Hopf algebroid of $(A,\cH)$ along $\sigma$}. Symmetrically, one can define a \emph{right translation Hopf algebroid of $(A,\cH)$} by employing {\em right} comodule algebras. 

\begin{rem}
\label{remark:Orbits}
In subsection \ref{ssec:PB}, we discussed the notion of {\em orbit set} of a given left $\Gg$-set over a groupoid $\Gg$.  In the Hopf algebroid context, the analogous notion is given as follows: 
%
for a Hopf algebroid $(A,\cH)$ and any commutative algebra $C$, one can consider its underlying presheaf of groupoids,
 canonically defined by $C \to (\Hh(C), A(C))=({ \rm Alg}_{\Bbbk}(\cH,C), {\rm Alg}_{\Bbbk}(A,C))$ is the  groupoid
 $\xymatrix@C=25pt{\cH(C)\ar@<0.5ex>@{->}[r] \ar@<-0.5ex>@{->}[r] & \ar@{->}[l] A(C)}$ defined by reversing the structure maps of $(A,\cH)$. 
This leads then to the \emph{orbit presheaf} $C \mapsto \Oo(C):=A(C)/\Hh(C)$. 
Clearly, there is a morphism  $\Oo \to {\rm Alg}_{\scriptscriptstyle{\Bbbk}}(A^{{\coinv_{\cH}}},-)$ of presheaves, 
where $A^{{\coinv_{\cH}}}$ is the coinvariant subalgebra of $A$, that is, the set of elements $a \in A$ such that $\Sf{s}(a)=\Sf{t}(a)$. Thus, $A^{\coinv_{\cH}}$ can be thought of as the coordinate ring of the orbit space. 
In case of a general  left $\cH$-comodule algebra $(R,\alpha)$ and for any commutative algebra $C$,
the groupoid $\Hh(C)$ acts on $R(C)$ via $(g, x) \mapsto gx$ given by the algebra map
$$
gx: R \to C, \quad r \mapsto g(r_{(-1)})x(r_{(0)}).
$$
This determines the presheaf $\Oo_{\scriptscriptstyle{R}}: C \mapsto R(C)/\Hh(C)$ of orbits together with a morphism of presheaves $\Oo_{\scriptscriptstyle{R}} \to {\rm Alg}_{\scriptscriptstyle{\Bbbk}}\big(R^{\scriptscriptstyle{\coinv_{\cH}}}, - \big)$. So as before, $R^{\scriptscriptstyle{\coinv_{\cH}}}$ is the coordinate ring of the orbit space.  On the other hand, one can easily check that $R^{{\coinv_{\cH}}} =  R^{{\coinv_{(\cH\tensor{A}R)}}}$, where $(R,\cH\tensor{A}R)$ is the left translation Hopf algebroid as above.
\end{rem}

\subsection{The coinvariant subalgebra for the tensor product of comodule algebras}
\label{ssec:pcomalg}
For any two left $\cH$-comodule algebras $(R,\alpha)$ and $(S,\sigma)$, the comodule tensor product 
$S{\tensor{A}}R$ is an $A$-algebra by means of the algebra map 
$A \to S\tensor{A}R, \ a \mapsto \sigma(a)\tensor{A}1_{\scriptscriptstyle{R}}= 1_{\scriptscriptstyle{R}}\tensor{A} \alpha(a)$. 
This algebra clearly admits the structure of a left $\cH$-comodule algebra the coinvariant subalgebra of it can be described as follows:
\begin{lemma}
\label{lemma:prod}
For any two left $\cH$-comodule algebras $(R,\alpha)$ and $(S,\sigma)$, we have an isomorphism 
$$ 
(S{\tensor{A}}R)^{{\coinv_{\cH}}} \,\cong\, S^{\scriptscriptstyle{o}}\cotensor{\cH}R
$$ 
of algebras, 
where $(S^{\scriptscriptstyle{o}},\sigma)$ is the opposite right $\cH$-comodule algebra of $(S,\sigma)$. 
\end{lemma}
\begin{proof}
For an element $s \tensor{A}r \in (S{\tensor{A}}R)^{\scriptscriptstyle{\coinv_{\cH}}}$, the equality
\begin{equation}
\label{Eq:sr}
1_{\scriptscriptstyle{\cH}}\tensor{A}s\tensor{A}r\,\,=\,\, s_{(-1)}r_{(-1)}\tensor{A}s_{(0)}\tensor{A}
r_{(0)}
\end{equation}
holds in $\cH\tensor{A}S\tensor{A}R$. Applying $(\id_\cakka \otimes m_\cakka \otimes \id_\erre) \circ \tau_{12} \circ (\mathscr{S} \otimes \id_\esse \otimes \lcoaction{R}{\cH})$ to both sides, where $\tau_{12}$ denotes the tensor flip and $m_\cakka$ the multiplication in $\cH$, we obtain
\begin{equation*}
\begin{split}
s\tensor{A}r_{(-1)}\tensor{A}r_{(0)} &= s_{(0)} \tensor{A}\mathscr{S}(s_{(-1)}r_{{(-2)}})
r_{(-1)}\tensor{A}r_{(0)} \\ 
&= s_{(0)} \tensor{A}\mathscr{S}\big(s_{(-1)}\big) \Sf{t}\big(\varepsilon(r_{(-1)})\big)\tensor{A}r_{(0)}
\\  
&= 
s_{(0)} \tensor{A}\mathscr{S}\big(s_{(-1)}\big) \tensor{A}r,
\end{split}
\end{equation*}
which shows that $s\tensor{A}r \in   S^{\scriptscriptstyle{o}}\cotensor{\cH}R$.  The converse is similarly deduced. 
\end{proof}

\begin{rem}
\label{rem:ortitsCotensor}
Taking Remarks \ref{rem:tensorG}  and \ref{remark:Orbits} into account, Lemma \ref{lemma:prod} describes the analogue of the tensor product over groupoids in the Hopf algebroid context. That is, the cotensor product of (left and right) $\cH$-comodule algebras should be thought of as the orbit space of their tensor product as comodule algebras. 
\end{rem}

\subsection{Bicomodule algebras and two-sided translation Hopf algebroids}
\label{ssec:bicomodalg}

In what follows, we give the construction for Hopf algebroids analogous to the two-sided translation groupoid as expounded in \S\ref{sec:PSets}, and show some corresponding results.

For two Hopf algebroids $(A,\cH)$ and $(B,\cK)$, consider an $(\cH,\cK)$-bicomodule $P$ such that $(P,\alpha)$ is a left $\cH$-comodule algebra and $(P,\beta)$ is a right $\cK$-comodule algebra.  We then say that the triple $(P,\alpha,\beta)$ is an \emph{$(\cH,\cK)$-bicomodule algebra}. A \emph{morphism of $(\cH,\cK)$-bicomodule algebras} is a map which is simultaneously  a morphism of left $\cH$-comodule algebras  and right $\cK$-comodule algebras. 

\begin{lemdfn}
\label{lemma:bisemidirect}
Let $(P, \alpha, \beta)$ be an $(\cH,\cK)$-bicomodule algebra. Then $(P, \cH\tensor{A}P\tensor{B}\cK)$ with tensor product 
defined by $\cH \lJoin P \rJoin \cK:=\shopf{\cH}\tensor{A}P\tensor{B} \shopf{\cK}$ carries a canonical structure of a flat Hopf algebroid the structure maps of which are given by:
\begin{enumerate}
\item the source and target are given by 
\begin{equation}
\label{jugendohnegott} 
\Sf{s}(p)\,:=\, 1_{\scriptscriptstyle{\cH}} \tensor{A}p\tensor{B}1_{\scriptscriptstyle{\cK}}, \quad \Sf{t}(p)\,:=\, \mathscr{S}(p_{(-1)}) \tensor{A}p_{(0)}\tensor{B}p_{(1)}; 
\end{equation}
\item the comultiplication and counit are as follows:
$$
\Delta(u\tensor{A}p\tensor{B}w) \,:= \, \big( u_{(1)}\tensor{A}p\tensor{B}w_{(1)}\big)\tensor{P}\big( u_{(2)}\tensor{A}1_{\scriptscriptstyle{P}}\tensor{B}w_{(2)}\big) ,\quad \varepsilon(u\tensor{A}p\tensor{B}w)\,:=\, \alpha\big( \varepsilon(u)\big)p \beta\big( \varepsilon(w)\big);
$$
\item whereas the antipode is defined as:  
$$ 
\mathscr{S}\big( u\tensor{A}p\tensor{B}w\big)\,:=\, \mathscr{S}(u p_{(-1)})\tensor{A}p_{(0)}
\tensor{B}p_{(1)}\mathscr{S}(w).$$
\end{enumerate}
Furthermore, there is a diagram
$$
\xymatrix@R=15pt{ & (P,\cH \lJoin P \rJoin \cK) & \\ (A,\cH) \ar@{->}^-{{\B{\alpha}=(\alpha, \,\alpha_1)}}[ur] & & \ar@{->}_-{{\B{\beta}=(\beta, \,\beta_1)}}[ul] (B,\cK) }
$$ 
of Hopf algebroids,
where $\alpha_{1}$ and $\gb_1$ are the maps $h \mapsto h \otimes_\ahha 1_\pehhe \otimes_\behhe 1_{\scriptscriptstyle{\cK}}$ and  
$k \mapsto 1_{\scriptscriptstyle{\cH}} \otimes_\ahha 1_\pehhe \otimes_\behhe k$, respectively.
This Hopf algebroid will be termed \emph{two-sided translation Hopf algebroid}.
\end{lemdfn}

\begin{proof}
The fact that $\Sf{s}: P \to \shopf{\cH}\tensor{A}P\tensor{B} \shopf{\cK}$ is a flat extension is clear since $\shopf{\cH}$ and $\shopf{\cK}$ are flat; hence $\shopf{\cH}\tensor{A}P\tensor{B} \shopf{\cK}$ will give a flat Hopf algebroid over $P$. 
Using the source map \rmref{jugendohnegott}, the comultiplication $\gD$ and the counit $\varepsilon$ are obviously left $P$-linear; the
right $P$-linearity follows from
$$
\varepsilon\big( (u\tensor{A}p'\tensor{B}w)  \Sf{t}(p)\big)\,=\, \varepsilon\big(u\mathscr{S}(p_{(-1)})\tensor{A}p'p_{(0)} 
\tensor{B}p_{(1)}w\big) \,\overset{\eqref{Eq:penultima}}{=} \, \varepsilon(u\tensor{A}p' \tensor{B}w) p
$$
as well as
\begin{eqnarray*}
\Delta\big( (u\tensor{A}p'\tensor{B}w) \, \Sf{t}(p)\big) &=& \Delta\big(u\mathscr{S}(p_{(-1)})\tensor{A}p'p_{(0)} 
\tensor{B}p_{(1)}w\big)
\\ &=&  \big( u_{(1)}\mathscr{S}(p_{(-1)})\tensor{A}p'p_{(0)}\tensor{B}w_{(1)}p_{(1)}\big)\tensor{P}\big( u_{(2)}\mathscr{S}(p_{(-2)})\tensor{A}1_{\scriptscriptstyle{P}}\tensor{B}w_{(2)}p_{(2)}\big) 
\\ &=& (u_{(1)}\tensor{A}p'\tensor{B}w_{(1)}) \, \Sf{t}(p_{(0)}) \tensor{P}\big( u_{(2)}\mathscr{S}(p_{(-1)})\tensor{A}1_{\scriptscriptstyle{P}}\tensor{B}w_{(2)}p_{(1)}\big) 
\\ &=& \big( u_{(1)}\tensor{A}p'\tensor{B}w_{(1)}\big)\tensor{P}\big( u_{(2)}\mathscr{S}(p_{(-1)})\tensor{A}p_{(0)}\tensor{B}w_{(2)}p_{(1)}\big) 
\\ &=& \big( u_{(1)}\tensor{A}p'\tensor{B}w_{(1)}\big)\tensor{P} (u_{(2)}\tensor{A}1_{\scriptscriptstyle{P}}\tensor{B}w_{(2)}) \, \Sf{t}(p).
\end{eqnarray*}
In order to define a Hopf algebroid, we need these maps to satisfy Eqs.~\eqref{Eq:es}--\eqref{Eq:ultima}, which are either clear from definitions or follow by computations similar to the subsequent one proving  \eqref{Eq:penultima}: 
we have 
\begin{eqnarray*}
\mathscr{S} (u_{(1)}\tensor{A}p\tensor{B}w_{(1)}) (  u_{(2)}\tensor{A} 1_{\scriptscriptstyle{P}}\tensor{B}w_{(2)}) 
&=& \mathscr{S}(u_{(1)})\mathscr{S}(p_{(-1)}) u_{(2)} \tensor{A}p_{(0)}\tensor{B}p_{(1)}\mathscr{S}(w_{(1)}) w_{(2)}
\\ &=& \Sf{t}(\varepsilon(u))\mathscr{S}(p_{(-1)})  \tensor{A}p_{(0)}\tensor{B}p_{(1)}\Sf{t}(\varepsilon(w))
\\ &=& \mathscr{S}\big( \Sf{s}(\varepsilon(u)) p_{(-1)} \big)  \tensor{A}p_{(0)}\tensor{B}p_{(1)} \Sf{t}(b)
\\ &=& \Sf{t}\big( \alpha(\varepsilon(u)) \, p \, \beta(\varepsilon(w))\big)
= \Sf{t}\big(\varepsilon( u\tensor{A}p\tensor{B}w) \big). 
\end{eqnarray*}
The last statement is easily checked as well. 
\end{proof}

Finally note that
for
a morphism 
$\mathfrak{f}: (P,\alpha,\beta) \to (P',\alpha',\beta')$ 
of 
 $(\cH,\cK)$-bicomodule algebras, Lemma \ref{lemma:bisemidirect} leads to a commutative diagram 
\begin{equation}\label{Eq:qseramana}
\xymatrix@R=20pt@C=30pt{ & (P,\cH \lJoin P \rJoin \cK)  \ar@{->}|-{{(\fk{f},\, \cH\tensor{A}\fk{f}\tensor{B}\cK)}}[dd]  & \\ (A,\cH) \ar@{->}|-{{\B{\alpha}}}[ur] \ar@{->}|-{{\B{\alpha '}}}[dr] & & \ar@{->}|-{{\B{\beta}}}[ul] (B,\cK) \ar@{->}|-{{\B{\beta '}}}[dl] \\   & (P',\cH \lJoin P' \rJoin \cK) &  }
\end{equation}
of flat Hopf algebroids.

\begin{example}\label{exm:bhb}
Let $(A,\cH)$ be a Hopf algebroid, $C$ any algebra, and  $h: \cH \to C$ an algebra morphism. Using $\phiup := h \circ \Sf{s}: A \to C$ and $\psiup := h \circ \Sf{t}: A \to C$, construct the scalar extension Hopf algebroids $(C,\cH_{\scriptscriptstyle{\phiup}}:=C\tensor{\phiup}\cH\tensor{\phiup}C)$ resp.\ $(C,\cH_{\scriptscriptstyle{\psiup}}:=C\tensor{\psiup}\cH\tensor{\psiup}C)$, 
where we used the notation $\tensor{\phiup}$ resp.\ $\tensor{\psiup}$  to distinguish between the two $A$-module structures on $C$ given by either $\phiup$ or $\psiup$. From \cite[Lemma 6.4]{HovStr:CALEHT} we deduce that $(C,\cH_{\scriptscriptstyle{\phiup}}) \cong (C,\cH_{\scriptscriptstyle{\psiup}})$ as Hopf algebroids; indeed, this isomorphism is explicitly given by:
$$
C\tensor{\phiup}\cH\tensor{\phiup}C \to C\tensor{\psiup}\cH\tensor{\psiup}C, \quad c\tensor{\phiup}u\tensor{\phiup}c' \mapsto  c\, h(u_{(1)})\tensor{\psiup}u_{(2)}\tensor{\psiup}h\big(\mathscr{S}(u_{(3)})\big)c',
$$ 
with inverse 
$d \tensor{\psiup} v \tensor{\psiup} d' \mapsto  d \, h\big(\mathscr{S}(v_{(1)})\big) \tensor{\phiup} v_{(2)} \tensor{\phiup} h(v_{(3)}) d'$. 

Now, assume that $C$ is of the form $C := B\tensor{\phi}\cH\tensor{\psi}B'$ for some extensions $\xymatrix@C=15pt{B &\ar@{->}_-{\scriptscriptstyle{\phi}}[l] A \ar@{->}^-{\scriptscriptstyle{\psi}}[r] & B'}$ along with the obvious algebra map $h: \cH \to C$ as well as $\phiup: A \to C$ and $\psiup:A \to C$. We can consider 
$(C,\phiup,\psiup)$ as an $(\cH_{\scriptscriptstyle{\phi}},\cH_{\scriptscriptstyle{\psi}})$-bicomodule algebra in a canonical way; this, in fact, is the bicomodule algebra arising from the cotensor product algebra $P^{\co}\cotensor{\cH}P$ by considering, respectively,  $P:=\cH\tensor{\phi}B$ and $P':=\cH\tensor{\psi}B'$ as $(\cH,\cH_{\scriptscriptstyle{\phi}})$- and $(\cH,\cH_{\scriptscriptstyle{\psi}})$-bicomodule algebras with obvious coactions.  

Let $(C,\cH_{\scriptscriptstyle{\phi}} \lJoin C \rJoin \cH_{\scriptscriptstyle{\psi}})$ be the associated two-sided translation Hopf algebroid. Then one can show that there is an isomorphism 
$$
(C, \cH_{\scriptscriptstyle{\phiup}}) \, \cong\, (C,\cH_{\scriptscriptstyle{\phi}} \lJoin C \rJoin \cH_{\scriptscriptstyle{\psi}}) \, \cong\, (C, \cH_{\scriptscriptstyle{\psiup}})
$$
of Hopf algebroids as can be seen by adapting the proof of Proposition \ref{prop:Cielitomissyoutoo} below.
\end{example}

\section{Principal bibundles in the Hopf algebroid context}
\label{sec:pb}

\subsection{General definitions}
\label{ssec:gpb}
In this section, we will introduce one of the main notions in this article. Similar concepts in the framework of Hopf algebras appeared under the name {\em quantum principal bundle} in \cite{BrzMaj:QGGTOQS, Brz:TMIQPP} or {\em bi-Galois extension} in \cite{Schau:HBGE, Schau:HGABGE}. In analogy to Definition \ref{def:pbset}, we define principal bundles in the Hopf algebroid context as follows. 
\begin{definition}
\label{def:PB}
A \emph{left principal $(\cH, \cK)$-bundle} $(P, \ga, \gb)$ 
for two Hopf algebroids $(A, \cH)$ and $(B, \cK)$ 
is an $(\cH,\cK)$-bicomodule algebra as in \S\ref{ssec:comoalg}, that is,  $P$  is equipped with a left $\cH$-comodule algebra and a right $\cK$-comodule algebra structures with respect to the algebra maps $\alpha: A \to P$ resp.\ $\beta: B \to P$ such that  
\begin{enumerate}
\item $\beta$ is a faithfully flat extension; 
\item the canonical map
$$
\can{\cH}{P}: P \otimes_\behhe P \to \cH\tensor{A}P, \quad p\tensor{B}p' \mapsto p_{(-1)} \tensor{A} p_{(0)} p'
$$ 
is bijective.
\end{enumerate}
\end{definition}

At times, when the context is clear and hence (we think that) no confusion can arise, 
the subscripts in the notation $\mathsf{can}$ of the canonical map are dropped.

Maps between principal bundles are defined as follows:

\begin{definition}
A \emph{morphism of left principal $(\cH, \cK)$-bundles}  $(P, \ga, \gb)$ and $(P', \ga', \gb')$ is a map  $\fk{f}:P \to P'$ that is 
a morphism of $(\cH,\cK)$-bicomodule algebras, {\em i.e.}, simultaneously a morphism of $A$-algebras, $B$-algebras, and a morphism of $(\cH, \cK)$-bicomodules. 
We will also call such a morphism an \emph{equivariant morphism}.  
An \emph{isomorphism of left principal bundles} is a bijective morphism of left principal bundles.
The category of left  principal $(\cH,\cK)$-bundles will be denoted by $\lPB{\cH}{\cK}$.
\end{definition}

Let us denote the inverse of $\can{\cH}{P}$ by a sort of Sweedler type notation,
$$
\can{\cH}{P}^{-1}: \cH\tensor{A}P \to P \tensor{B} P, \quad u \tensor{A} p \mapsto u_+\tensor{B} u_- p.
$$ 
where
\begin{equation}
\label{plusminus}
\tauup_\pehhe:= \can{\cH}{P}^{-1}(- \otimes_\ahha 1_\pehhe): \cH \to P\tensor{B}P, \quad u \mapsto u_+ \tensor{B} u_-
\end{equation}
denotes the \emph{translation map}. 
The following lemma summarises the properties of this map and its compatibility with the Hopf algebroid structure:

\begin{lemma}
Let  $(P, \alpha,\beta)$ be a left principal $(\cH,\cK)$-bundle.
One has for all $a, a' \in A$, $u, v \in \cH$, and $p \in P$: 
\begin{eqnarray}
\label{ceuta}
(uv)_+\tensor{B} (uv)_- &=& u_+ v_+ \tensor{B} v_- u_-, \\ 
\label{japan} 
u_{+(-1)} \tensor{A} u_{+(0)} \tensor{B} u_{-} &=& u_{(1)} \tensor{A} u_{(2)+} \tensor{B} u_{(2)-}, 
\\ 
\label{Eq:eps+-} 
u_+u_-&=& \alpha(\varepsilon(u)), 
\\
\label{Eq:p-+} 
p_{(-1)+}  \tensor{B} p_{(-1)-} p_{(0)} &=& p \tensor{B} 1_{\scriptscriptstyle{P}}, 
\\
\label{Eq:u-+} 
u_{+(-1)} \tensor{A} u_{+(0)} u_- &=& u \tensor{A} 1_{\scriptscriptstyle{P}}, 
\\
\label{bolognacentrale}
(\mathsf{s}(a)\mathsf{t}(a'))_+ \otimes_\behhe (\mathsf{s}(a)\mathsf{t}(a'))_- &=& \ga(a) \otimes_\behhe \ga(a').
\end{eqnarray}
Furthermore,
\begin{eqnarray}
\label{casadiaugusto1}
u_{+(0)} \otimes_\behhe u_{-(0)} \otimes_\behhe  u_{+(1)} u_{-(1)} &=& u_{+} \otimes_\behhe  u_{-} \otimes_\behhe 1_\cK \quad \in P \otimes_\behhe P \otimes_\behhe \cK, \\
\label{Eq:Su}
\mathscr{S}(u) \otimes_\ahha 1_\pehhe &=& u_{{-(-1)}} \otimes_\ahha u_{{-(0)}} u_{{+}}, \\
\label{zucchero}
\mathscr{S}(u)_+ \tensor{B} \mathscr{S}(u)_- &=& u_- \tensor{B} u_+, \\
\label{Eq:S+-}
u_{{(1)+}} \otimes_\behhe u_{{(1)-}} \tensor{A} \mathscr{S}(u_{{(2)}}) &=& u_+ \otimes_\behhe u_{{-(0)}} \tensor{A} u_{{-(-1)}}.
\end{eqnarray}
\end{lemma}
\begin{proof}
The first six equations are proved along the lines of the proof of \cite[Prop.~3.7]{Schau:DADOQGHA}, where the special case in which $P:=\cH$ is treated. Eq.~\rmref{casadiaugusto1} is obtained by the fact that the canonical map (and hence its inverse) is a morphism of right $\cK$-comodules, as follows from Lemma \ref{lemma:comdalg} {\em (iv)}.
Eq.~\rmref{Eq:Su} is proven as follows: since $P$ is a left $\cH$-comodule algebra and the coaction is $A$-linear, one has 
\begin{equation*}
\begin{split}
\mathscr{S}(u) \otimes_\ahha 1_\pehhe &= \mathscr{S}(u_{(1)})\mathsf{s}(\gve(u_{(2)})) \otimes_\ahha 1_\pehhe 
= \mathscr{S}(u_{(1)})\big(\ga(\gve(u_{(2)}))\big)_{(-1)} 
\otimes_\ahha \big(\ga(\gve(u_{(2)}))\big)_{(0)} \\
& \!\! \overset{\rmref{Eq:eps+-}}{=} 
 \mathscr{S}(u_{(1)})(u_{(2)+}u_{(2)-})_{(-1)} 
\otimes_\ahha (u_{(2)+}u_{(2)-})_{(0)} \overset{\rmref{japan}}{=} 
 \mathscr{S}(u_{(1)}) u_{(2)}u_{(3)-(-1)} 
\otimes_\ahha u_{(3)+}u_{(3)-(0)} \\
& \! \! \overset{\rmref{maxxi}}{=} 
\mathsf{t}(\gve(u_{(1)})) u_{(2)-(-1)} 
\otimes_\ahha u_{(2)+}u_{(2)-(0)} 
 \overset{\rmref{bolognacentrale}}{=} 
u_{-(-1)} 
\otimes_\ahha u_{-(0)} u_+. 
\end{split}
\end{equation*}
Eq.~\rmref{zucchero} now follows by simply applying the inverse of the canonical map to both sides, using \rmref{Eq:p-+}.
Finally, Eq.~\rmref{Eq:S+-} is seen by applying \rmref{japan} to the element $\mathscr{S}(u)$, using \rmref{zucchero} and the fact that the antipode is an anti-coring morphism.
\end{proof}

Right  principal bundles use the  right $\cK$-comodule algebra structure of $P$ and the canonical map:
$$
\can{P}{\cK}: P \tensor{A} P \to P\tensor{B}\cK, \quad p'\tensor{A}p \mapsto p'p_{(0)}\tensor{B}p_{(1)}.
$$ 
In this way,  $P$ is said to be a \emph{right principal $(\cH, \cK)$-bundle} if $\alpha$ is a faithfully flat extension and the canonical map $\can{P}{\cK}$ is bijective. The triple $(P,\alpha,\beta)$ is said to \emph{principal $(\cH, \cK)$-bibundle} provided $P$ is both left and right principal.

Since we will explicitly use principal bibundles, we also need the notation and the properties for the right translation map.  The  inverse of $\can{P}{\cK}$ is  denoted by
$$
P \tensor{B} \cK \to P \otimes_\ahha P, \quad p \otimes_\behhe v \mapsto pv^- \otimes_\ahha v^+,
$$
which fulfils the relations
\begin{eqnarray}
(vw)^+ \otimes_\ahha (vw)^- &=& v^+ w^+ \otimes_\ahha w^- v^-, \\ 
v^- v^+ &=& \gb(\varepsilon(v)), \\
p_{(0)} {p_{(1)}}^- \otimes_\ahha {p_{(1)}}^+ &=& 1_\pehhe \otimes_\ahha p, \label{Eq:colosseo} \\
v^- {v^{+}}_{(0)} \otimes_\behhe {v^{+}}_{(1)} &=& 1_\pehhe \otimes_\behhe v, \label{Eq:DavidagainstGoliat} \\
v^- \otimes_\ahha {v^+}_{(0)} \otimes_\ahha  {v^+}_{(1)} &=& {v_{(1)}}^- \otimes_\ahha {v_{(1)}}^+ \otimes_\ahha  v_{(2)}, \\
(\mathsf{s}(b)\mathsf{t}(b'))^- \otimes_\ahha (\mathsf{s}(b)\mathsf{t}(b'))^+ &=& \gb(b) \otimes_\ahha \gb(b').
\end{eqnarray}
With a similar argumentation that lead to \rmref{casadiaugusto1}, we have the identity
\begin{equation}
\label{casadiaugusto2}
{v^-}_{(-1)} {v^+}_{(-1)} \otimes_\ahha {v^-}_{(0)} \otimes_\ahha {v^+}_{(0)} = 1_\cH \otimes_\ahha {v^-} \otimes_\ahha  {v^+} \quad \in \cH \otimes_\ahha P \otimes_\ahha P. 
\end{equation}
Analogously, one obtains
\begin{eqnarray*}
\mathscr{S}(v)^- \otimes_\ahha \mathscr{S}(v)^+ &=& v^+ \otimes_\ahha v^-, \\
{v_{(2)}}^- \otimes_\ahha {v_{(2)}}^+ \otimes_\behhe \mathscr{S}(v_{(1)}) &=&
v^- \otimes_\ahha {v^+}_{(0)} \otimes_\behhe {v^+}_{(1)} \quad \in P \otimes_\ahha P \otimes_\behhe \cK, \\
1_\pehhe \otimes_\behhe \mathscr{S}(v) &=&  v^+ {v^-}_{(0)} \otimes_\behhe {v^-}_{(1)}.
\end{eqnarray*}
In a similar way, one can define a morphism between right principal $(\cH, \cK)$-bundles. The obtained category  will be denoted by $\rPB{\cH}{\cK}$.
Morphisms of principal bibundles are simultaneously morphisms of left and right principal bundles. The category obtained this way will be denoted by $\bPB{\cH}{\cK}$.

\begin{rem}
\label{rema:PB} \noindent
\begin{enumerate}
\item For a morphism $\fk{f}:(P, \alpha, \gb) \to (P', \alpha', \gb')$ in $\lPB{\cH}{\cK}$, we have a  commutative diagram:
\begin{equation}\label{Eq:barril}
\xymatrix@R=20pt{ \cH \ar@{->}^-{\tauup_P}[rr] \ar@{->}_-{\tauup_{P'}}[rrd]  & & P\tensor{B}P \ar@{->}^-{\fk{f}\tensor{B}\fk{f}}[d] \\ & & P'\tensor{B}P', }
\end{equation}
where $\tauup$ is the corresponding translation map. 
\item 
The definition above is left-right symmetric: if ${}_{\cakka}P_{\ckppa}$ is a left  principal $(\cH, \cK)$-bundle, then the opposite bicomodule ${}_{\ckppa}P^{\co}{}_{\cakka}$ is a right principal $(\cK, \cH)$-bundle with respect to the canonical map
$$
P^{\co} \otimes_\behhe P^{\co} \to P^{\co} \otimes_\ahha \cH, \quad p' \otimes_\behhe p \mapsto p'p_{(0)} \otimes_\ahha \mathscr{S}(p_{(-1)}).
$$ 
Using \rmref{zucchero}, one immediately verifies that 
$$
P^{\co} \otimes_\ahha \cH \to P^{\co} \otimes_\behhe P^{\co}, \quad p \otimes_\ahha h \mapsto ph_+ \otimes_\behhe h_-
$$
defines the inverse of this map. If we denote by $\alpha^{{\co}}: A \to P^{\co}$ and $\beta^{{\co}}: B \to P^{{\co}}$, respectively, the corresponding algebra maps, then the correspondence $(P,\alpha,\beta) \to (P^{{\co}}, \beta^{{\co}}, \alpha^{{\co}})$ establishes an isomorphism of categories between $\lPB{\cH}{\cK}$ and $\rPB{\cK}{\cH}$.  The bundle $(P^{{\co}}, \beta^{{\co}}, \alpha^{{\co}})$ so constructed is called the \emph{opposite bundle} of $(P,\alpha, \beta)$.
\item 
Since $P_\behhe$ is  faithfully flat, we know 
by the faithfully flat descent theory (see, for instance \cite[Theorem 3.10]{Kaoutit/Gomez:2003a}) 
that the  subalgebra of $\cH$-coinvariants is  $P^{{\coinv_{\cH}}}= \beta(B)$ as $\beta$ is injective.  
Moreover, since $\alpha: A\to P$ is a right $\cH$-colinear map, we have the following commutative diagram 
\vspace*{-.5cm}
$$
\xymatrix{ A^{\scriptscriptstyle{\coinv_{\cH}}} \ar@{->}^-{\alpha^{\scriptscriptstyle{\coinv_{\cH}}}}[rr]   \ar@{_{(}->}^-{}[d] & & P^{\scriptscriptstyle{\coinv_{\cH}}} \cong B  \ar@{^{(}->}^-{\beta}[d] \\ A \ar@{->}^-{\alpha}[rr] && P}
$$
of algebras.
On the other hand, the category of relative left comodules, that is, the category of left $(\cH\tensor{A}P)$-comodule is (monoidally) 
equivalent to the category of $B$-modules, where $(P,\cH\tensor{A}P)$ is the translation Hopf algebroid along $\alpha$. 
Conversely, given an $(\cH,\cK)$-bicomodule algebra $(P, \alpha, \beta)$ such that the functor $-\tensor{B}P: \rmod{B} \to \rcomod{\cH\tensor{A}P}$ establishes an equivalence of categories, $(P, \alpha, \beta)$ carries the structure of a left principal $(\cH, \cK)$-bundle.  
\item 
For the trivial Hopf algebroid $(B,\cK):=(B,B)$, a left 
principal $(\cH, B)$-bundle is a left $\cH$-comodule algebra $(P,\alpha)$ with a faithfully flat extension $\beta: B \to P$  
whose $\cH$-coaction is a $B$-linear map and where $\can{\cH}{P}: P\tensor{B} P \to \cH\tensor{A}P$ is bijective. 
\end{enumerate}
\end{rem}

\begin{example}[{\em Unit bundles}]
\label{Exam:H}
The underlying $\cH$-bicomodule of any flat Hopf algebroid $(A,\cH)$ is a left  
principal 
$(\cH, \cH)$-bundle. More precisely, $\cH$ is an $\cH$-bicomodule via the algebra maps $\Sf{s}, \Sf{t}: A \to \cH$ and both ring extensions are faithfully flat by assumption. 
So, we only need to check {\em (ii)} in Definition \ref{def:PB}. In this case we have
$$
\can{\cH}{\cH}: \cH \tensor{A} \cH \to \cH\tensor{A}\cH, \quad u\tensor{A}v \mapsto u_{(1)}\tensor{A} u_{(2)}v,
$$
where the domain tensor product is defined by $\thopf{\cH}$ in both factors, while the codomain tensor product is the standard one from the coproduct of ${\cH}$.
The inverse of $\can{\cH}{\cH}$ is, as for Hopf algebras,  
$$
\can{\cH}{\cH}^{-1}: \cH\tensor{A}\cH \to \cH \tensor{A} \cH, \quad u\tensor{A}v \mapsto u_{(1)}\tensor{A} \mathscr{S}(u_{(2)})v.
$$ 
This bundle is refereed to as the \emph{unit principal bundle} and will be denoted by $\mathscr{U}(\cH)$. Note that $\mathscr{U}(\cH)$ is both a left and a right 
principal $(\cH, \cH)$-bundle, and therefore a principal bibundle. 
\end{example}

\begin{example}[{\em Induced} or {\em pull-back bundles}]
\label{Exam:Ind}
For a morphism $\boldsymbol{\psi}=(\psi_{\scriptscriptstyle{0}},\psi_{\scriptscriptstyle{1}}): (B,\cK)\to (C, \cJ)$ of Hopf algebroids and a left  principal $(\cH, \cK)$-bundle $(P,\alpha,\beta)$, 
consider $P\tensor{B}C$ with the obvious algebra extensions $\td{\alpha}: A \to  P\tensor{B}C$ and $\td{\beta}: C \to P\tensor{B}C$.  It is clear that $\td{\beta}$ is a faithfully flat  extension and that $P\tensor{B}C$ is an $(\cH,\cJ)$-bicomodule: 
its left coaction is $\lcoaction{\cH}{P\tensor{B}C}:=\lcoaction{\cH}{P}\tensor{B}C$ and its right coaction is defined by the composition
$$
\xymatrix@C=40pt{\rcoaction{\cJ}{P\tensor{B}C}: P\tensor{B}C\ar@{->}^-{\scriptscriptstyle{\rcoaction{\cK}{P}\tensor{A}B}}[r] &  P\tensor{B}\cK\tensor{B}C \ar@{->}^-{\scriptscriptstyle{P\tensor{B}\psi_{\scriptscriptstyle{1}}\tensor{B}C}}[r] & P\tensor{B}\cJ\tensor{B}C \ar@{->}^-{\scriptscriptstyle{P\tensor{B}\xiup_{\cJ}}}[r] & (P\tensor{B}C)\tensor{C}\cJ,} 
$$
where $\xiup_{\cJ}: \cJ\tensor{B}C \to C\tensor{C}\cJ, \ w\tensor{B}c \mapsto 1_C\tensor{C}w\Sf{t}(c)$. Explicitly, one obtains 
$$
\rcoaction{\cJ}{P\tensor{B}C}(p\tensor{B}c)\,=\, (p_{(0)}\tensor{B}1_C)\tensor{C} \psi_{\scriptscriptstyle{1}}(p_{(1)})\Sf{t}(c),
$$ 
and both coactions are algebra maps. Thus, $P\tensor{B}C$ is both a left $\cH$-comodule algebra and a right $\cJ$-comodule algebra.
The canonical map $\can{\cH}{P\tensor{B}C}$ is bijective since, up to canonical isomorphisms, it is of the form $\can{\cH}{P}\tensor{B}C$.
Hence, $(P\tensor{B}C, \td{\alpha},\td{\beta})$ is a left 
principal 
$(\cH, \cJ)$-bundle, called the \emph{induced bundle} of $P$ or \emph{pull-back bundle} of $P$, and denoted $\psi^*(P)$ or $\psi^*\big((P,\alpha,\beta)\big)$. 
Of course, this establishes a functor  $\lPB{\cH}{\cK} \to  \lPB{\cH}{\cJ}$.
\end{example}

\begin{example}[{\em Restricted principal bundles}]
\label{exam:ResPB}
For a left 
principal $(\cH, \cK)$-bundle $(P,\alpha, \beta)$ and an algebra map $\tauup: B \to R$, consider the scalar extension Hopf algebroid $(R,\cK_{\scriptscriptstyle{R}}) :=(R,R\tensor{B}\cK\tensor{B}R)$, along with the obvious algebra maps
$\alpha_{\scriptscriptstyle{R}}: A \to P \to P_{\scriptscriptstyle{R}}$ and $\beta_{\scriptscriptstyle{R}}: R \to P_{\scriptscriptstyle{R}}$, where $P_{\scriptscriptstyle{R}}:=P\tensor{B}R$. 
  It is clear that $P_{\scriptscriptstyle{R}}$ admits the structure of an $(\cH,\cK_{\scriptscriptstyle{R}})$-bicomodule with coactions, up to natural isomorphisms, defined by $\lcoaction{\cH}{P_{\scriptscriptstyle{R}}} := \lcoaction{\cH}{P}\tensor{B}R$ and 
$\rcoaction{\cK_{\scriptscriptstyle{R}}}{P_{\scriptscriptstyle{R}}} := \rcoaction{\cK}{P}\tensor{B}R$. These are clearly algebra maps which convert $(P_{\scriptscriptstyle{R}}, \lcoaction{\cH}{P_{\scriptscriptstyle{R}}})$ and  $(P_{\scriptscriptstyle{R}}, \rcoaction{\cK_{\scriptscriptstyle{R}}}{P_{\scriptscriptstyle{R}}})$ into comodule algebras.  The canonical maps are, up to natural isomorphism, given by 
$$
\can{\cH}{P_{\scriptscriptstyle{R}}}\, := \, \can{\cH}{P}\tensor{B}R,\qquad  
\can{P_{\scriptscriptstyle{R}}}{\cK_{\scriptscriptstyle{R}}}\, := \, R\tensor{B}\can{P}{\cK}\tensor{B}R.
$$
Obviously, $\beta_{\scriptscriptstyle{R}}$ is a faithfully flat extension, hence $(P_{\scriptscriptstyle{R}},\alpha_{\scriptscriptstyle{R}}, \beta_{\scriptscriptstyle{R}})$ is  again a left 
principal $(\cH,\cK_{\scriptscriptstyle{R}})$-bundle, and we have that $(P_{\scriptscriptstyle{R}})^{\scriptscriptstyle{\coinv_{\cH}}} \simeq R$. We refer to this construction as the \emph{restricted} principal bundle of $(P,\alpha,\beta)$ with respect to $\tauup$.  
Again, this yields a functor
 $\lPB{\cH}{\cK} \to  \lPB{\cH}{\cK_\erre}$.
\end{example}

\begin{rem}
\label{rem:ir}  \noindent
\begin{enumerate}
\item 
If we assume that $(P, \alpha,\beta)$ in Example \ref{exam:ResPB} is only an $(\cH, \cK)$-bicomodule algebra, then 
\pagebreak
it is possible to compute 
the coinvariant subalgebra $(P_{\scriptscriptstyle{R}})^{\scriptscriptstyle{\coinv_{\cH}}}$ of the restricted $(\cH, \cK_{\scriptscriptstyle{R}})$-bicomodule algebra  $(P_{\scriptscriptstyle{R}}, \alpha_{\scriptscriptstyle{R}}, \beta_{\scriptscriptstyle{R}})$ by means of the coinvariant subalgebra $P^{\scriptscriptstyle{\coinv_{\cH}}}$ provided that $\tauup$ is a flat extension. One then has the following chain of algebra isomorphisms:
$$
(P_{\scriptscriptstyle{R}})^{\scriptscriptstyle{\coinv_{\cH}}} \,\,\cong \,\, A\cotensor{\cH}( P\tensor{B}R)\, \, \cong \,\, ( A\cotensor{\cH}P) \tensor{B}R \,\, \cong \,\,  P^{\scriptscriptstyle{\coinv_{\cH}}} \tensor{B}R.
$$
\item  
For a left principal $(\cH, \cK)$-bundle $(P,\alpha,\beta)$ and a morphism $\B{\psi}=( \psi_{\scriptscriptstyle{0}},\psi_{{1}}): (B,\cK) \to (C,\cJ)$ 
of Hopf algebroids, one can consider the induced left 
principal $(\cH,\cJ)$-bundle $\psi^*((P,\alpha,\beta))$ on the one hand, and the restricted left 
principal $(C,\cK_{\scriptscriptstyle{C}})$-bundle $(P_{\scriptscriptstyle{C}}, \alpha_{\scriptscriptstyle{C}}, \beta_{\scriptscriptstyle{C}})$ on the other hand. 
However,  using the canonical morphism $\B{\Psi}$ of Hopf algebroids associated to $\B{\psi}$ as defined in Eq.~\eqref{Eq:PHI},  the bundle $(P_{\scriptscriptstyle{C}}, \alpha_{\scriptscriptstyle{C}}, \beta_{\scriptscriptstyle{C}})$ induced by $\B{\Psi}$ coincides with $\psi^*(P)$, {\em i.e.},
$$
\psi^*\big((P,\alpha,\beta)\big)\,\,=\,\, \B{\Psi}^*\big((P_{\scriptscriptstyle{C}}, \alpha_{\scriptscriptstyle{C}}, \beta_{\scriptscriptstyle{C}}) \big).
$$
\end{enumerate}
\end{rem}

\begin{example}[{\em Trivial Bundles}]
\label{exam:HopfMorph}
An example of an induced principal bundle is the following, which although rather basic will reveal important in subsequent sections;  \emph{cf}.~ also Example \ref{exm:bhb}. For any morphism $(\phi_{\scriptscriptstyle{0}}, \phi_{\scriptscriptstyle{1}}): (A,\cH) \to (B,\cK)$ of Hopf algebroids, consider 
\begin{equation}
\label{klaipeda}
P := \cH\tensor{\phi}B := \cH \tensor{A} B = \cH \tensor{\Bbbk} B/{\rm span}\{\mathsf{t}(a) u \otimes b - u \otimes \phi_{\scriptscriptstyle{0}}(a)b \mid u \in \cH, b \in B, a \in A \},
\end{equation}
as a left principal $(\cH, \cK)$-bundle  by pulling back the unit bundle $\mathscr{U}(\cH)$. More precisely, consider the following algebra maps:
$$ 
\ga: A \to P, \quad a \mapsto \mathsf{s}(a) \otimes_\ahha 1_\behhe, 
\qquad \mbox{and} \qquad
\gb: B \to P, \quad  b \mapsto 1_{\scriptscriptstyle{\cH}} \otimes_\ahha b.
$$
Obviously, $P_\behhe$ is a faithfully flat module, that is, $\beta$ is a faithfully flat extension.  The algebra $P$ is  an $(\cH,\cK)$-bicomodule with left coaction  $\lcoaction{\cH}{P}:=\Delta_{\cH}\tensor{A}B$ along with the right coaction 
$$
\rcoaction{\cK}{P}: P \to P\tensor{B}\cK,
\quad u\tensor{A}b \mapsto (u_{(1)}\tensor{A}1_B)\tensor{B} \phi_{\scriptscriptstyle{1}}(u_{(2)})\Sf{t}(b).
$$
Both left and right coactions are easily seen to be morphisms of algebras.
The canonical map is defined as 
$$
\can{\cH}{P}: P\tensor{B}P \to \cH\tensor{A}P, \quad (u\tensor{A}b)\tensor{B}(v\tensor{A}b') \mapsto u_{(1)}\tensor{A}(u_{(2)}v\tensor{A}bb'),
$$
which by Example \ref{Exam:H} is clearly bijective, and the corresponding translation map reads:
$$
\tauup_{P}: H \to P\tensor{B}P,\qquad u \mapsto (u_{\scriptscriptstyle{(1)}} \tensor{A}1_{\scriptscriptstyle{B}}) \tensor{B} (\mathscr{S}(u_{\scriptscriptstyle{(2)}}) \tensor{A}1_{\scriptscriptstyle{B}}).
$$
The fact that  the subalgebra of $\cH$-coinvariant elements is isomorphic to $B$, see Remark \ref{rema:PB} {\em (ii)}, can be deduced directly in this case: from the isomorphisms 
$$
A\cotensor{\cH}(\cH\tensor{A}B) \cong (A\cotensor{\cH}\cH)\tensor{A}B\cong B
$$ 
one obtains that $P^{\scriptscriptstyle{\coinv_{\cH}}}\cong A\bx_{\cH}P \cong B$ via $\beta$. 
The second canonical map is in this case given by 
\begin{equation}
\label{Eq:canK}
\can{P}{\cK}: P\tensor{A}P \to P\tensor{B}\cK, \quad (u\tensor{A}b)\tensor{A}(v\tensor{A}b') \mapsto (uv_{(1)}\tensor{A}b)\tensor{B} \phi_{\scriptscriptstyle{1}}(v_{(2)})\Sf{t}(b').
\end{equation}
\end{example}

This example motivates the following definition.

\begin{definition}\label{def:bibundle}
We say that  a left  principal $(\cH, \cK)$-bundle $P$  is \emph{trivial} if it is isomorphic to  an induced bundle of the  unit bundle $\mathscr{U}(\cH)$ as defined in Example \ref{Exam:H}, {\em i.e.}, if there is an isomorphism 
$$
P \,\, \cong\,\, \phi^*(\mathscr{U}(\cH)):=\cH\tensor{\phi}B
$$ 
of principal bundles with respect to some Hopf algebroid morphism $\B{\phi}: (A, \cH) \to (B, \cK)$.
\end{definition}

Sufficient  and necessary conditions under which a left principal bundle is trivial are given in the subsequent proposition.

\begin{proposition}
\label{prop:indPB}
Let $(P,\alpha,\beta)$ be a left principal $(\cH, \cK)$-bundle. The following are equivalent:
\begin{enumerate} 
\item $(P,\alpha, \beta)$ is a  trivial  principal bundle;
\item  $\beta$ splits as an algebra map, that is, there is an algebra map  $\gamma: P \to B$ such that  $\gamma \circ \beta =\id_\behhe$.
\end{enumerate}
\end{proposition}
\begin{proof}
Proving $(i) \Rightarrow (ii)$ is immediate from the definitions. 
To prove $(ii) \Rightarrow (i)$, we first need 
to construct a Hopf algebroid morphism $(\phi_{\scriptscriptstyle{1}}, \phi_{\scriptscriptstyle{0}}):(A,\cH) \to (B, \cK)$. 
Here, the algebra map $\phi_{\scriptscriptstyle{0}}: A \to B$ will be defined as the composition  $\phi_{\scriptscriptstyle{0}}= \gamma \circ \alpha$, whereas $\phi_{\scriptscriptstyle{1}}$ is given by 
$$
\phi_{\scriptscriptstyle{1}}: \cH \to \cK, \quad u \mapsto \Sf{s}(\gamma(u_{+(0)})) u_{+(1)} \Sf{t}(\gamma(u_-)),
$$ 
using the notation in \rmref{plusminus} for the translation map;
a routine computation shows that $\boldsymbol{\phi}=(\phi_{\scriptscriptstyle{0}}, \phi_{\scriptscriptstyle{1}})$ is a morphism of Hopf algebroids, indeed.
Consider then the trivial left principal $(\cH, \cK)$-bundle $\cH\tensor{\scriptscriptstyle{\phi}}B = \cH \tensor{A} B$ as in \rmref{klaipeda}. 
Let us check that 
$$
f: \cH\tensor{A}B \to P, \quad u\tensor{A}b \mapsto u_+ \beta(\gamma(u_-)) \beta(b),
$$
is a bijection 
whose inverse will be  
$$
g:  P \to \cH\tensor{A}B, \quad p \mapsto p_{(-1)} \tensor{A} \gamma(p_{(0)}).
$$
For any $p \in P$, we have
\begin{eqnarray*}
f(g(p)) &=& f\big(p_{(-1)} \tensor{A}\gamma(p_{(0)})\big) \\ 
&=&  p_{(-1)+} \beta(\gamma(p_{(-1)-})) \beta(\gamma(p_{(0)}))\\ 
&=&  p_{(-1)+} \beta\big(\gamma(p_{(-1)-} p_{(0)})\big) \\ &\overset{\scriptscriptstyle{\eqref{Eq:p-+}}}{=}& p \beta(\gamma(1_{\scriptscriptstyle{P}})) \,\,=\,\, p.
\end{eqnarray*}
On the other hand, for any $u\tensor{A}b \in \cH\tensor{A}B$, one computes 
\begin{eqnarray*}
g(f(u\tensor{A}b)) &=& g\big(u_+ \beta\big(\gamma(u_-)\beta(b)\big) \\ 
&=&  u_{+(-1)} \tensor{A} \gamma(u_{+(0)}) \gamma(u_-) b \\ 
&=&  u_{+(-1)} \tensor{A}\gamma\big(u_{+(0)} u_- \big)b  \\ &\overset{\scriptscriptstyle{\eqref{Eq:u-+}}}{=}& u \tensor{A} \gamma(1_{\scriptscriptstyle{P}})b \,\,=\,\, u\tensor{A}b.
\end{eqnarray*} Thus, $f$ and $g$ are mutually inverse. It is also clear that $g$ is both an $A$-algebra and $B$-algebra map, as well as an $(\cH,\cK)$-bicomodule map. Therefore, $g$ is an isomorphism of left principal $(\cH, \cK)$-bundles. 
\end{proof}

The following lemma is an analogue of the respective statement for Lie groupoids in \cite[p.~165]{MoeMrc:LGSAC}. However, the proof given in this context here is direct and does not rely on local triviality of bundles.

\begin{lemma}
\label{lemma:ISO}
Any morphism between  left principal $(\cH, \cK)$-bundles is an isomorphism. In particular, the category of left principal bundles $\lPB{\cH}{\cK}$ is a groupoid.
\end{lemma}
\begin{proof}
Let $\fk{f}:(P,\alpha,\beta) \to (P',\alpha',\beta')$ be a morphism between two left principal $(\cH, \cK)$-bundles. By definition both $\beta$ and $\beta'$ are faithfully flat extensions; hence, it suffices to check that either $\fk{f}\tensor{B}P'$ or $P\tensor{B}\fk{f}$ is an isomorphism as $\fk{f}$ is  an $A$-algebra and $B$-algebra map. To this end, consider the following chain 
$$
\xymatrix@C=30pt{P\tensor{B}P' \ar@{->}^-{\cong}[r] & (P\tensor{B}P)\tensor{P} P' \ar@{->}^-{{\sf{can}}\tensor{P}P'}[r] & (\cH\tensor{A}P)\tensor{P}P' \ar@{->}^-{\cong}[r] &  \cH\tensor{A}P' \ar@{->}^-{\sf{can}^{-1}}[r] & P'\tensor{B}P'} 
$$
of isomorphisms,
where we have used the fact that $\can{\cH}{P}$ is right $P$-linear, is explicitly given by 
\begin{equation*}
p\tensor{B}p' \longmapsto (p\tensor{B}1)\tensor{P}p' \longmapsto p_{(-1)} \tensor{A} p_{(0)} \tensor{P} p' \longmapsto   p_{(-1)} \tensor{A} \fk{f}(p_{(0)})  p' \longmapsto  p_{(-1) +} \tensor{B} p_{(-1) -} \fk{f}(p_{(0)})  p'
\end{equation*}
which by equation \eqref {Eq:p-+} is exactly  the map $p\tensor{B}p' \mapsto \fk{f}(p)\tensor{B}p'$ as $\fk{f}$ is a comodule morphism. Therefore, $\fk{f}\tensor{B}P'$ is an isomorphism and so is $\fk{f}$.
\end{proof}

\subsection{Comments on local triviality of principal bundles}
\label{ssec:ltpb}
In  the Lie groupoid context, it is well-known that any left principal bundle is locally trivial  \cite[p.~165]{MoeMrc:LGSAC}. 
Thus, the study of principal bundles in this context can be done locally. 
In the Hopf algebroid framework, the notion of ``local triviality'' is not so clear. 
The perhaps right way to treat local triviality in this context might be to consider the site of all affine schemes over $\Spectre{\Bbbk}$ with a certain Grothendieck topology $\tau$, and say that a left principal bundle $(P,\alpha,\beta)$ is locally trivial if there is  a $\tau$-cover $\Spectre{B'} \to \Spectre{B}$ such that the pull-back bundle $P\tensor{B}B'$ is a trivial left principal $(\cH,\cK_{\scriptscriptstyle{B'}})$-bundle. However, as we will see below, when $\tau$ is the Zariski topology,  any locally trivial left principal bundle is also globally trivial. 
Also, the local triviality for the \emph{fpqc} (faithfully flat quasi-compact) 
topology is tautologically true since for any left principal bundle $(P,\alpha,\beta)$, the map $\beta: B \to P$ is by definition a faithfully flat extension.

Moreover, the naive approach to  local triviality  by localisation apparently does  not yield anything new:
let $(P,\alpha,\beta)$ be a left principal $(\cH,\cK)$-bundle. Denote by $\mathscr{Y}:=\Spectre{B}$ the underlying topological space of the locally ringed space associated to $B$, and by $\Omega(B)$ its subspace of maximal ideals.
Take a prime ideal $y \in \mathscr{Y}$ and consider the localisation $B_{\scriptscriptstyle{y}}$ at this point (the stalk) with $\tauup_{\scriptscriptstyle{y}}: B \to B_{\scriptscriptstyle{y}}$ as the canonical localisation algebra map. Using the notation  $\beta_{\scriptscriptstyle{y}}: B_{\scriptscriptstyle{y}} \to P_{\scriptscriptstyle{y}} := P\tensor{B}B_{\scriptscriptstyle{y}}$ and $\alpha_{\scriptscriptstyle{y}}: A \to P \to P_{\scriptscriptstyle{y}}$, we obtain the restricted left 
principal 
$(\cH,\cK_{\scriptscriptstyle{y}})$-bundle $(P_{\scriptscriptstyle{y}},\alpha_{\scriptscriptstyle{y}},\beta_{\scriptscriptstyle{y}})$ with respect to $\tauup_{\scriptscriptstyle{y}}$ as defined in Example \ref{exam:ResPB}.  
In this way, any left  principal $(\cH, \cK)$-bundle  $(P,\alpha, \beta)$ can be  restricted to a ``local principal bundle''  $(P_{\scriptscriptstyle{y}},\alpha_{\scriptscriptstyle{y}},\beta_{\scriptscriptstyle{y}})$ for every $y \in \mathscr{Y}$. One can say that $(P, \alpha,\beta)$ is locally trivial if and only if $(P_{\scriptscriptstyle{y}},\alpha_{\scriptscriptstyle{y}},\beta_{\scriptscriptstyle{y}})$ is trivial for every $y \in \mathscr{Y}$. Hence, by Proposition  \ref{prop:indPB}, this happens  if and only if  $\beta_{\scriptscriptstyle{y}}: B_{\scriptscriptstyle{y}} \to P_{\scriptscriptstyle{y}}$ splits as an  algebra map for every $y \in \mathscr{Y}$; if and only if  $\beta_{\scriptscriptstyle{\fk{m}}}: B_{\scriptscriptstyle{\fk{m}}} \to P_{\scriptscriptstyle{\fk{m}}}$ splits as an  algebra map for every $\fk{m} \in \Omega(B)$; if and only if $\beta: B \to P$ splits as an algebra map,  see \cite[p.~111{\it f}.]{Bou:AC12}.  In this sense, $P$ would be locally trivial if and only if it is globally so. 

In a different direction, assume that there exists for  any  $y \in \mathscr{Y}$ 
an element $f \notin y$ such that $\beta_{\scriptscriptstyle{f}}: B_{\scriptscriptstyle{f}} \to P_{\scriptscriptstyle{f}}$ splits as an algebra map, which by Proposition \ref{prop:indPB}  means  that the restricted left principal bundle $(P_{\scriptscriptstyle{f}},\alpha_{\scriptscriptstyle{f}},\beta_{\scriptscriptstyle{f}})$ is trivial on the  open neighbourhood  $\mathscr{Y}_{\scriptscriptstyle{f}}:=\Spectre{B_{\scriptscriptstyle{f}}}$  of $y$ in $\mathscr{Y}$: there is a section $\sigma_{\scriptscriptstyle{f}}: \mathscr{Y}_{\scriptscriptstyle{f}} \to \Spectre{P_{\scriptscriptstyle{f}}} \to \Spectre{P}$, that is, ${}^a\beta_{\scriptscriptstyle{f}} \circ \sigma_{\scriptscriptstyle{f}} =\id_{\mathscr{Y}_{\scriptscriptstyle{f}}}$, where ${}^a\beta_{\scriptscriptstyle{f}}: \Spectre{P_{\scriptscriptstyle{f}}} \to \mathscr{Y}_{\scriptscriptstyle{f}}$ is the associate continuous  map of $\beta_{\scriptscriptstyle{f}}: B_{\scriptscriptstyle{f}} \to P_{\scriptscriptstyle{f}}$. 
Again, 
one sees that a left bundle $(P,\alpha,\beta)$ with this assumption is in fact a (globally) trivial bundle. Indeed, take a maximal ideal $\fk{m} \in \Omega(B)$:  under the assumptions made, there is an $h \notin \fk{m}$ such that $\beta_{\scriptscriptstyle{h}}: B_{\scriptscriptstyle{h}} \to P_{\scriptscriptstyle{h}}$ splits as an algebra map; write $\sigma_{\scriptscriptstyle{h}}:  P_{\scriptscriptstyle{h}} \to  B_{\scriptscriptstyle{h}}$ for this splitting. Then one can easily check that  $$P_{\scriptscriptstyle{\fk{m}}}=P\tensor{B}B_{\scriptscriptstyle{\fk{m}}} \,\cong\, P\tensor{B}B_{\scriptscriptstyle{h}} \tensor{B_{\scriptscriptstyle{h}}}B_{\scriptscriptstyle{\fk{m}}}\,=\, \xymatrix@C=20pt{P_{\scriptscriptstyle{h}}\tensor{B_{\scriptscriptstyle{h}}} B_{\scriptscriptstyle{\fk{m}}} \ar@{->}^-{\sigma_{\scriptscriptstyle{h}}\tensor{B}B_{\scriptscriptstyle{h}}}[rr]& &  B_{\scriptscriptstyle{h}}\tensor{B_{\scriptscriptstyle{h}}}B_{\scriptscriptstyle{\fk{m}}}\, \cong \,  B_{\scriptscriptstyle{\fk{m}}} }
$$
is an algebra map which splits $\beta_{\scriptscriptstyle{\fk{m}}}$. Thus, $\beta_{\scriptscriptstyle{\fk{m}}}$ splits for every $\fk{m} \in \Omega(B)$, and so does $\beta$. Therefore, $(P,\alpha,\beta)$ is a trivial bundle.

Now assume that the topology $\tau$ is the Zariski one. Then, for a locally trivial left principal bundle $(P,\alpha,\beta)$ there exists an extension $B \to B'\,:=\, \prod_{\scriptstyle{1\leq i \leq n}} B_{\scriptscriptstyle{f_i}}$ for some set $\{f_i\}_{\scriptstyle{1\leq i \leq n}}$ of elements in $B$ such that $B=\sum_{\scriptstyle{1\leq i \leq n}} Bf_i$ and such that $P\tensor{B}B'$ is a trivial bundle. For any maximal ideal $\fk{m} \in \Omega(B)$, there must be some $f_j \notin \fk{m}$ for which the bundle $(P_{\scriptscriptstyle{f_j}},\alpha_{\scriptscriptstyle{f_j}},\beta_{\scriptscriptstyle{f_j}})$ is trivial. We then conclude, as above, that $(P,\alpha,\beta)$ must be also trivial.

On the other hand, it seems that the local triviality property of a given left principal $(\cH,\cK)$-bundle $(P,\alpha,\beta)$ is already contained in our condition of faithfully flatness of $\beta$.  More specifically, since $\beta$ is a flat extension, $\beta_{\scriptscriptstyle{y}}$ is also a flat extension for every $y \in \mathscr{Y}$. Therefore, also $B_{\scriptscriptstyle{y}} \to P_{\scriptscriptstyle{z}}$ is a flat extension
for every $y \in \mathscr{Y}$ and $ z \in ({}^{\scriptscriptstyle{a}}\beta)^{-1}(y)$,  
where ${}^{\scriptscriptstyle{a}}\beta: \Spectre{P} =: \mathscr{X} \to \Spectre{B}=\mathscr{Y}$ is the associated continuous  map of $\beta$. 
In other words, $\mathscr{Y}$ is flat over $\mathscr{X}$ \cite[p.~254]{Hartshorne}; hence, as mentioned in \cite[Def.\ 1.2]{Pfl:ADTATWQORM}, this appears to be  a good substitute for local triviality, see \cite[Sec.\ 3]{Palamodov} for a deeper discussion of this point.


\subsection{Natural comodule transformations}
In this subsection, we explore the Hopf algebroid analogue of natural transformations for groupoid-sets 
as in Lemma \ref{lema:NatIso}.

Let $(P, \ga, \gb)$ be a left principal $(\cH, \cK)$-bundle. As mentioned before,  one can define a functor 
$$
- \cotensor{\cH} P: \rcomod{\cH} \to \rcomod{\cK}
$$ 
since our Hopf algebroids are all assumed to be flat.
We will give some natural transformations involving  this functor, which will be useful in the sequel.

\begin{lemma}
\label{levissima}
One has the following natural transformations:
\begin{enumerate}
\item
for any right $\cH$-comodule $M$, the map
\begin{equation}
\label{giotto1}
\zeta_\emme: (M \cotensor{\cH} P) \otimes_\behhe P \to M \otimes_\ahha P, 
\quad (m \cotensor{\cH} p) \otimes_\behhe p' \mapsto m \otimes_\ahha pp'
\end{equation}
is an isomorphism of right $\cK$-comodules, where the coaction of the left hand side is the codiagonal one.  The inverse of $\zeta_\emme$ is given by
$$
\bar\zeta_\emme: m \otimes_\ahha p \mapsto (m_{\scriptscriptstyle{(0)}} \cotensor{\cH} m_{\scriptscriptstyle{(1)+}}) \otimes_\behhe m_{\scriptscriptstyle{(1)-}}p; 
$$
\item
for any right $\cH$-comodule $M$, the map
\begin{equation}
\label{giotto2}
\eta_\emme: M \to (M \cotensor{\cH} P) \cotensor{\cK} P^{\co},
\quad m \mapsto (m_{\scriptscriptstyle{(0)}} \cotensor{\cH} m_{\scriptscriptstyle{(1)+}}) \cotensor{\cK} m_{\scriptscriptstyle{(1)-}}
\end{equation}
defines a morphism of right $\cH$-comodules.
\end{enumerate}
\end{lemma}

\begin{proof}
To prove {\em (i)}, we proceed as follows:
that \rmref{giotto1} is a morphism of comodules follows from the fact that $P$ is a comodule algebra. Moreover, from \rmref{japan} one deduces that the inverse is well-defined and using the flatness of $P$ over $B$ along with \rmref{Eq:eps+-} and \rmref{Eq:p-+}, one checks that the given maps are mutually inverse: 
for example, 
$$
\bar\zeta_\emme \circ \zeta_\emme\big((m \cotensor{\cH}p) \otimes_\behhe p'\big)
= (m_{\scriptscriptstyle{(0)}} \cotensor{\cH} m_{\scriptscriptstyle{(1)+}}) \otimes_\behhe m_{\scriptscriptstyle{(1)-}}pp' 
=  (m \cotensor{\cH} p_{\scriptscriptstyle{(-1)+}}) \otimes_\behhe p_{\scriptscriptstyle{(-1)-}}p_{\scriptscriptstyle{(0)}} p'
\overset{\scriptscriptstyle{\rmref{Eq:p-+}}}{=} (m \cotensor{\cH}p) \otimes_\behhe p',
$$
where in the second step we used that $m \cotensor{\cH}p$ lies in $M \cotensor{\cH} P$.

As for {\em (ii)},
since $P$ is flat over $B$, the inclusion $(M \cotensor{\cH} P) \otimes_\behhe P \hookrightarrow M \otimes_\ahha P \otimes_\behhe P$ is the kernel of the map $M \otimes_\ahha P \otimes_\behhe P \to M \otimes_\ahha \cH \otimes_\ahha P \otimes_\behhe \cK \otimes_\behhe P$ given by
\begin{eqnarray*}
m \otimes_\ahha p \otimes_\behhe q &\longmapsto& m_{(0)} \otimes_\ahha m_{(1)} \otimes_\ahha p_{(0)} \otimes_\behhe p_{(1)} \otimes_\behhe q -  m_{(0)} \otimes_\ahha m_{(1)} \otimes_\ahha p \otimes_\behhe \mathscr{S}(q_{(1)}) \otimes_\behhe q_{(0)} \\
&& - m \otimes_\ahha p_{(-1)} \otimes_\ahha p_{(0)} \otimes_\behhe p_{(1)} \otimes_\behhe q 
+ m \otimes_\ahha p_{(-1)} \otimes_\ahha p_{(0)} \otimes_\behhe \mathscr{S}(q_{(1)}) \otimes_\behhe q_{(0)}.
\end{eqnarray*}
Composing this map with
$
M \to M \otimes_\ahha P \otimes_\behhe P, \ m \mapsto m_{(0)} \otimes_\ahha m_{(1)+} \otimes_\behhe m_{(1)-},
$
and applying \rmref{japan} shows that \rmref{giotto2} is well-defined on the given cotensor products; that it is 
also a morphism of comodules follows from \rmref{Eq:S+-}.
\end{proof}

\section{Principal bibundles versus weak equivalences}\label{sec:pbwe}

Parallel to Lemma \ref{lema:WE}, we will investigate in this subsection how weak equivalences arise from principal bundles. We first analyse the particular case of trivial bundles and then the general case. 

As recalled in Definition \ref{quellala1}, 
a morphism $\B{\phi}: (A,\cH) \to (B,\cK)$ of flat Hopf algebroids  is said to be a \emph{weak equivalence} if and only if the induced functor $\B{\phi}_{*}: \rcomod{\cH} \to \rcomod{\cK}$ of Eq.~\eqref{Eq:indfunct} establishes an equivalence of categories (which is, in fact, a monoidal symmetric equivalence). 

Let us consider the  trivial bundle $P=\cH\tensor{\scriptscriptstyle{\phi}}B$  associated to a given morphism $\B{\phi}$. 
One can easily check that the opposite bundle is $P^{\co} = B \tensor{\scriptscriptstyle{\phi}} \cH$ as defined in Remark \ref{rema:PB} {\em (ii)}. The associated functors are,  up to natural isomorphisms, 
$$
\B{\phi}_{*} \cong -\cotensor{\cH}P\quad \text{ and }\quad {}_{*}\B{\phi} \cong - \cotensor{\cK} P^{\co}.
$$ 
Moreover, as mentioned before, $- \cotensor{\cK} P^{\co}$ is a right adjoint to $-\cotensor{\cH}P$.

\subsection{The case of trivial principal bibundles}
\label{ssec:tpb-we}
Part of the following proposition was shown in \cite[Theorem 6.2]{HovStr:CALEHT} by using a different approach,
 see also \cite[Theorem D \& 5.5]{Hovey:02}. In Theorem \ref{aromanflower} below we give a more general result. 

\begin{proposition}
\label{lemma:trivial bibundle}
Let $\boldsymbol{\phi}=(\phi_{\scriptscriptstyle{0}}, \phi_{\scriptscriptstyle{1}}): (A,\cH) \to (B,\cK)$ be a morphism of flat Hopf algebroids, and consider the associated trivial bundle $P=\cH\tensor{\phi}B$.  The following are equivalent:
\begin{enumerate}
\item $P$ is a  principal $(\cH, \cK)$-bibundle.
\item The canonical morphism
$$
\Phi: B\tensor{A}\cH\tensor{A}B \to \cK, \quad b\tensor{A}u\tensor{A}b' \mapsto \Sf{s}(b) \phi_{\scriptscriptstyle{1}}(u)\Sf{t}(b')
$$ 
 of Hopf $B$-algebroids 
is an isomorphism, and $\alpha$ is a faithfully flat extension.
\item The morphism $\boldsymbol{\phi}$ is a weak equivalence.
\end{enumerate}
\end{proposition}

\begin{proof}
To prove $(i) \Rightarrow (ii)$, we only need to check that $\Phi$ is bijective. 
By assumption, $\can{P}{\cK}$ is bijective, and denote the translation map here as
$$
\tau: \cK \to P\tensor{A}P, \quad k \mapsto (u^k\tensor{\phi}b^k)\tensor{A}(v^k\tensor{\phi}c^k), 
$$ 
which means that for every $k \in \cK$
$$
1_\pehhe \tensor{B}k \, = \, (1_{\scriptscriptstyle{\cH}}\tensor{\phi}1_\behhe)\tensor{A}k\,=\, \big(u^kv^k_{(1)}\tensor{A}b^k\big)\tensor{B}\phi_{\scriptscriptstyle{1}}(v^k_{(2)})\Sf{t}(c^k), 
$$ 
Applying the counit of $\cH$ we obtain 
$$
k \,=\,  \Sf{s}(b^k)\phi_{\scriptscriptstyle{1}}\big(\Sf{s}(\varepsilon(u^k))v^k\big)\Sf{t}(c^k).
$$
Define now the map 
$$  
\Lambda: \cK \to B \tensor{A} \cH \tensor{A} B, \quad k \mapsto \phi_{\scriptscriptstyle{0}}(\varepsilon(u^k))b^k \tensor{A}v^k\tensor{A}c^k.
$$ 
Using the previous equality, we easily get that $\Phi\circ \Lambda = \id$. In the opposite direction, we have 
$$
\Lambda \circ \Phi(b\tensor{A}u\tensor{A}b') \,=\,  b\tensor{A}u\tensor{A}b'
$$ 
since $k=\Sf{s}(b)\phi_{\scriptscriptstyle{1}}(u)\Sf{t}(b')$ is uniquely determined by the equation 
$$ 
1_\pehhe \tensor{B}k\,=\, \big(1_{\scriptscriptstyle{\cH}}\tensor{\phi}b\big) \tensor{B} \phi_{\scriptscriptstyle{1}}(u)\Sf{t}(b').
$$

In order to prove $(ii) \Rightarrow (iii)$, we already know by definition that $\B{\phi}_{*}=-\cotensor{\cH}P$ is a symmetric monoidal functor.  We need to establish natural isomorphisms 
\begin{equation}\label{Eq:natisos}
(-\cotensor{\cH}P) \circ(-\cotensor{\cK}P^{\co}) \cong \id_{\rcomod{\cK}}, \quad (-\cotensor{\cK}P^{\co}) \circ (-\cotensor{\cH}P) \cong \id_{\rcomod{\cH}}.
\end{equation}
First recall that we have a commutative diagram 
$$
\xymatrix@C=50pt@R=15pt{ 0 \ar@{->}^-{}[r] & \ar@{->}^-{\cong}[d]  P^{\co} \cotensor{\cH} P  \ar@{->}^-{}[r] & P^{\co} \tensor{A} P  \ar@{=}^-{}[d] \\    0 \ar@{->}^-{}[r] &  B\tensor{A}\cH\tensor{A}B  \ar@{->}^-{B\tensor{A}\Delta\tensor{A}B}[r]   &  B\tensor{A}\cH\tensor{A}\cH\tensor{A}B. }
$$ 
Hence, the canonical injection $ P^{\co} \cotensor{\cH} P  \hookrightarrow P^{\co} \tensor{A} P$ splits in the category of $B$-bimodules.  
For a right  $\cK$-comodule $N$, we then have a chain of  isomorphisms 
$$
\big(N\cotensor{\cK} P^{\co} \big) \cotensor{\cH}P \,\cong \,\, N\cotensor{\cK}\big(P^{\co} \cotensor{\cH}P\big)  \,\, \cong \,\, N \cotensor{\cK}\cK \,\, \cong \,\, N
$$ 
of right $\cK$-comodules, 
where we used the fact that $\Phi$ is an isomorphism of $\cK$-bicomodules. 
Clearly, the resulting  isomorphism is natural and this gives the first natural isomorphism in \eqref{Eq:natisos}. To establish the second one, we will use the faithfully flatness of $P_\ahha$, that is, of $\alpha$. 
For a right $\cH$-comodule $M$ define by means of Eq.~\eqref{giotto2} 
the following morphism 
$$
\theta_{M}: M \to \big( M\cotensor{\cH}P\big)\cotensor{\cK}P^{\co}, \quad m \mapsto \big(m_{\scriptscriptstyle{(0)}}\cotensor{\cH}(m_{\scriptscriptstyle{(1)}}\tensor{\phi}1_{\scriptscriptstyle{B}}) \big) \cotensor{\cK}  (1_{\scriptscriptstyle{B}} \tensor{\phi} \mathscr{S}(m_{\scriptscriptstyle{(2)}}))
$$
of right $\cH$-comodules.
Using the natural isomorphisms $\zeta$ of \eqref{giotto1}, one can show that  $\theta_{M}\tensor{A}P$ is an isomorphism, and hence that $\theta$ is a natural isomorphism. Therefore, $\B{\phi}_{*}$ is an equivalence of categories. 

The step $(iii) \Rightarrow (i)$ is seen as follows: by Example \ref{exam:HopfMorph}, $P$ is a left principal $(\cH, \cK)$-bundle. To check that $P$ is also a right principal $(\cH, \cK)$-bundle, we  need to verify that the canonical map $\can{P}{\cK}$ of Eq.~\eqref{Eq:canK} is bijective as well as  that $\alpha$ is a faithfully flat extension.
Since $\boldsymbol{\phi}_{*}$ is an equivalence of categories, there is a natural isomorphism 
$$
-\tensor{A} \boldsymbol{\phi}_{*}(\cH) \,\cong \, \boldsymbol{\phi}_{*}\circ ( -\tensor{A}\cH),
$$
where $-\tensor{A}\cH: \rcomod{\cH} \to \rcomod{\cH}$ is the composition of the forgetful functor with the functor defined as in \eqref{Eq:RM}, and where $P=\boldsymbol{\phi}_{*}(\cH)$ is an $A$-module via the algebra map 
$\alpha: A \to P, \ a \mapsto \Sf{s}(a)\tensor{A}1_{\scriptscriptstyle{B}}$. Hence, such a natural isomorphism directly implies 
that  $\alpha$ is a faithfully flat extension.

Let us then prove that $\can{P}{\cK}$ is bijective. Since the counit of the adjunction $\phi_{*} \dashv \big(-\cotensor{\cK}\,{}_{*}\phi(\cH)\big)$ is a natural isomorphism (see \S\ref{ssec:bicomod}), we denote by  
$$
\B{\xiup}_{\scriptscriptstyle{\cK}}: \cK \to B\tensor{A}\cH\tensor{A}B, \quad k  \mapsto b^k\tensor{A}u^k\tensor{A}c^k
$$  
its inverse at $\cK$, with the help of which we can write
$$
k=\Sf{s}(b^k) \phi_{\scriptscriptstyle{1}}(u^k)\Sf{t}(c^k)
$$ 
for every $k \in \cK$. Define moreover 
$$
\Psi: P\tensor{B}\cK \cong \cH\tensor{A}\cK \to P\tensor{A}P, \quad u\tensor{A}k \mapsto \big(u\mathscr{S}(v^k_{(1)})\tensor{\phi}b^k\big)\tensor{A}\big(v^k_{(2)}\tensor{\phi}c^k\big)
$$ 
and compute
\begin{eqnarray*}
\Psi\circ \can{P}{\cK}\big( (u\tensor{A}b)\tensor{A}(v\tensor{A}b')\big) 
&=& \Psi\big((uv_{(1)}\tensor{A}b)\tensor{B} \phi_{\scriptscriptstyle{1}}(v_{(2)})\Sf{t}(b')\big)\\
&=& \Psi\big(uv_{(1)}\tensor{A}\mathsf{s}(b) \phi_{\scriptscriptstyle{1}}(v_{(2)})\mathsf{t}(b')\big)\\ 
&=& \Psi\big(uv_{(1)}\tensor{A}\Phi\big(b \tensor{A}v_{(2)}\tensor{A}b'\big)\big) \\ 
&=& uv_{(1)}\mathscr{S}(v_{(2)})\tensor{A}b \tensor{A}v_{(3)}\tensor{A}b' \\ 
&=& \big(u\Sf{s}(\varepsilon(v_{(1)}))\tensor{A}b\big) \tensor{A}\big(v_{(2)}\tensor{A}b'\big) \\ 
&=& \big(u\tensor{A}b\big) \tensor{A}\big(v\tensor{A}b'\big),
\end{eqnarray*}
which shows that $\Psi \circ \can{P}{\cK}=\id$. The opposite direction is verified as follows:
\begin{eqnarray*}
\can{P}{\cK}\circ \Psi(u\tensor{A}k) &=& \can{P}{\cK}\Big(\big(u\mathscr{S}(v^k_{(1)})\tensor{A}b^k\big)\tensor{A}\big( v^k_{(2)}\tensor{A}c^k\big)\Big)\\
&=& \big(u\mathscr{S}(v^k_{(1)})v^k_{(2)}\tensor{A}b^k\big)\tensor{B}\big(\phi_{\scriptscriptstyle{1}}(v^k_{(3)})\Sf{t}(c^k)\big) \\ 
&=& \big(u\tensor{A}\phi_{\scriptscriptstyle{0}}(\varepsilon(v^k_{(1)})b^k\big)\tensor{B}\big( \phi_{\scriptscriptstyle{1}}(v^k_{(2)})\Sf{t}(c^k)\big) \\ 
&=& u\tensor{A}\big(\Sf{s}\big(\phi_{\scriptscriptstyle{0}}(\varepsilon(v^k_{(1)})b^k\big) \phi_{\scriptscriptstyle{1}}(v^k_{(2)})\Sf{t}(c^k)\big)  \\ 
&=&  u\tensor{A}\big(\Sf{s}(b^k) \phi_{\scriptscriptstyle{1}}(v^k)\Sf{t}(c^k)\big) \\
& =& u\tensor{A}k,
\end{eqnarray*} 
which gives the desired equality.
\end{proof}

\begin{rem}
\label{rem:HS}
The statement that $\alpha$ is a flat extension is equivalent to saying that $B$ is Landweber exact over $(A,\cH)$ in the sense of \cite[Def.\ 2.1]{HovStr:CALEHT}, see Lemma 2.2 in {\em op.~cit}. This, as mentioned before, implies in particular that $(B,B\tensor{A}\cH\tensor{A}B)$ is a flat Hopf algebroid. 
\end{rem}

\subsection{The case of general principal bibundles}
\label{ssec:lpb-we}

Let now $(P, \alpha, \beta)$ be an $(\cH,\cK)$-bicomodule algebra. Consider the two-sided translation  Hopf algebroid $(P,\cH\lJoin P\rJoin \cK)$ as in Lemma \ref{lemma:bisemidirect}. 
Recall that the tensor product $\cH\tensor{A}P\tensor{B}\cK$ is defined by using the module structures  $\shopf{\cH}$, $\due P \ahha \behhe$, and $\shopf{\cK}$, and also that there is a diagram of Hopf algebroids 
 $$
\xymatrix@R=15pt{ 
 & (P ,\cH\lJoin P\rJoin \cK)& \\ \ar@{->}^-{\B{\alpha}}[ru] (A,\cH)  & & (B,\cK) \ar@{->}_-{\B{\beta}}[ul]  
}
$$ 
where $\B{\beta}$ and $\B{\alpha}$ are the maps as in Lemma \ref{lemma:bisemidirect}.
On the other hand, one can consider the extended Hopf algebroids 
$(P, P\tensor{A}\cH\tensor{A}P)$ and $(P, P\tensor{A}\cK\tensor{A}P)$, together with the morphisms of Hopf algebroids:
\begin{small}
\begin{equation}
\label{Eq:MSB}
P\tensor{B}\cK\tensor{B}P \to \cH\lJoin P\rJoin \cK,\;\; p'\tensor{B}w\tensor{B}p \mapsto \Sf{s}(p')\beta_{{1}}(w) \Sf{t}(p) = \mathscr{S}\big(p_{{(-1)}} \big) \tensor{A}p_{{(0)}}p'\tensor{B}p_{{(1)}}w,
\end{equation}
\begin{equation}
\label{Eq:LSB}
P\tensor{A}\cH\tensor{A}P \to \cH\lJoin P\rJoin \cK,\;\; p'\tensor{A}u\tensor{A}p \mapsto \Sf{s}(p')\alpha_{{1}}(u) \Sf{t}(p) = u\mathscr{S}\big(p_{{(-1)}} \big) \tensor{A}p_{{(0)}}p'\tensor{B}p_{{(1)}},
\end{equation}
\end{small}
where $\sf{s}$ and $\sf{t}$ are the source and the target maps of $\cH\lJoin P\rJoin \cK$ as given in Lemma \ref{lemma:bisemidirect}.

The following proposition shows that principal bundles lead to  weak equivalences.
\begin{prop}
\label{prop:Cielitomissyoutoo}
We have the following implications:
\begin{enumerate}
\item If $(P,\alpha,\beta)$ is a left principal $(\cH,\cK)$-bundle, then $\B{\beta}$ is a weak equivalence.
\item If $(P,\alpha,\beta)$ is a right principal $(\cH,\cK)$-bundle, then $\B{\alpha}$ is a weak equivalence.
\item If $(P,\alpha,\beta)$ is a principal $(\cH,\cK)$-bibundle, then $\B{\beta}$ and $\B{\alpha}$ are weak equivalences. In this case, $(A, \cH)$ and $(B,\cK)$ are weakly equivalent, see Definition \ref{quellala1}.
\end{enumerate} 
\end{prop}
\begin{proof}
Part {\em (iii)} is clearly derived from {\em (i)} and {\em (ii)}. We only prove {\em (i)} since 
{\em (ii)} is obtained {\em mutatis mutandum}. 
Using Proposition \ref{lemma:trivial bibundle}, we need 
to check that the map $B \to \cK\tensor{B}P$ is faithfully flat, which is clear from the assumptions, and that the map in 
Eq.~\eqref{Eq:MSB} is bijective. Denote this map by $\tilde{\boldsymbol{\gb}}$
and by 
$\tilde{\boldsymbol{\gb}}'$ 
what is going to be its inverse,
given by
$$
\tilde{\boldsymbol{\gb}}': \cH\lJoin P\rJoin \cK \to P\tensor{B}\cK\tensor{B}P, 
\quad 
u\tensor{A}p\tensor{B}w \mapsto pu_{{+}} \tensor{B} \mathscr{S}(u_{{-(1)}}) w\tensor{B}u_{{-(0)}}.
$$
We compute from one hand 
\begin{eqnarray*}
\td{\B{\beta}} \circ \tilde{\boldsymbol{\gb}}' (u\tensor{A}p\tensor{B}w) 
&=& \td{\B{\beta}}( pu_{{+}} \tensor{B} \mathscr{S}(u_{{-(1)}}) w\tensor{B}u_{{-(0)}}) \\ 
&=&  \mathscr{S}(u_{{-(-1)}}) \tensor{A} u_{{-(0)}} u_{{+}}p \tensor{B} u_{{-(1)}} \mathscr{S}(u_{{-(2)}}) w \\ 
&=& \mathscr{S}(u_{{-(-1)}}) \tensor{A} u_{{-(0)}} u_{{+}}p \tensor{B}  w \\ 
&\overset{\eqref{Eq:Su}}{=}&  u\tensor{A}p\tensor{B}w.
\end{eqnarray*}
From the other hand, to check that also $\td{\B{\beta}}' \! \circ \td{\B{\beta}} = \id$, 
we first deduce from Eq.~\rmref{Eq:p-+}
\begin{equation}
\label{waitingagain}
p_{(0)} \otimes_\behhe p_{(1)} \otimes_\behhe p_{(2)} \otimes_\behhe 1_\behhe = 
p_{(-1)+(0)} \otimes_\behhe p_{(-1)+(1)} \otimes_\behhe p_{(1)} \otimes_\behhe p_{(-1)-}  p_{(0)},  
\end{equation}
which we use to see that
\begin{eqnarray*}
 \td{\B{\beta}}' \! \circ \td{\B{\beta}} ( p'\tensor{B}w\tensor{B}p) 
&=&  \td{\B{\beta}}'(\mathscr{S}(p_{{(-1)}}) \tensor{B} p_{{(0)}}p' \tensor{B}p_{{(1)}} w) \\ 
&\overset{\eqref{zucchero}}{=}& 
p_{(0)} p_{(-1)-} p' \tensor{B} \mathscr{S}(p_{(-1)+(1)}) p_{(1)} w \tensor{B} p_{(-1)+(0)} \\ 
&\overset{\eqref{waitingagain}}{=}& 
p' \tensor{B} \mathsf{t}(\gve(p_{(1)})) w \tensor{B} p_{(0)} \\
&\overset{\eqref{maxxi}}{=}& 
p'\tensor{B}w\tensor{B}p,
\end{eqnarray*}
and this concludes the proof.
\end{proof}

\begin{corollary}
\label{coro:mpb}
Let $\fk{f}:(P,\alpha,\beta) \to (P',\alpha',\beta')$ be a morphism in $\mathsf{PB}^\ell(\cH,\cK)$. Then the associated morphism 
$$
(\fk{f}, \cH\tensor{A}\fk{f}\tensor{B}\cK): (P,\cH\lJoin P\rJoin \cK) \to (P',\cH\lJoin P'\rJoin \cK)
$$ 
between the two-sided translation  Hopf algebroids (see \S\ref{ssec:bicomodalg})
is an isomorphism of Hopf algebroids and therefore a  weak equivalence. 
\end{corollary}
\begin{proof}
This directly follows from Lemma  \ref{lemma:ISO}. That this morphism is a weak equivalence can also be deduced from Proposition \ref{prop:Cielitomissyoutoo} {\em (i)} and the commutative diagram \eqref{Eq:qseramana}.
\end{proof}

\begin{rem}
As mentioned in \S\ref{ssec:ltpb}, in the Lie groupoid context it is well-known that any morphism between principal bundles is an isomorphism \cite[p.~165]{MoeMrc:LGSAC}, and hence induces  an isomorphism between the associated two-sided translation  groupoids. 
Corollary \ref{coro:mpb} states an analogous result for the associated two-sided Hopf algebroids attached to flat Hopf algebroids.  
As a consequence, any two-stage zigzag of weak equivalences, as described in the isosceles triangle in the Introduction, is unique up to an isomorphism.
\end{rem}

\section{The bicategory of principal bundles as a universal solution}
\label{sec:bicat}
In this section, we introduce the cotensor product of two  principal bundles in the Hopf algebroid context, which is the analogue  of the tensor product of  principal bundles in the framework of Lie groupoids \cite[p.~166]{MoeMrc:LGSAC}, where it is defined as the orbit space of the fibred product of the underlying bundles, see also Remark \ref{rem:tensorG} for abstract groupoids. In the case of Hopf algebroids, the cotensor product leads to the orbit space (which is the coinvariant subalgebra as mentioned in \S\ref{ssec:pcomalg}) of the tensor product of the underlying comodule algebras.
With this product,  principal bundles can be shown to form a bicategory. It turns out that  trivial bundles constitute a $2$-functor from the canonical $2$-category of flat Hopf algebroids to this bicategory, which yields a certain universal solution (or a calculus of fractions with respect to weak equivalences).  
\vspace{-0,3cm}

\subsection{The cotensor product of  principal bundles}
\label{ssec:ppb}

Consider  three flat Hopf algebroids $(A, \cH)$, $(B, \cK)$, and $(C, \cJ)$, and let 
$(P, \alpha, \beta)$ be a left principal $(\cH, \cK)$-bundle and $(Q, \sigma, \theta)$ a left   
principal $(\cK,\cJ)$-bundle.
Recall from \rmref{Eq:bicotensor} that $P\cotensor{\cK}Q$ carries the structure of an $(\cH, \cJ)$-bicomodule. 
Moreover, it is clear from the definition of a comodule algebra  
that this is simultaneously  an $A$-algebra and $C$-algebra via the following commutative diagram 
\begin{equation}\label{Eq:earthquake}
\xymatrix@R=15pt@C=30pt{ A \ar@{->}^-{\alpha}[r]  \ar@/_2pc/@{-->}_-{\td{\alpha}}[rrd] & P \ar@{->}^-{- \tensor{B} 1_ {\scriptscriptstyle{Q}}}[rr] & &  P\tensor{B}Q  \\   & & P\cotensor{\cK}Q \ar@{->}^-{}[ru] &  \\  & 0 \ar@{->}^-{}[ru] &   & Q \ar@{->}_-{1_ {\scriptscriptstyle{P}} \tensor{B} -}[uu]   \\    &  & & C.  \ar@/^2pc/@{-->}^-{\td{\theta}}[luu] \ar@{->}_-{\theta}[u]  }
\end{equation}
This structure converts the triple $(P\cotensor{\cK}Q,\td{\alpha}, \td{\theta})$ into an $(\cH,\cJ)$-bicomodule algebra. In the subsequent lemma we show that this gives in particular a left principal bundle:

\begin{lemma}
\label{lemma:cotensorpb}\noindent
\begin{enumerate}
\item The correspondence 
\begin{eqnarray*}   
\mathsf{PB}^\ell(\cH,\cK) \times  \mathsf{PB}^\ell(\cK,\cJ) & \longrightarrow & \mathsf{PB}^\ell(\cH,\cJ), \\ 
\big( (P, \alpha, \beta), (Q, \sigma, \theta) \big) & \longmapsto &   (P\cotensor{\cK}Q,\td{\alpha}, \td{\theta}), \\ 
(F,G)  & \longmapsto & F \cotensor{\cK} G 
\end{eqnarray*}
gives a well-defined functor.
\item
The canonical algebra extension $P\cotensor{\cK}Q \hookrightarrow P\tensor{B}Q$ is faithfully flat.
\end{enumerate}
\end{lemma}
\begin{proof}
Part {\em (i)}:
as we have seen before, the obvious algebra map $\theta': C \to Q \to P\tensor{B}Q$ factors through 
\begin{equation*}
\label{Eq:triangle}
\xymatrix{ C \ar@{->}^-{\td{\theta}}[r] \ar@{->}_-{{\theta'}}[rd] &  P\cotensor{\cK}Q \ar@{_{(}->}[d] \\ & P\tensor{B}Q,}
\end{equation*} 
and  $\theta'$ is a faithfully flat extension since $\beta$ and $\theta$ are so. The faithfully flatness of the map $\td{\theta}: C \to P \cotensor{\cK} Q$ is seen as follows: one has a chain of $C$-module isomorphisms
\begin{equation}
\label{Eq:Paciencia}
\xymatrix@C=40pt@R=0pt{(P\cotensor{\cK}Q)\tensor{C}Q \ar@{->}^-{\cong}[r] &  P\cotensor{\cK} (Q\tensor{C}Q) \ar@{->}_-{\cong}^-{P\cotensor{\cK}\Sf{can}}[r] &  
P\cotensor{\cK}(\cK\tensor{B}Q) \ar@{->}^-{\cong}[r] & P\tensor{B}Q \\ 
(p\cotensor{\cK}q)\tensor{C}q' \ar@{|->}[r] & p\cotensor{\cK}(q\tensor{C}q')\ar@{|->}[r] &  p\cotensor{\cK}(q_{\scriptscriptstyle{(-1)}}\tensor{B}q_{\scriptscriptstyle{(0)}}q') \ar@{|->}[r] & p\tensor{B} qq', }
\end{equation}
hence $ (P \cotensor{\cK} Q)\otimes_\cehhe Q $ is also faithfully flat over $C$, and since by assumption $Q$ is so over $C$, we deduce that 
$P \cotensor{\cK} Q$ is faithfully flat over $C$.
For better distinction, let us denote the involved translation maps as
$$ 
\tauup_\pehhe: \cH \to P\tensor{B}P,\quad u \mapsto u_{+}\tensor{B}u_{-},
\qquad \tauup_\quhhu: \cK \to Q\tensor{C}Q, \quad w \mapsto w_{[+]} \tensor{C} w_{[-]}.
$$
The canonical map that turns the cotensor product into a bundle is given as
$$
\mathsf{can}: 
(P\cotensor{\cK}Q ) \tensor{C} (P\cotensor{\cK}Q) \to \cH\tensor{A} (P\cotensor{\cK}Q),\quad  (p\cotensor{\cK}q) \tensor{C}(p'\cotensor{\cK}q') \mapsto p_{(-1)}\tensor{A} (p_{(0)}p'\cotensor{\cK} qq'), 
$$
and what is going to be its inverse is defined by 
$$
\tilde{\mathsf{can}}: 
\cH\tensor{A} ( P\cotensor{\cK}Q ) \to   (P\cotensor{\cK}Q ) \tensor{C} (P\cotensor{\cK}Q),
\quad u\tensor{A}(p\cotensor{\cK} q)  \mapsto 
(u_{+(0)} \cotensor{\cK} u_{+ (1)[+]} ) \tensor{C} (p u_{-} \cotensor{\cK} q u_{+ (1)[-]} ),
$$
which are well-defined maps by the $A$-linearity of the coaction as well as using \rmref{bolognacentrale}.
We then compute
\begin{equation*}
\begin{split}
(\tilde{\mathsf{can}} \circ \mathsf{can})\big( (p\cotensor{\cK}q) \tensor{C}(p'\cotensor{\cK}q')\big) &=  \tilde{\mathsf{can}}\big(p_{(-1)}\tensor{A} (p_{(0)}p'\cotensor{\cK} qq' ) \big) \\
&=  (p_{(-1) + (0)} \cotensor{\cK} p_{(-1) + (1)[+]} ) \tensor{C} (p_{(0)} p_{(-1)-}  p' \cotensor{\cK} qq' p_{(-1) + (1)[-]} )   \\
&\overset{\eqref{Eq:p-+} }{=} 
(p_{(0)}\cotensor{\cK} p_{(1)[+]} ) \tensor{C} (p' \cotensor{\cK} qq' p_{(1)[-]}) \\
&= (p\cotensor{\cK} q_{(-1)[+]} ) \tensor{C} ( p' \cotensor{\cK} q' q_{(0)} q_{(-1)[-]} )  \\
&\overset{\eqref{Eq:p-+}}{=}  (p\cotensor{\cK} q) \otimes_\cehhe (p' \cotensor{\cK} q'), \\
\end{split}
\end{equation*}
where we used the definition of the cotensor product in the fourth step. The opposite verification is left to the reader. To prove part {\em (ii)}, 
consider the isomorphism of Eq.~\eqref{Eq:Paciencia}. It is clear that this is an isomorphism of left $P\cotensor{\cK}Q$-modules; since $Q$ is a faithfully flat $C$-module, $(P\cotensor{\cK}Q)\tensor{C}Q \cong P\tensor{B}Q$ is a faithfully flat $P\cotensor{\cK}Q$-module as well.
%
\end{proof}

\begin{rem}
\label{jumpingwindow}
Of course, the construction of the functor 
in Lemma \ref{lemma:cotensorpb} can be adapted
{\em mutatis mutandum} for right principal bundles as well as for principal bibundles.
\end{rem}

An example of the cotensor product construction above arises from the following proposition,

\begin{proposition}\label{prop:1out2}
Let $(A,\cH)$ and $(C_i,\cJ_i)$, $i=1,2$, be  flat Hopf algebroids. Then any diagram of weak equivalences
\begin{small}
$$
\xymatrix@R=15pt{ (C_1,\cJ_1) & & (C_2,\cJ_2)  \\ & \ar@{->}^-{\B{\theta}_1}[lu] (A,\cH) \ar@{->}_-{\B{\theta}_2}[ru] &}
$$
\end{small}
can be completed to the following diagram 
\begin{small}
\begin{equation}
\label{kontoauszug}
\xymatrix@R=15pt{ & \big(P_1^{\co}\cotensor{\cH}P_2, \cJ_1\lJoin \big(P_1^{\co}\cotensor{\cH}P_2\big) \rJoin \cJ_2\big) & \\ (C_1,\cJ_1) \ar@{->}^-{\B{\zeta}_1}[ru] & & (C_2,\cJ_2), \ar@{->}_-{\B{\zeta}_2}[lu]  \\ & \ar@{->}^-{\B{\theta}_1}[lu] (A,\cH) \ar@{->}_-{\B{\theta}_2}[ru] &}
\end{equation}
\end{small}
of weak equivalences,
where $P_i=\cH\tensor{\theta_i}C_i$, $i=1,2$, are the respective associated trivial bundles. 
\end{proposition}
\begin{proof}
Since $\theta_i$ is a weak equivalence, 
$P_i$ is a  principal $(\cH,\cJ_i)$-bibundle by Proposition \ref{lemma:trivial bibundle}. 
Therefore, by Lemma \ref{lemma:cotensorpb} (and its right hand side version, see Remark \ref{jumpingwindow}),   
the cotensor product $P^{\co}_1\cotensor{\cH}P_2$ is a principal $(\cJ_1,\cJ_2)$-bibundle as well and the proof is completed using Proposition \ref{prop:Cielitomissyoutoo} {\em (iii)}.
\end{proof}

\begin{example}\label{exm:HS}
A particular situation of  Proposition \ref{prop:1out2} is the one considered in Example \ref{exm:bhb}:  let  $\phi: B \rightarrow  A \leftarrow B': \psi$  be a diagram of commutative algebras. Assume that $\alpha:A \to P:=\cH\tensor{\phi}B$, $a \mapsto \Sf{s}(a)\tensor{A}1_{\scriptscriptstyle{B}}$, and  $\alpha':A \to P':=\cH\tensor{\psi}B'$ are faithfully flat extensions.  This, in particular, means that $B$ and $B'$ are Landweber exact. 
Consider the algebra $C:=B\tensor{A}\cH\tensor{A}B'$ along with the scalar extension Hopf algebroids $(C,\cH_{\scriptscriptstyle{\phiup}})$ and  $(C,\cH_{\scriptscriptstyle{\psiup}})$, where  $\phiup: A \rightarrow C \leftarrow A: \psiup$ are the obvious maps constructed from $\phi$ resp.\ $\psi$ as in Example \ref{exm:bhb}. Now 
$(B,\cH_{\scriptscriptstyle{\phi}}) \overset{\scriptscriptstyle{\B{\alpha}}}{\longleftarrow}  (A,\cH)  \overset{\scriptscriptstyle{\B{\alpha'}}}{\longrightarrow} (B',\cH_{\scriptscriptstyle{\psi}})$ is a diagram of weak equivalences by Proposition \ref{lemma:trivial bibundle}. Applying 
Proposition \ref{prop:1out2}, we get a diagram 
$$
(B,\cH_{\scriptscriptstyle{\phi}}) \longrightarrow  (C,\cH{\scriptscriptstyle{\phiup}}) \,\cong\, (C,\cH{\scriptscriptstyle{\phi}} \lJoin C \rJoin \cH{\scriptscriptstyle{\psi}})  \,\cong\,  (C,\cH{\scriptscriptstyle{\phiup}}) \longleftarrow (B',\cH_{\scriptscriptstyle{\psi}})\
$$
of weak equivalences, where the middle isomorphisms are as in Example  \ref{exm:bhb}.
This, in fact,  is part of the proof given in \cite[Theorem 6.5]{HovStr:CALEHT}.
\end{example}

\subsection{The bicategory of principal bundles}
\label{ssec:bicat}
In particular, the constructions in the preceding subsection allow for the main observation in this section:
\begin{prop}
\label{piazzadellorologio}
The data given by
\begin{center}
\begin{enumerate}[\quad \raisebox{1pt}{$\scriptstyle{\bullet}$}]
\item
flat Hopf algebroids (as $0$-cells),
\item
left principal bundles (as $1$-cells), 
\item
as well as morphisms of left principal bundles (as $2$-cells)
\end{enumerate}
\end{center}
define a bicategory.
\end{prop}
\begin{proof}
The unit $0$-cells in this bicategory are the unit bundles of the form $\mathscr{U}(\cH)$ as in Example \ref{Exam:H}.
The multiplication of two principal bundles ({\em i.e.}, their cotensor product) and of their morphisms is given as in Lemma \ref{lemma:cotensorpb}. 
The associativity of the cotensor product is not obvious in this case as it does not follow directly from the flatness of the involved Hopf algebroids: let $(A, \cH)$, $(B, \cK)$, $(C, \cJ)$, and $(D, \cI)$ be flat Hopf algebroids, as well as $(P, \ga, \gb)$, $(Q, \gs, \theta)$,  and $(S, \gamma, \gd)$ be left principal 
$(\cH, \cK)$-, $(\cK, \cJ)$-, resp.\ $(\cJ, \cI)$-bundles. First of all, 
we have the following diagram
\begin{small}
$$
\xymatrix@R=10pt{ P\cotensor{\cK} (Q \cotensor{\cJ} S)   \ar@{^{(}->}^-{}[r] &  P\tensor{B} (Q \cotensor{\cJ} S) \; \ar@{^{(}->}^-{}[rd]  & \\ &  &   P\tensor{B}Q\tensor{C}S, \\ 
(P \cotensor{\cK} Q) \cotensor{\cJ} S \ar@{^{(}->}^-{}[r]  &   (P \cotensor{\cK} Q) \tensor{C} S\;  \ar@{^{(}->}^-{}[ru]  &
}
$$
\end{small}
where the upper injections result from definitions and the flatness of $P$ over $B$. The second map of the lower injections 
follows from the fact that, as in Lemma \ref{lemma:cotensorpb} {\em (ii)}, the injection $P \cotensor{\cK} Q \hookrightarrow P \otimes_\behhe Q$ is faithfully flat. Using the universal property of kernels, we deduce the desired natural  isomorphism
$$
(P \cotensor{\cK} Q) \cotensor{\cJ} S \overset{\simeq}{\lra} P \cotensor{\cK} (Q \cotensor{\cJ} S).
$$
The remaining axioms to be verified in a bicategory are left to the reader.
\end{proof}

We denote this bicategory by ${\mathsf{PB}}^\ell$  and refer to it as the \emph{bicategory of (left) principal bundles}. The category of $1$- and $2$-cells from $(A, \cH)$ to $(B, \cK)$  then is the category $\lPB{\cH}{\cK}$, see \S\ref{ssec:gpb}. 

Similarly, we can introduce the \emph{bicategory of right principal bundles} $\mathsf{PB}^{\scriptscriptstyle{r}}$ and also the  \emph{bicategory of principal bibundles} $\mathsf{PB}^{\scriptscriptstyle{b}}$
as mentioned in Remark \ref{jumpingwindow}. On the other hand,
by Remark \ref{rema:PB} {\em (ii)} there is an isomorphism 
${\mathsf{PB}}^\ell \cong ({\mathsf{PB}}^{\scriptscriptstyle{r}})^\op$ of bicategories, using B\'enabou's terminology \cite[\S3]{Benabou}: for a bicategory $\mathscr{B}$, denote by $\mathscr{B}^\op$ its \emph{transpose bicategory}, obtained from $\mathscr{B}$ by reversing $1$-cells. On the other hand, its \emph{conjugate bicategory} $\mathscr{B}^{\co}$ is obtained by reversing $2$-cells. We will call a  morphism between two bicategories in the sense of \cite[\S4]{Benabou} a  \emph{$2$-functor}. 

\subsection{Invertible $1$-cells}\label{ssec:1cells}

Recall that an \emph{internal equivalence} between two $0$-cells $(A,\cH)$ and $(B,\cK)$ in ${\mathsf{PB}}^\ell$ is given by two $1$-cells $(P,\alpha,\beta)$ and $(Q,\sigma,\theta)$ in $\lPB{\cH}{\cK}$ resp.\ $\lPB{\cK}{\cH}$, such that
$$
P\cotensor{\cK}Q \cong \mathscr{U}(\cH),\quad  Q\cotensor{\cH}P \cong \mathscr{U}(\cK),
$$
holds
as $1$-cells, respectively, in  $\bPB{\cH}{\cH}$ and $\bPB{\cK}{\cK}$. 
Here we are implicitly assuming the triangle property, that is, 
we assume  the following diagrams
\begin{small}
\begin{equation}\label{Eq:trianglI}
\xymatrix@R=13pt{  & Q \cotensor{\cH} \cH  \ar@{->}^-{\cong}[r] & Q\cotensor{\cH}\big(P\cotensor{\cK}Q\big)\;\; \ar@{^{(}->}^-{}[rd] \ar@{->}^-{\cong}[dd]  & \\  Q \ar@{->}^-{\cong}[ru] \ar@{->}_-{\cong}[rd] & & & Q\tensor{A}P\tensor{B}Q \\ & \cK \cotensor{\cK} Q \ar@{->}^-{\cong}[r] & \big(Q\cotensor{\cH}P\big)\cotensor{\cK}Q \;\; \ar@{^{(}->}^-{}[ru]  &  }
\end{equation}
\end{small}
and 
\begin{small}
\begin{equation}\label{Eq:trianglII}
\xymatrix@R=13pt{  & P \cotensor{\cK} \cK  \ar@{->}^-{\cong}[r] & P\cotensor{\cK}\big(Q\cotensor{\cH}P\big) \;\; \ar@{^{(}->}^-{}[rd] \ar@{->}^-{\cong}[dd]  & \\  P \ar@{->}^-{\cong}[ru] \ar@{->}_-{\cong}[rd] & & & P\tensor{B}Q\tensor{A}P \\ & \cH \cotensor{\cH} P \ar@{->}^-{\cong}[r] & \big(P\cotensor{\cK}Q\big)\cotensor{\cH}P \;\; \ar@{^{(}->}^-{}[ru] &  }
\end{equation}
\end{small}
to be commutative. In this case, we also say that $(A, \cH)$ and $(B, \cK)$ are \emph{internally equivalent in $\mathsf{PB}^\ell$}.  Internal equivalences are, up to $2$-isomorphisms,  uniquely determined. More precisely, given a $1$-cell $P$ in $\mathsf{PB}^\ell$, if we assume that there exists $Q$ and $Q'$ in $\mathsf{PB}^\ell$ such   that 
$$
Q\cotensor{\cH}P \cong \mathscr{U}(\cK),\quad P\cotensor{\cK}Q \cong \mathscr{U}(\cH),
$$ and 
$$
Q'\cotensor{\cH}P \cong \mathscr{U}(\cK),\quad P\cotensor{\cK}Q' \cong \mathscr{U}(\cH),
$$  then we have $Q\cong Q'$ as $1$-cells. As in the general case, this  is an  easy consequence  of the associativity of the cotensor product in $\mathsf{PB}^\ell$. Such a $P$ is called an \emph{invertible left principal bundle}.  

Examples of invertible left principal bundles are typically obtained by bibundles:
\begin{prop}
\label{fororomano} 
Let $(P,\alpha, \beta)$ be a left principal $(\cH, \cK)$-bundle and let $(Q, \sigma, \gamma)$ be a right principal $(\cK, \cH)$-bundle.
\begin{enumerate}
\item
The translation map 
$\tauup: \cH \to P\tensor{B}P$ 
factors through the map 
$$
\tauup': \cH \to    P\cotensor{\cK}P^{\co}. 
$$ 
Analogously, the translation map
$\nuup: \cK \to Q\tensor{A}Q$ factors through
$$
\nuup': \cK \to Q \cotensor{\cH} Q^{\co}.
$$
\item
Assume moreover that $(P,\alpha, \beta)$ is a  
principal $(\cH,\cK)$-bibundle. Then $(P^{\co}, \beta,\alpha)$ is a principal $(\cK, \cH)$-bibundle and the translation maps induce isomorphisms
$$
\mathscr{U}(\cH) \overset{\simeq}{\longrightarrow}  P\cotensor{\cK}P^{\co}, 
\quad \mathscr{U}(\cK)  \overset{\simeq}{\longrightarrow}  P^{\co} \cotensor{\cH} P
$$
of 
principal $(\cH, \cH)$-bibundles resp.\ of principal $(\cK, \cK)$-bibundles. Furthermore, $(P,\alpha, \beta)$ is an invertible $1$-cell in $\mathsf{PB}^\ell(\cH,\cK)$.
\end{enumerate}
\end{prop}
\begin{proof}
Part {\em (i)}: to show that the image of the map $\tauup: u \mapsto u_+ \otimes_\behhe u_-$ lands for every $u \in \cH$ in the cotensor product $P\cotensor{\cK}P^{\co}$, we need to show that 
$$
u_{+(0)} \otimes_\behhe u_{+(1)} \otimes_\behhe u_- = u_{+} \otimes_\behhe \mathscr{S}(u_{-(1)}) \otimes_\behhe u_{-(0)} \quad \in P \otimes_\behhe \cK \otimes_\behhe P,
$$
where we used the coopposite comodule structure given in \rmref{Eq:leftright}. This is done by applying the map
$$
 P \otimes_\behhe P \otimes_\behhe \cK \to P \otimes_\behhe \cK \otimes_\behhe P, \quad  p' \otimes_\behhe p \otimes_\behhe w \mapsto  p' \otimes_\behhe w\mathscr{S}(p_{(1)}) \otimes_\behhe p_{(0)} 
$$
to both sides of Eq.~\rmref{casadiaugusto1}. The situation for right bundles is proven {\em mutatis mutandum}.

Part {\em (ii)}: by Lemma \ref{lemma:cotensorpb} the cotensor product carries the structure of a principal bundle. It is furthermore clear that $\tauup'$ is compatible with the source and target maps of $\cH$. The fact that $\tauup'$ is left $\cH$-colinear follows directly from \rmref{japan}. To show that this map is also right $\cH$-colinear one uses \rmref{Eq:leftright} along with \rmref{Eq:S+-}. 
To prove that $\tauup'$ is an isomorphism then follows from Lemma \ref{lemma:ISO} as it is, by Eq.~\rmref{ceuta}, a morphism of principal (bi)bundles.
%
To check the last statement, one only needs to show the triangle property \eqref{Eq:trianglI} (notice that here there is, in fact, only one diagram).  Using the notation of \S\ref{ssec:gpb}, the commutativity of \eqref{Eq:trianglI} reads in this case:
$$
p_{(0)} \tensor{B} p_{(1)}{}^{-} \tensor{A}  p_{(1)}{}^{+} \, = \,  
p_{(-1) +} \tensor{B} p_{(-1) -} \tensor{A}  p_{(0)} \quad \in   P \tensor{B} P \tensor{A} P,
$$ 
for every $p \in P$. To verify this, one  first  applies the map $P\tensor{B}\can{P}{\cK}^{-1}$ to both terms and then uses Eq.~\eqref{Eq:p-+} in order to obtain the same element $p_{(0)} \tensor{B}  1_{\scriptscriptstyle{P}} \tensor{B} p_{(1)}$ in $P\tensor{B}P\tensor{B}\cK$.
\end{proof}

\begin{proposition}
\label{prop:el parte}\noindent
\begin{enumerate}
\item
Let $(P, \alpha, \beta)$ be a left principal $(\cH,\cK)$-bundle. Assume moreover that $(P, \alpha, \beta)$ is an invertible $1$-cell in ${\mathsf{PB}}^\ell$ with inverse $(Q,  \theta,\sigma) \in {\mathsf{PB}}^\ell(\cK,\cH)$. Then $(P, \alpha, \beta)$ is a principal $(\cH,\cK)$-bibundle and $(Q, \theta, \sigma)$ is a principal $(\cK,\cH)$-bibundle. Furthermore, we have an isomorphism 
$$
Q \cong P^{\co}
$$
of principal bundles.
\item 
Let $\B{\phi}: (A, \cH) \to (B, \cK)$ be a morphism of flat Hopf algebroids. Then $\B{\phi}$ is a weak equivalence if and only if the trivial bundle $P=\cH \otimes_{\scriptscriptstyle{\phi}} B$ is an invertible $1$-cell in $\lPB{\cH}{\cK}$.
\end{enumerate}
\end{proposition}
\begin{proof}
For better orientation, we recall here that the algebra diagrams defining $P$ and $Q$ are
 $$
\xymatrix{A\ar@{->}^-{\alpha}[r] & P & B \ar@{->}_-{\beta}[l]}, \qquad \xymatrix{A\ar@{->}^-{\sigma}[r] & Q & B, \ar@{->}_-{\theta}[l]} 
$$ 
where $\beta$ and $\sigma$ are faithfully flat, and also that the canonical maps $\can{\cH}{P}$ and $\can{\cK}{Q}$ are bijective.  

Part {\em (i)}:  by assumption, we have the following  $2$-isomorphisms 
$$
\chi: \cH \overset{\simeq}{\lra} P \cotensor{\cK} Q, \quad
u \longmapsto p^{\scriptscriptstyle{u}} \cotensor{\cK}  q^{\scriptscriptstyle{u}}, \qquad \mbox{and} \qquad
\zeta: \cK \overset{\simeq}{\lra} Q \cotensor{\cH} P, \quad
 w \longmapsto q^{\scriptscriptstyle{w}} \cotensor{\cH}  p^{\scriptscriptstyle{w}},
$$
where $\chi$ is, in particular, a morphism of $\cH$-bicomodules and $\zeta$ is a morphism of $\cK$-bicomodules. 
The triangle properties then say that we have, up to a canonical isomorphism,  
\begin{equation}
\label{Eq:chizeta}
\begin{array}{rcll}
\chi(p_{{(-1)}}) \cotensor{\cH} p_{(0)} &=& p_{(0)} \cotensor{\cK} \zeta(p_{{(1)}})   & \in P\tensor{B}Q\tensor{A}P, \\
\zeta(q_{{(-1)}}) \cotensor{\cK} q_{(0)} &=& q_{(0)} \cotensor{\cH} \chi(q_{(1)}) & \in Q\tensor{B}P\tensor{A}Q, 
\end{array}
\end{equation}
for all  $p \in P$, $q \in Q$. 
On the other hand, we also have an isomorphism 
\begin{equation}\label{Eq:Gamma}
\xymatrix@C=35pt{ P\tensor{B}Q  \ar@{->}^-{\cong}[r] & (P\cotensor{\cK}Q)\tensor{A}Q \ar@{->}^-{\chi^{-1}\tensor{A}Q}[r] &  
  \cH\tensor{A}Q}
\end{equation}
of $(\cH,\cK)$-bicomodules,
where the first isomorphism is the natural transformation of Eq.~\eqref{giotto1}. 
Using this  isomorphism, we can easily check that $\alpha$ is a faithful extension. 
Indeed, take a morphism $f$ such that $f\tensor{A}P=0$; 
then $f\tensor{A}\cH \tensor{A} Q = 0$ which yields $f=0$ since $\due \cH \ahha {}$ and $\due Q \ahha {}$ are faithfully flat.   
Now for a monomorphism $i: X \to X'$ of $A$-modules, we obtain,  using again the isomorphism \eqref{Eq:Gamma}, that $\ker(i\tensor{A}P)\tensor{B}Q = 0$, which by the bijectivity of the canonical map $\can{\cK}{Q}$ implies that $\ker(i\tensor{A}P)=0$ 
since $\due \cK \behhe {}$ and $\due Q \ahha {}$ are faithfully flat.  
This shows that $\alpha$ is a faithfully flat extension. 

We still need to check that the canonical map $\mathsf{can} : P\tensor{A}P \to P\tensor{B}\cK$ is bijective. To this end, we define what is going to be its inverse as
$$
\td{\mathsf{can}}: P\tensor{B}\cK \to P\tensor{A}P, \quad 
p\tensor{B}w \mapsto p \mathsf{g}(q^{\scriptscriptstyle{w}})\tensor{A}p^{\scriptscriptstyle{w}},
$$
where $\mathsf{g}$ is simultaneously the $A$-algebra and $B$-algebra map given explicitly by 
$$
\mathsf{g}:Q \to P, \quad q \mapsto  \beta \big(\varepsilon\big( \zeta^{-1}(q_{{(0)}}\cotensor{\cH}q_{(1)+})\big)\big) q_{(1)-}.
$$
This map satisfies
\begin{equation}
\label{Eq:G}
p^{{u}} \mathsf{g}(q^{{u}}) = \alpha\big(\varepsilon(u) \big), \qquad  
\mathsf{g}(q^{\scriptscriptstyle{w}}) p^{\scriptscriptstyle{w}} = \beta\big(\varepsilon(w) \big), 
\end{equation}
for every $u \in \cH, w \in \cK$, which is seen as follows: 
as for the second one, we have for $w \in \cK$ 
\begin{eqnarray*}
\mathsf{g}(q^{\scriptscriptstyle{w}}) p^{\scriptscriptstyle{w}} &=& \beta \big(\varepsilon\big( \zeta^{-1}(q^{\scriptscriptstyle{w}}{}_{{(0)}}\cotensor{\cH}q^{\scriptscriptstyle{w}}{}_{(1)+})\big) \big) q^{\scriptscriptstyle{w}}{}_{(1)-} p^{\scriptscriptstyle{w}} \\ 
&{=}&  \beta\big(\varepsilon\big( \zeta^{-1}(q^{\scriptscriptstyle{w}}\cotensor{\cH}p^{\scriptscriptstyle{w}}{}_{(-1)+})\big) \big) p^{\scriptscriptstyle{w}}{}_{(-1)-} p^{\scriptscriptstyle{w}}{}_{(0)} \\ 
&\overset{\eqref{Eq:u-+}}{=}&  \beta\big(\varepsilon\big( \zeta^{-1}(q^{\scriptscriptstyle{w}}\cotensor{\cH}p^{\scriptscriptstyle{w}})\big) \big) \\ 
&=&  \beta\big(\varepsilon(w) \big).
\end{eqnarray*}
As for the first equation in \eqref{Eq:G}, by the right $\cH$-colinearity of $\chi$ and Eq.~\eqref{japan}  
$$
p^{\scriptscriptstyle{u}} \tensor{B} (q^{\scriptscriptstyle{u}}{}_{(0)}\cotensor{\cH} q^{\scriptscriptstyle{u}}{}_{(1)+})
\tensor{B}q^{\scriptscriptstyle{u}}{}_{(1)-} 
\,=\, p^{\scriptscriptstyle{u_{+ (-1)}}} \tensor{B} (q^{\scriptscriptstyle{u_{+ (-1)}}}\cotensor{\cH}u_{+(0)}) \tensor{B} u_{-} \quad \in P\tensor{B}\big(Q\cotensor{\cH}P\big)\tensor{B}P,
$$ 
holds for any $u \in \cH$, an equation which can be seen in $P\tensor{B}Q\tensor{A}P\tensor{B}P$ since $P_\behhe$ is flat. Therefore, 
$$
p^{\scriptscriptstyle{u}} \mathsf{g}(q^{\scriptscriptstyle{u}}) \,=\, 
p^{\scriptscriptstyle{u_{+ (-1)}}} \beta\big( \varepsilon \zeta^{-1}\big( q^{\scriptscriptstyle{u_{+ (-1)}}}\cotensor{\cH}u_{+(0)}\big) \big) u_{-}.
$$
On the other hand, by the first equality of Eq.~\eqref{Eq:chizeta},  
$$
(p^{\scriptscriptstyle{u_{+ (-1)}}} \cotensor{\cK} q^{\scriptscriptstyle{u_{+ (-1)}}} ) \cotensor{\cH} u_{+(0)}\tensor{B} u_{-} 
\, = \, \big( \chi(u_{+ (-1)})\cotensor{\cH}u_{+(0)}\big)\tensor{B} u_{-}  \,=\,  \big( u_{+ (0)}\cotensor{\cK}\zeta(u_{+(1)})\big)\tensor{B} u_{-},
$$
which implies that
$$
p^{\scriptscriptstyle{u_{+ (-1)}}} \tensor{B} ( q^{\scriptscriptstyle{u_{+ (-1)}}}\cotensor{\cH}u_{+(0)} ) \tensor{B} u_{-} \, = \, u_{+(0)}\tensor{B}\zeta(u_{+ (1)})\tensor{B} u_{-},
$$ 
from which, in turn, we obtain that 
$$
p^{\scriptscriptstyle{u}} \mathsf{g}(q^{\scriptscriptstyle{u}}) \, =\, 
p^{\scriptscriptstyle{u_{+\,(-1)}}} \beta\big( \varepsilon \zeta^{-1}\big( q^{\scriptscriptstyle{u_{+(-1)}}}\cotensor{\cH}u_{+(0)}\big) \big) u_{-}
\, = \, u_{+(0)}\beta\big(\varepsilon(u_{+(1)})\big) u_{-}\, = \, u_{+}u_{-} \, \overset{\eqref{Eq:eps+-}}{=} \, \alpha(\varepsilon(u)),
$$ 
as claimed. Using Eqs.~\eqref{Eq:G}, we now compute from one hand,
\begin{eqnarray*}
\mathsf{can} \circ \td{\mathsf{can}} (p\tensor{B}w) & = & \mathsf{can} ( p\,\mathsf{g}(q^{\scriptscriptstyle{w}})\tensor{A}p^{\scriptscriptstyle{w}}) \\
& = &  p\,\mathsf{g}(q^{\scriptscriptstyle{w}})p^{\scriptscriptstyle{w}}{}_{{(0)}}\tensor{B} p^{\scriptscriptstyle{w}}{}_{{(1)}} \\
& = &  p\,\mathsf{g}(q^{{w_{(1)}}})p^{{w_{(1)}}}\tensor{B} w_{{(2)}} \\
& = &  p\,\beta(\varepsilon(w_{(1)}))\tensor{B} w_{(2)} \\ &=& p\tensor{B}w,
\end{eqnarray*}
and from the other side, 
\begin{eqnarray*}
\td{\mathsf{can}} \circ \mathsf{can}(p'\tensor{A}p) &=& \td{\mathsf{can}}(p'p_{{(0)}}\tensor{B}p_{{(1)}}) \\ 
& =& p'p_{{(0)}}
\mathsf{g}(q^{\scriptscriptstyle{p_{(1)}}}) \tensor{A}p^{\scriptscriptstyle{p_{(1)}}} \\
&\overset{\eqref{Eq:chizeta}}{ =} & p'p^{\scriptscriptstyle{p_{(-1)}}}
\mathsf{g}(q^{\scriptscriptstyle{p_{(-1)}}}) \tensor{A}p_{{(0)}} \\
& =& p'\alpha\big( \varepsilon( p_{{(-1)}})\big) 
 \tensor{A}p_{{(0)}} \\ &=& p'\tensor{A}p,
\end{eqnarray*}
which gives the desired bijection, and so $(P,\alpha,\beta)$ is a principal bibundle.  Similarly, one checks that $(Q,\theta,\sigma)$ is so as well.

To complete the proof of the first part, we also need to check that $Q$ is the opposite bundle of $P$. For this, we use the following chain of isomorphisms of $\Bbbk$-modules 
$$
 P\tensor{A}P   \cong  \cK\tensor{B}P \cong (Q\cotensor{\cH}P) \tensor{B}P \cong Q\tensor{A}P,
$$ where the last isomorphism is given by Eq.~\eqref{giotto1},
which leads to an isomorphism $P\cong Q$ of  $A$-modules since $P$ is faithfully flat over $A$ 
(alternatively, one can try to check that $\mathsf{g}: P \to Q$ is a bundle map and hence an isomorphism by Lemma \ref{lemma:ISO}).  
In the same way, using the faithfully flatness of $P$ over $B$, one shows 
that this is also an isomorphism of $B$-modules, and thus that $Q$ is the opposite bundle of $P$.

To prove {\em (ii)}, assume first that $\B{\phi}$ is a weak equivalence.
Then $P$ is a right principal $(\cH,\cK)$-bundle by Proposition \ref{lemma:trivial bibundle}, along with the fact that $- \cotensor{\cH} P$ 
defines an equivalence of categories with inverse
$
- \cotensor{\cK} P^{\co}.
$
From this it is clear that
$ 
P \cotensor{\cK} P^{\co} \simeq \mathscr{U}(\cH)
$ 
and 
$P^{\co}  \cotensor{\cH} P\simeq \mathscr{U}(\cK)$, see Example \ref{Exam:H} for notation.
To prove the converse, using Proposition \ref{lemma:trivial bibundle} again, we only have to show that $P = \cH \otimes_{\scriptscriptstyle{\phi}} B$ is a bibundle, which is a direct consequence of {\em (i)}. 
\end{proof}

Recall that a {\em bigroupoid} (see, {\em e.g.}, \cite{Noo:NOTGTGACM})
is a bicategory in which every $1$-cell and every $2$-cell has an inverse (not necessarily in the strict sense for $1$-cells).

\begin{corollary}\label{coro:twoJapanese}
For two $0$-cells $(A,\cH)$ and $(B,\cK)$ (that is, flat Hopf algebroids), the full subcategory of invertible $1$-cells in $\mathsf{PB}^\ell(\cH,\cK)$ coincides with the full subcategory  $\mathsf{PB}^{\scriptscriptstyle{b}}(\cH,\cK)$ of principal bibundles. In particular, the bicategory $\mathsf{PB}^\ell$ is a bigroupoid.
\end{corollary}

The last statement follows from Lemma \ref{lemma:ISO}

\subsection{The  $2$-functor $\mathscr{P}$ and principal bundles as universal solution}
\label{ssec:2F}
It is well-known that  groupoids, functors, and natural transformations form a $2$-category. 
Adapting this to  Hopf algebroids, one can construct a $2$-category as observed in \cite[\S3.1]{Naumann:07}.  
Here,  $0$-cells are Hopf algebroids (or even flat ones), $1$-cells are morphisms of Hopf algebroids, and for 
two $1$-cells $(\zeta_{{0}},\zeta_{{1}}), (\theta_{{0}},\theta_{{1}}): (A, \cH) \to (B, \cK)$, a $2$-cell $\fk{c}:  (\zeta_{{0}},\zeta_{{1}}) \to (\theta_{{0}},\theta_{{1}})$ is defined to be an algebra map $\fk{c}: \cH \to B$ 
that makes the diagrams
\begin{equation}\label{Eq:2cells}
\xymatrix{ \cH \ar@{->}^-{\fk{c}}[r] & B \\ A \ar@{->}_-{\zeta_{\scriptscriptstyle{0}}}[ru] \ar@{->}^-{\sf{s}}[u] & }\qquad  \xymatrix{ \cH \ar@{->}^-{\fk{c}}[r] & B \\ A \ar@{->}_-{\theta_{\scriptscriptstyle{0}}}[ru] \ar@{->}^-{\sf{t}}[u] & } \qquad \xymatrix{ \cH \ar@{->}^-{\Delta}[rr] \ar@{->}_-{\Delta}[d] &  & \cH\tensor{A}\cH \ar@{->}^-{m_\ckppa (\zeta_{\scriptscriptstyle{1}}\tensor{A}\sf{t}\fk{c})}[d]  \\ \cH\tensor{A}\cH \ar@{->}_-{m_\ckppa(\sf{s}{\fk{c}}\,\tensor{A}\theta_{\scriptscriptstyle{1}})}[rr]  & & \cK }
\end{equation}
commutative, where $m_\ckppa$ denotes the multiplication in $\cK$.
The identity $2$-cell for $(\zeta_{{0}},\zeta_{{1}})$ is given by $1_{\scriptscriptstyle{\zeta}}:=  \zeta_{{0}} \circ \varepsilon$.
The tensor product (or vertical composition) of $2$-cells  
is given as
$$
\xymatrix{  \fk{c}' \circ \fk{c}:  (\zeta_{{0}},\zeta_{{1}}) \ar@{->}^-{\fk{c}}[r] & (\theta_{{0}},\theta_{{1}}) \ar@{->}^-{\fk{c}'}[r] & (\xi_{{0}},\xi_{{1}}),}
$$
which yields a map
\begin{equation}
\label{Eq:verticalComp}
\fk{c}' \B{\circ} \fk{c}: \cH  \to B, \quad u \mapsto \fk{c}(u_{{(1)}}) \fk{c}'(u_{{(2)}}).
\end{equation}
We denote by  $2\text{-}\mathsf{HAlgd}$ 
the $2$-category whose $0$-cells are flat Hopf algebroids. 
Examples of $2$-cells in this $2$-category are described by the following lemma:

\begin{lemma}\label{lemma:volveranotro}
Let $\B{\phi}: (A,\cH) \to (B,\cK)$ be a morphism of flat Hopf algebroids. As in Example \ref{exam:HopfMorph}, consider its associated trivial left principal $(\cH,\cK)$-bundle $(P :=\cH\tensor{\scriptscriptstyle{\phi}}B,\alpha,\beta)$ together with the diagram 
$$
\xymatrix@R=15pt{ & (P,\cH\lJoin P\rJoin\cK) & \\ (A,\cH) \ar@{->}^-{{\B{\alpha}=(\alpha, \,\alpha_1)}\;\;}[ur] \ar@{->}_-{\B{\phi}}[rr]& & \ar@{->}_-{\;\;{\B{\beta}=(\beta, \,\beta_1)}}[ul] (B,\cK)}
$$ 
of Hopf algebroids,
where the top is the two-sided translation Hopf algebroid defined in Lemma \ref{lemma:bisemidirect}. 
Then there is a $2$-isomorphism 
$\B{\alpha} \cong \B{\beta}\circ \B{\phi}$, that is, the above diagram is commutative up to an isomorphism.
\end{lemma}
\begin{proof}
Consider the following two algebra maps 
$$
\fk{c}: \cH \to P, \quad u \mapsto u\tensor{\scriptscriptstyle{\phi}}1_{ \scriptscriptstyle{B}}, \quad \mbox{and} \quad \fk{c}': \cH \to P, \quad u \mapsto \mathscr{S}(u)\tensor{\scriptscriptstyle{\phi}}1_{ \scriptscriptstyle{B}}.
$$
Let us check that $\fk{c}: \B{\alpha} \to \B{\beta} \circ \B{\phi}$ and $\fk{c'}: \B{\beta} \circ \B{\phi} \to  \B{\alpha}$ are $2$-cells in $2\text{-}\mathsf{HAlgd}$. To this end, we need to show the commutativity of the diagrams in Eq.~\eqref{Eq:2cells}, corresponding to $\fk{c}$ and $\fk{c'}$. 
By definition, it is clear that the triangles
$$
\xymatrix{ \cH \ar@{->}^-{\fk{c}}[r] & P \\ A \ar@{->}_-{\alpha}[ru] \ar@{->}^-{\sf{s}}[u] & }\qquad  \xymatrix{ \cH \ar@{->}^-{\fk{c}}[r] & P \\ A \ar@{->}_-{\beta\phi_{\scriptscriptstyle{0}}}[ru] \ar@{->}^-{\sf{t}}[u] & } \qquad \qquad  
\xymatrix{ \cH \ar@{->}^-{\fk{c}'}[r] & P \\ A \ar@{->}_-{\beta\phi_{\scriptscriptstyle{0}}}[ru] \ar@{->}^-{\sf{s}}[u] & }\qquad  \xymatrix{ \cH \ar@{->}^-{\fk{c}'}[r] & P \\ A \ar@{->}_-{\alpha}[ru] \ar@{->}^-{\sf{t}}[u] & }
$$
commute.
We only show the rectangle in \eqref{Eq:2cells} for $\fk{c'}$ since an analogous proof works for $\fk{c}$. 
Thus, we want to show that $m_{\scriptscriptstyle{\cH\lJoin P\rJoin\cK}}\circ \big( (\beta_1 \circ \phi_{\scriptscriptstyle{1}})\tensor{A}(\Sf{t} \circ \fk{c'})\big) \circ \Delta \,=\,  m_{\scriptscriptstyle{\cH\lJoin P\rJoin\cK}} \circ  \big( (\Sf{s} \circ \fk{c'})\tensor{A}\alpha_{\scriptscriptstyle{1}}\big) \circ \Delta$, where the target and source $\Sf{t}, \Sf{s}$ are those of $\cH\lJoin P\rJoin\cK$.
Taking into account the structure maps of Lemma  \ref{lemma:bisemidirect}, we compute for $u \in \cH$   
\begin{eqnarray*}
m_{\scriptscriptstyle{\cH\lJoin P\rJoin\cK}}\circ \big( (\beta_1 \circ \phi_{\scriptscriptstyle{1}})\tensor{A}(\Sf{t} \circ \fk{c'})\big) \circ \Delta(u) &=& \big( 1_{ \scriptscriptstyle{\cH}}\tensor{A}1_{\scriptscriptstyle{P}}\tensor{B} \phi_{\scriptscriptstyle{1}}(u_{(1)})\big) \Sf{t}\big(\mathscr{S}(u_{(2)})\tensor{\scriptscriptstyle{\phi}}1_{\scriptscriptstyle{B}} \big) \\ 
&=& \big( 1_{ \scriptscriptstyle{\cH}}\tensor{A}1_{\scriptscriptstyle{P}}\tensor{B} \phi_{\scriptscriptstyle{1}}(u_{(1)})\big) \big(u_{(4)} \tensor{A} (\mathscr{S}(u_{(3)}) \tensor{\scriptscriptstyle{\phi}}1_{\scriptscriptstyle{B}}) \tensor{B} \phi_{\scriptscriptstyle{1}}(\mathscr{S}(u_{(2)})) \big) \\ 
&=&  u_{(4)} \tensor{A} (\mathscr{S}(u_{(3)})\tensor{\scriptscriptstyle{\phi}}1_{\scriptscriptstyle{B}}) \tensor{B} \phi_{\scriptscriptstyle{1}}(u_{(1)})\phi_{\scriptscriptstyle{1}}(\mathscr{S}(u_{(2)})) \\ 
&=&  u_{(3)}\tensor{A} (\mathscr{S}(u_{(2)})\tensor{\scriptscriptstyle{\phi}}1_{\scriptscriptstyle{B}}) \tensor{B} \Sf{s}(\phi_{\scriptscriptstyle{0}}(\varepsilon(u_{(1)}))) \\ 
&=& u_{(3)}\tensor{A} \big( \mathscr{S}(u_{(2)})\tensor{\scriptscriptstyle{\phi}}\phi_{\scriptscriptstyle{0}}(\varepsilon(u_{(1)})) \big)\tensor{B} 1_{ \scriptscriptstyle{\cK}} \\ 
&=&  u_{(2)}\tensor{A} \big( \mathscr{S}(u_{(1)})\tensor{\scriptscriptstyle{\phi}} 1_{ \scriptscriptstyle{B}} \big)\tensor{B} 1_{ \scriptscriptstyle{\cK}} \\ 
&=& m_{\scriptscriptstyle{\cH\lJoin P\rJoin\cK}}\circ \big( (\Sf{s}\circ \fk{c'})\tensor{A}\alpha_1\big) \circ \Delta(u).
\end{eqnarray*}
Finally, using the vertical composition as defined in \eqref{Eq:verticalComp}, one can easily check that $\fk{c} \B{\circ} \fk{c'}= (\beta \phi_{\scriptscriptstyle{0}}) \circ \varepsilon$ and that  $\fk{c'} \B{\circ} \fk{c}= \alpha \circ \varepsilon$. Therefore  $\fk{c} \B{\circ} \fk{c'}=1_{\scriptscriptstyle{\B{\beta}\circ \B{\phi}}}$ and $\fk{c'} \B{\circ} \fk{c} =1_{\scriptscriptstyle{\alpha}}$, and this completes the proof.
\end{proof}

For a non necessarily trivial  bundle, one has the following property:
\begin{lemma}\label{lemma:paciencia}
Let $(P,\alpha, \beta)$ be a $1$-cell in $\mathsf{PB}^\ell(\cH,\cK)$, and denote by $(P,\cH\lJoin  P\rJoin \cK)$ the two-sided translation Hopf algebroid, together with the diagram 
$$
\xymatrix@R=15pt{ & (P,\cH\lJoin  P\rJoin \cK) & \\ (A,\cH) \ar@{->}^-{{\B{\alpha}=(\alpha, \,\alpha_1)}\;\;}[ur] & & \ar@{->}_-{\;\;{\B{\beta}=(\beta, \,\beta_1)}}[ul] (B,\cK) }
$$ 
of flat Hopf algebroids.
Consider the trivial bundles $\B{\alpha}^*\big( \mathscr{U}(\cH)\big)=\cH\tensor{\B{\alpha}}P$ and $\B{\beta}^*\big( \mathscr{U}(\cK)\big)=\cK\tensor{\B{\beta}}P$.
Then the map
$$
\fk{h}:(P,\alpha,\beta) \longrightarrow \Big(\B{\alpha}^*\big( \mathscr{U}(\cH)\big)  \cotensor{\cH\lJoin  P\rJoin \cK}  \B{\beta}^*\big( \mathscr{U}(\cK)\big)^{\co},\td{\alpha}, \td{\beta}\Big), \quad p \longmapsto \big(p_{(-1)}\tensor{\B{\alpha}}p_{(0)} \big) \, \cotensor{\cH\lJoin  P\rJoin \cK}\, \big(1_{\scriptscriptstyle{P}}\tensor{\B{\beta}}p_{(1)} \big)
$$
defines an isomorphism of left principal $(\cH,\cK)$-bundles.
\end{lemma}
\begin{proof}
Recall that a generic element of the form $(u\tensor{\B{\alpha}}p) \tensor{P} (p'\tensor{\B{\beta}}w) \in \B{\alpha}^*\big( \mathscr{U}(\cH)\big)\, \tensor{P}\, \B{\beta}^*\big( \mathscr{U}(\cK)\big)^{\co}$ belongs to the cotensor product $\B{\alpha}^*\big( \mathscr{U}(\cH)\big)\, \cotensor{\cH\lJoin  P\rJoin \cK}\, \B{\beta}^*\big( \mathscr{U}(\cK)\big)^{\co}$ if and only if
\begin{equation}
\label{firenzesmn}
(u_{(1)}\tensor{\B{\alpha}} p)  \tensor{P} \big( u_{(2)}\tensor{A}1_{\scriptscriptstyle{P}}\tensor{B} 1_{\scriptscriptstyle{\cK}}\big) \tensor{P} (p'\tensor{\B{\beta}}w)  =  (u\tensor{\B{\alpha}}1_{\scriptscriptstyle{P}})  \tensor{P} \big( 1_{\scriptscriptstyle{\cH}}\tensor{A}pp'\tensor{B} w_{(1)}\big) \tensor{P} ( 1_{\scriptscriptstyle{P}}\tensor{\B{\beta}}w_{(2)})
\end{equation}
holds true in $\B{\alpha}^*\big( \mathscr{U}(\cH)\big)\tensor{P} (\cH\lJoin  P\rJoin \cK) \tensor{P} \B{\beta}^*\big( \mathscr{U}(\cK)\big)^{\co}$. Hence, in order to check that $\fk{h}$ is well-defined, one needs to show this equality for $\fk{h}(p)$,  for all $ p \in P$.  The left hand side in \rmref{firenzesmn} for $\fk{h}(p)$ reads as
$$
(p_{(-2)}\tensor{\B{\alpha}}1_{\scriptscriptstyle{P}})  \tensor{P} \big( p_{(-1)}\tensor{A}1_{\scriptscriptstyle{P}}\tensor{B} 1_{\scriptscriptstyle{\cK}}\big) \tensor{P} ( p_{(0)}\tensor{\B{\beta}}p_{(1)}),
$$
while the right hand side becomes
$$
(p_{(-1)}\tensor{\B{\alpha}}1_{\scriptscriptstyle{P}})  \tensor{P} \big( 1_{\scriptscriptstyle{\cH}}\tensor{A}p_{(0)}\tensor{B} p_{(1)}\big) \tensor{P} ( 1_{\scriptscriptstyle{P}}\tensor{\B{\beta}}p_{(2)}).
$$
Using the expression of the target map of $\cH\lJoin  P\rJoin \cK$ given in Lemma \ref{lemma:bisemidirect}, we have that 
\begin{eqnarray*}
& & (p_{(-2)}\tensor{\B{\alpha}}1_{\scriptscriptstyle{P}})  \tensor{P} \big( p_{(-1)}\tensor{A}1_{\scriptscriptstyle{P}}\tensor{B} 1_{\scriptscriptstyle{\cK}}\big) \tensor{P} ( p_{(0)}\tensor{\B{\beta}}p_{(1)})  \\ &= & (p_{(-2)}\tensor{\B{\alpha}}1_{\scriptscriptstyle{P}})  \tensor{P} \big( p_{(-1)}\tensor{A} 1_{\scriptscriptstyle{P}} \tensor{B}  1_{\scriptscriptstyle{\cK}} \big) \Sf{t}(p_{(0)}) \tensor{P} ( 1_{\scriptscriptstyle{P}}\tensor{\B{\beta}}p_{(1)})\\ &= & (p_{(-3)}\tensor{\B{\alpha}}1_{\scriptscriptstyle{P}})  \tensor{P} \big( p_{(-2)}\mathscr{S}(p_{(-1)})\tensor{A}p_{(0)}\tensor{B} p_{(1)}\big) \tensor{P} ( 1_{\scriptscriptstyle{P}}\tensor{\B{\beta}}p_{(2)})  \\ &\overset{\eqref{Eq:penultima}}{=} & (p_{(-2)}\tensor{\B{\alpha}}1_{\scriptscriptstyle{P}})  \tensor{P} \big( \Sf{s}(\varepsilon(p_{(-1)}))\tensor{A}p_{(0)}\tensor{B} p_{(1)}\big) \tensor{P} ( 1_{\scriptscriptstyle{P}}\tensor{\B{\beta}}p_{(2)}) \\ & = & (p_{(-1)}\tensor{\B{\alpha}}1_{\scriptscriptstyle{P}})  \tensor{P} \big( 1_{\scriptscriptstyle{\cH}}\tensor{A}p_{(0)}\tensor{B} p_{(1)}\big) \tensor{P} ( 1_{\scriptscriptstyle{P}}\tensor{\B{\beta}}p_{(2)}),
\end{eqnarray*} which shows that $\fk{h}$ is a well-defined map. 
Recall now that the algebra maps $\td{\alpha}$ and $\td{\beta}$ are given by 
$$
\td{\alpha}(a)= (\Sf{s}(a) \tensor{\B{\alpha}}1_{\scriptscriptstyle{P}})\cotensor{\cH\lJoin  P\rJoin \cK} (1_{\scriptscriptstyle{P}}\tensor{\B{\beta}}1_{\scriptscriptstyle{\cK}}); \qquad \td{\beta}(b)= (1_{\scriptscriptstyle{\cH}} \tensor{\B{\alpha}}1_{\scriptscriptstyle{P}})\cotensor{\cH\lJoin  P\rJoin \cK} (1_{\scriptscriptstyle{P}}\tensor{\B{\beta}}\Sf{t}(b)).
$$
Clearly, $\fk{h}$ is  simultaneously an $A$-algebra and a $B$-algebra map, and the fact that $\fk{h}$ is an $(\cH, \cK)$-bicomodule map is also clear from the definitions. 
Thus, $\fk{h}$ is a morphism of left principal bundles,
and so an isomorphism by Lemma \ref{lemma:ISO}.
\end{proof}

Next we give a further property of  the Diagram \rmref{kontoauszug} that appeared in Proposition \ref{prop:1out2}.

\begin{lemma}\label{lema:1out2}
Let $\theta_i:(A,\cH) \to (C_i,\cJ_i)$, $i=1,2$, be two weak equivalences. Then the 
diagram of weak equivalences  \rmref{kontoauszug} constructed in Proposition \ref{prop:1out2} is commutative up to a $2$-isomorphism.
\end{lemma}
\begin{proof}
Denote by $P_i := \cH\tensor{\theta_i}C_i$, $i=1,2$ the respective associated trivial bibundles of $\theta_i$. Up to a canonical isomorphism, the bundle $Q := P_1^{\co}\cotensor{\cH}P_2$ is of the form $Q=C_1\tensor{A}\cH\tensor{A}C_2$. So,  considering the obvious algebra map $\fk{c}: \cH \to Q$, $u \mapsto 1\tensor{A}u \tensor{A}1$ and writing $\B{\phi}:= \B{\zeta}_1 \circ \B{\theta}_1$ and $\B{\psi}:=\B{\zeta}_2 \circ \B{\theta}_2$, one can use the definition of the maps $\B{\zeta}_i$ in Lemma \ref{lemma:bisemidirect} to show that the diagrams in \eqref{Eq:2cells} are commutative, and that hence $\fk{c}: \B{\phi} \to \B{\psi}$ is a $1$-cell in $2\text{-}\mathsf{HAlgd}$. Its inverse is $\fk{c}^{-1}: \cH \to Q$ which sends $u \mapsto 1\tensor{A}\mathscr{S}(u)\tensor{A}1$.
\end{proof}

Denote by $ \mathsf{PB}^{\ell \, \co}$ the conjugate  bicategory of $\mathsf{PB}^\ell$, defined by reversing $2$-cells. 

\begin{proposition}\label{prop:P}
There is a $2$-functor 
$$
\mathscr{P}: 2\text{-}\mathsf{HAlgd} \longrightarrow \mathsf{PB}^{\ell \, \co},
$$
which sends any $1$-cell $\B{\phi}: (A,\cH) \to (B, \cK)$ to its associated trivial left principal bundle  $P=\cH\tensor{\scriptscriptstyle{\phi}}B$. Moreover, a $1$-cell $\B{\phi}$ in $2\text{-}\mathsf{HAlgd}$ is a weak equivalence if and only if $\mathscr{P}(\B{\phi})$ is an invertible $1$-cell in $\mathsf{PB}^{\ell \, \co}$.
\end{proposition}
\begin{proof}
Let $\fk{c}: \B{\phi} \to \B{\psi}$ be a $2$-cell in $2\text{-}\mathsf{HAlgd}$. Then its  image by $\mathscr{P}$ is given by  
$$
\mathscr{P}(\fk{c}):  \cH\tensor{\psi}B  \to \cH\tensor{\phi}B, \quad u\tensor{\psi}b  \mapsto u_{(1)}\tensor{\phi}\fk{c}(u_{(2)})b,  
$$ 
which is easily shown to be a morphism of left principal bundles.  The remaining axioms which $\mathscr{P}$ is required to fulfil are also easily shown and therefore left to the reader. Nevertheless, notice that for two composable $1$-cells $\B{\phi}:(A,\cH) \to (B, \cK)$ and $\B{\phi}': (B, \cK) \to (C,\cJ)$ one has
$$
\mathscr{P}(\B{\phi}' \circ \B{\phi})\,\, \cong\,\, \mathscr{P}(\B{\phi})\cotensor{\cK} \mathscr{P}(\B{\phi}'),
$$
that is, $\mathscr{P}$ is contravariant. The last statement is a direct consequence of  Proposition \ref{prop:el parte} {\em (ii)}.
\end{proof}

The following  theorem is Theorem \ref{thm:Ramadan}  in  the Introduction and is our second main result:

\begin{thm}\label{thm:elevenyearsandsevenmonths}
Let $\mathscr{F}: 2\text{-}\Sf{HAlgd} \to \mathscr{B}$ be a 2-functor which sends weak equivalences to invertible $1$-cells. Then,  up to isomorphism (of $2$-functors), there is a unique $2$-functor $\td{\mathscr{F}}$ such that the following  diagram 
\begin{equation}
\label{tiburtina}
\xymatrix@R=20pt{ 2\text{-}\Sf{HAlgd} \ar@{->}_-{\mathscr{F}}[rrd]   \ar@{->}^-{\mathscr{P}}[rr] & & \mathsf{PB}^{\ell \, \co} \ar@{.>}^-{\td{\mathscr{F}}}[d]  \\ & & \mathscr{B} }
\end{equation}
commutes up to an isomorphism of $2$-functors. 
\end{thm}
\begin{proof}
For two $0$-cells $(A,\cH)$ and $(B,\cK)$ and a $1$-cell  $(P,\alpha,\beta)$ in 
$\mathsf{PB}^{\ell \, \co}(\cH,\cK)$, from Proposition \ref{prop:Cielitomissyoutoo} one obtains that $\B{\beta}:(B,\cK) \to (P,\cH \lJoin P \rJoin \cK)$ is a weak equivalence. Then, by assumption, $\mathscr{F}(\B{\beta})$ is an invertible $1$-cell in $\mathscr{B}\big(\mathscr{F}(A,\cH), \mathscr{F}(B,\cK)\big)$; denote by $\mathscr{F}(\B{\beta})^{-1} \in \mathscr{B}\big(\mathscr{F}(B,\cK), \mathscr{F}(A,\cH)\big)$ its inverse. Define furthermore 
$$
\td{\mathscr{F}}(P,\alpha,\beta) :=  \mathscr{F}(\B{\beta})^{-1} \circ \mathscr{F}(\B{\alpha}),
$$ 
which gives a $1$-cell in  $\mathscr{B}\big(\mathscr{F}(A,\cH), \mathscr{F}(B,\cK)\big)$. In particular, the image of the unit bundle $(\mathscr{U}(\cH), \bf{\Sf{s}}, \bf{\Sf{t}})$  then is, by using Lemma  \ref{lemma:volveranotro}, of the form
$$
\td{\mathscr{F}}\big( \mathscr{U}(\cH)\big) \cong \mathscr{F}(\id_{\scriptscriptstyle{(A,\, \cH)}}) = \bf{1}_{\scriptscriptstyle{\mathscr{F}(A,\, \cH)}}, 
$$
the identity $1$-cell of the monoidal  category $\mathscr{B}\big(\mathscr{F}(A,\cH), \mathscr{F}(A,\cH) \big)$. Now, the image of a $2$-cell $\fk{f}: (P',\alpha',\beta') \to (P,\alpha,\beta)$ in $\mathsf{PB}^{\ell \, \co}(\cH,\cK)$ by $ \td{\mathscr{F}}$ is going to be a $2$-isomorphism: define 
$$
\td{\mathscr{F}}(\fk{f}):  \td{\mathscr{F}}(P',\alpha',\beta') =  \mathscr{F}(\B{\beta'})^{-1} \circ \mathscr{F}(\B{\alpha'}) \longrightarrow   \mathscr{F}(\B{\beta})^{-1} \circ \mathscr{F}(\B{\alpha}) = \td{\mathscr{F}}(P,\alpha,\beta) 
$$ 
as  the unique isomorphism in $\mathscr{B}\big(\mathscr{F}(A,\cH), \mathscr{F}(B,\cK)\big)$
satisfying 
$$
\mathscr{F}(\B{\beta'}) \circ \td{\mathscr{F}}(\fk{f}) \,=\, 
1_{\scriptscriptstyle{\mathscr{F}(\B{\alpha'})}} \,=\, 1_{\scriptscriptstyle{\mathscr{F}(\fk{f}) \circ \mathscr{F}(\B{\alpha})}}
$$
since from Diagram \eqref{Eq:qseramana} follows that $\fk{f} \circ \B{\alpha} = \B{\alpha'}$ and $\fk{f} \circ \B{\beta}=\B{\beta'}$ as $2$-cells in $2\text{-}\Sf{HAlgd}$, where, by abuse of notation, we did not distinguish between the vertical and horizontal composition in $\mathscr{B}$. 

The fact that $\td{\mathscr{F}}$ is compatible with both vertical and horizontal  compositions of $\mathsf{PB}^{\ell \, \co}$ is shown as follows:
first, as seen above, $\td{\mathscr{F}}\big( \mathscr{U}(\cH)\big) \cong  \bf{1}_{\scriptscriptstyle{\mathscr{F}(A,\, \cH)}}$ for every $0$-cell $(A,\cH)$. 
Second, for $(P, \alpha, \beta) \in \lPB{\cH}{\cK}$ and $(Q,\sigma,\theta) \in \lPB{\cK}{\cJ}$ consider their product  $$(P\cotensor{\cK}Q, \td{\alpha},\td{\theta}) \in \lPB{\cH \lJoin P \rJoin \cK}{\cK \lJoin Q \rJoin \cJ},$$ where $\td{\alpha}$ and $\td{\theta}$ are as in Diagram   \eqref{Eq:earthquake}.
Consider the morphism $\B{\sigma}:(B,\cK) \to (Q,\cK\lJoin Q \rJoin \cJ)$  of Hopf algebroids as in Lemma \ref{lemma:bisemidirect}.
From the trivial bundles $\B{\sigma}^*(\mathscr{U}(\cK)) \in \lPB{\cK}{\cK \lJoin Q\rJoin \cJ}$ and $\B{\beta}^*(\mathscr{U}(\cK)) \in \lPB{\cK}{\cH \lJoin P\rJoin \cK}$ we can construct their product 
$\B{\beta}^*(\mathscr{U}(\cK))^{\co} \cotensor{\cK} \B{\sigma}^*(\mathscr{U}(\cK))$, which  belongs to $\lPB{\cH \lJoin P\rJoin \cK}{\cK\lJoin Q\rJoin \cJ}$. On the other hand, an easy verification shows that $(P\tensor{B}\cK\tensor{B}Q, \gamma, \delta)$ 
is also a principal bundle in $\lPB{\cH\lJoin P\rJoin \cK}{\cK\lJoin Q\rJoin \cJ}$, where 
$$
\gamma: P \to P\tensor{B}\cK \tensor{B}Q, \quad p \mapsto p\tensor{B}1_{\scriptscriptstyle{\cK}}\tensor{B}1_{\scriptscriptstyle{Q}}; \qquad \delta: Q \to P\tensor{B}\cK \tensor{B}Q, \quad q \mapsto 1_{\scriptscriptstyle{P}}\tensor{B}1_{\scriptscriptstyle{\cK}} \tensor{B}q,
$$
and using the canonical bicomodule structure given by the coaction
\begin{small}
$$
P\tensor{B}\cK\tensor{B}Q \to (\cH\lJoin P\rJoin \cK) \tensor{P}(P\tensor{B}\cK\tensor{B}Q), \;\; p\tensor{B}w \tensor{B}q \mapsto (1_{\scriptscriptstyle{\cH}}\tensor{A}p\tensor{B}w_{(1)})\tensor{P}(1_{\scriptscriptstyle{P}}\tensor{B}w_{(2)}\tensor{B}q)
$$
\end{small}
as well as
\begin{small}
$$
P\tensor{B}\cK\tensor{B}Q \to (P\tensor{B}\cK\tensor{B}Q)  \tensor{Q} (\cK\lJoin Q\rJoin \cJ), \;\;  p\tensor{B}w \tensor{B}q \mapsto (p\tensor{B}w_{(1)} \tensor{B}1_{\scriptscriptstyle{Q}})\tensor{Q}(w_{(2)}\mathscr{S}(q_{(-1)})\tensor{B}q_{(0)}\tensor{C}q_{(1)}).
$$
\end{small}
Taking into account the canonical isomorphism 
$$
\B{\beta}^*(\mathscr{U}(\cK))^{\co} \cotensor{\cK} \B{\sigma}^*(\mathscr{U}(\cK)) \,=\, \big( P\tensor{\scriptscriptstyle{\beta}}\cK\big) \cotensor{\cK} \big(\cK \tensor{\scriptscriptstyle{\sigma}}Q \big) \cong P\tensor{B}\cK\tensor{B}Q
$$ 
of bicomodule algebras, we can then identify both principal bundles.  The two-sided translation Hopf algebroids associated to $(P\cotensor{\cK}Q, \td{\alpha},\td{\theta})$ resp.\ $(P\tensor{B}\cK\tensor{B}Q, \gamma, \delta)$ are now related via the morphism
$$
\B{\muup}: \Big(P\cotensor{\cK}Q,\, \cH\lJoin (P\cotensor{\cK}Q)\rJoin \cJ \Big) \to \Big(P\tensor{B}\cK\tensor{B}Q,\, (\cH \lJoin P\rJoin \cK) \lJoin(P\tensor{B}\cK\tensor{B}Q)\rJoin ( \cK\lJoin Q \rJoin \cJ)\Big)
$$ 
of Hopf algebroids,
sending 
$$
(p'\cotensor{\cK}q', u\tensor{A}(p\cotensor{\cK}q) \tensor{B} j)  \mapsto \big( p'_{(0)}\tensor{B}p'_{(1)}\tensor{B}q',\, \B{\td{\alpha}}(u) \tensor{P} (p_{(0)}\tensor{B}p_{(1)}\tensor{B}q) \tensor{Q} \B{\td{\theta}}(j) \big),
$$
where $\B{\td{\alpha}}$ and $\B{\td{\theta}}$ are the associated maps to $\td{\alpha}$ and $\td{\theta}$ as in Lemma \ref{lemma:bisemidirect}, and from which we deduce the following commutative diagram:
\begin{small}
$$
\xymatrix@R=15pt{   & &  \scriptstyle{\Big(P\cotensor{\cK}Q,\; \cH \lJoin (P\cotensor{\cK}Q) \rJoin \cJ \Big)} \ar@{->}^-{\scriptstyle{\B{\muup}}}[d]  & & \\  & &  \scriptstyle{\Big(P\tensor{B}\cK\tensor{B}Q,\; (\cH\lJoin P\rJoin \cK) \lJoin (P\tensor{B}\cK\tensor{B}Q) \rJoin (\cK\lJoin Q\rJoin \cJ)\Big)} & & \\  & \scriptstyle{(P,\cH\lJoin P\rJoin \cK)}  \ar@{->}^-{\scriptstyle{\B{\gamma}}}[ur] & & \scriptstyle{(Q, \cK \lJoin Q\rJoin \cJ)}  \ar@{->}_-{\scriptstyle{\B{\delta}}}[ul] & \\  \scriptstyle{(A,\cH)} \ar@/^4pc/^-{\scriptstyle{\B{\td{\alpha}}}}[uuurr]  \ar@{->}_-{\scriptstyle{\B{\alpha}}}[ur] & &  \ar@{->}^-{\scriptstyle{\B{\beta}}}[ul] \scriptstyle{(B,\cK)}  \ar@{->}_-{\scriptstyle{\B{\sigma}}}[ur]  & & \scriptstyle{(C,\cJ)}. \ar@{->}^-{\scriptstyle{\B{\theta}}}[ul]  \ar@/_4pc/_-{\scriptstyle{\B{\td{\theta}}}}[uuull]  } 
$$
\end{small}
Applying the functor $\mathscr{F}$ to this diagram and taking into account that $\B{\beta}, \B{\delta}, \B{\theta}$, and $\B{\td{\theta}}$ are weak equivalences by Proposition \ref{prop:Cielitomissyoutoo} (while $\B{\alpha}$ and $\B{\sigma}$ are not necessarily so since $P$ and $Q$ are just left bundles),  we obtain the equality
$$
\mathscr{F}(\B{\td{\theta}})^{-1} \circ \mathscr{F}(\B{\td{\alpha}}) \,=\, \mathscr{F}(\B{\theta})^{-1} \circ \mathscr{F}(\B{\sigma}) \circ  \mathscr{F}(\B{\beta})^{-1} \circ \mathscr{F}(\B{\alpha}),
$$ 
which means that 
$$
\td{\mathscr{F}}(Q,\sigma,\theta) \circ  \td{\mathscr{F}}(P,\alpha,\beta)  \,=\, \td{\mathscr{F}}(P\cotensor{\cK}Q, \td{\alpha},\td{\theta}), 
$$ 
that is, $\td{\mathscr{F}}$ is contravariant (in the proof of Proposition  \ref{prop:P} we saw that $\mathscr{P}$ is also contravariant, hence $\td{\mathscr{F}} \circ \mathscr{P}$ is covariant). To show that $\td{\mathscr{F}}$ is unique up to isomorphism, one uses Lemma  
\ref{lemma:paciencia}. Finally, to check that the Diagram \rmref{tiburtina} is commutative up to $2$-isomorphism, one makes use of Lemma \ref{lemma:volveranotro}. 
\end{proof}

\section{Principal bibundles and  Morita equivalences of categories of comodules}
In this section, which contains one of our main results (Theorem \ref{processocivile} in the Introduction), we explore the relationship between bibundles and Morita theory 
motivated by Theorem \ref{thm:AG}.   We remind the reader that, as in Definition \ref{quellala2}, two flat Hopf algebroids are said to be \emph{Morita equivalent} if their categories of (right) comodules are equivalent as symmetric monoidal categories.

\subsection{Principal bibundles versus monoidal equivalence}
The result we want to prove first and which will be part of the main theorem reads as follows:
\begin{thm}
\label{aromanflower}
Let $(A, \cH)$ and $(B, \cK)$ be two flat Hopf algebroids and 
$(P, \ga, \gb)$ be a principal $(\cH, \cK)$-bibundle. Then the functor
$$
- \cotensor{\cH} P: \rcomod{\cH} \lra \rcomod{\cK}
$$
induces a symmetric monoidal equivalence of categories.
\end{thm}

\begin{proof}
Let us first check that the functor is symmetric monoidal: 
by Remark \ref{rema:PB} {\em (ii)}, there is an algebra isomorphism
$$
A \cotensor{\cH} P \simeq P^{\coinv} \simeq B
$$
as $\gb$ is injective. Second, for two right $\cH$-comodules $M$ and $N$ define the map
$$
\gd: (M \cotensor{\cH} P) \otimes^\behhe (N \cotensor{\cH} P) \to (M \otimes^\ahha N)\cotensor{\cH} P, \quad 
(m \cotensor{\cH} p) \otimes^\behhe (n \cotensor{\cH} p') \mapsto (m \otimes^\ahha n) \cotensor{\cH} pp',
$$
which is a morphism of right $\cK$-comodules, where the tensor products are those of comodules as explained in Remark \ref{mondragone}.
In order to show that $\gd$ is an isomorphism, we proceed similarly as before
and show that $\gd \otimes_\behhe \id_\pehhe$ is an isomorphism since $P$ is faithfully flat over $B$. Now a straightforward verification proves that the composition 
$$
\zeta_{\scriptscriptstyle{M \otimes_\ahha N}} \circ  (\gd \otimes_\behhe \id_\pehhe): \big((M \cotensor{\cH} P) \otimes_\behhe (N \cotensor{\cH} P)\big)\otimes_\behhe P \to  (M \otimes_\ahha N) \otimes_\ahha P,
$$ 
using the natural transformation $\zeta$ from \rmref{giotto1}, coincides with the following chain 
\begin{equation*}
\begin{split}
\big((M \cotensor{\cH} P) \otimes_\behhe (N \cotensor{\cH} P)\big) \otimes_\behhe P &\overset{\id \otimes_\behhe \zeta_\enne}{ \relbar\joinrel\relbar\joinrel\longrightarrow}  (M \cotensor{\cH} P) \otimes_\behhe (N \otimes_\ahha P) \overset{\simeq}{\lra}  \big((M \cotensor{\cH} P) \otimes_\behhe N\big) \otimes_\ahha P  \\
&\overset{\zeta_\enne \otimes_\ahha \id_\enne}{\relbar\joinrel\relbar\joinrel\longrightarrow} (M \otimes_\ahha P) \otimes_\ahha N  \overset{\simeq}{\lra} (M \otimes_\ahha N) \otimes_\ahha P
\end{split}
\end{equation*}
of isomorphisms,
where the last step simply uses the tensor flip and the associativity of the tensor product. Clearly, $\gd$ is a natural transformation and compatible with the symmetry of the tensor product of comodules.

Now we check that $- \cotensor{\cH} P$ is an equivalence of categories, using the natural transformation
$$
\eta_\emme: M \to (M \cotensor{\cH} P) \cotensor{\cK} P^{\co},
\quad m \mapsto (m_{(0)} \cotensor{\cH} m_{(1)+}) \cotensor{\cK} m_{(1)-}
$$
for any right $\cH$-comodule $M$ from \rmref{giotto2}. As above, one shows that $\eta_\emme \otimes_\ahha P$ is an isomorphism by using the natural transformation $\zeta_{-}$ from \rmref{giotto1}. Explicitly, the inverse of $\eta_\emme \otimes_\ahha P$ is given by 
$$
\xymatrix@C=30pt{\big((M \cotensor{\cH} P) \cotensor{\cK} P^{\co}\big)\tensor{A}P \ar@{->}^-{\zeta_{\scriptscriptstyle{M \cotensor{\cH} P}}}[r] &  \ar@{->}^-{\zeta_{\scriptscriptstyle{M }}}[r] (M \cotensor{\cH} P) \tensor{B} P & M\tensor{A}P, } 
$$ where the first $\zeta$ corresponds to the left principal bundle $P^{\co}$ while the second one corresponds to $P$. 
One therefore has a natural isomorphism
$$
( - \cotensor{\cK} P^{\co}) \circ ( - \cotensor{\cH} P) \overset{\simeq}{\lra} \id_{\rcomod{\cH}}.
$$
Analogously, one obtains a natural isomorphism 
$
( - \cotensor{\cH} P) \circ ( - \cotensor{\cK} P^{\co}) \to \id_{\rcomod{\cK}},
$
which concludes the proof.
\end{proof}

The converse of Theorem \ref{aromanflower} will be investigated in  the next section; however, we give here a partial  answer when two Hopf algebroids are weakly equivalent.

\begin{proposition}
\label{prop:PHPKP}
Two flat Hopf algebroids $(A,\cH)$ and $(B,\cK)$ are weakly equivalent if and only if there is a principal bibundle connecting them.
\end{proposition}
\begin{proof}
The implication $(\Leftarrow)$  directly follows from part {\em (iii)} of Proposition \ref{prop:Cielitomissyoutoo}. As for the opposite direction $(\Rightarrow)$, assume that there is a diagram 
$$
\xymatrix@R=15pt{ & (C,\cJ)  & \\ (A,\cH) \ar@{->}^-{\B{\varphiup}}[ur] & & \ar@{->}_-{\B{\omegaup}}[ul] (B,\cK) }
$$
of flat Hopf algebroids,
where $\B{\varphi}$ and $\B{\omegaup}$ are weak equivalences.  Denote the associated trivial bundles
by $P:=\cK\tensor{\omega}C$ and $Q:=C\tensor{\varphi}\cH$. 
As shown in Proposition \ref{lemma:trivial bibundle} and explained in Remark \ref{jumpingwindow},
$P \in \bPB{\cK}{\cJ}$ and $Q \in \bPB{\cJ}{\cH}$ are trivial bibundles, and we can form the bundle $P\cotensor{\cJ}Q$, which is an object in $\bPB{\cK}{\cH}$, or equivalently $(P\cotensor{\cJ}Q)^{\co} \in \bPB{\cH}{\cK}$, and this finishes the proof.
\end{proof}

\subsection{Symmetric monoidal equivalence versus principal bibundles}\label{sec:tesoro} Starting with two Morita equivalent flat Hopf algebroids, the aim of this subsection is to extract from these data a principal bibundle.
To this end, let us first recall some basic facts on monoidal functors, restricting ourselves to the case of monoidal categories of comodules over flat Hopf algebroids. 

Let $(A, \cH)$ and $(B,\cK)$ be two flat Hopf algebroids, and assume 
that there is a symmetric monoidal equivalence 
$$
\mathcal{F}: \rcomod{\cH} \longrightarrow \rcomod{\cK}
$$
with inverse $\mathcal{G}$ in what follows. In particular, this means that there is a  natural isomorphism
\begin{equation}\label{Eq:Phi}
\phiup^{\scriptscriptstyle{1}}_{-,-}: \mathcal{F}( -{\comdtensor{A}}-)\, \longrightarrow\, \mathcal{F}(-){\comdtensor{B}} \mathcal{F}(-),  
\qquad \phiup^{\scriptscriptstyle{0}}: B \overset{\cong}{\longrightarrow} \mathcal{F}(A),
\end{equation}
where the latter is an algebra isomorphism, and the notation $-{\comdtensor{A}}-$ stands for the tensor product of comodules as was explained in Remark \ref{mondragone}. Both $\phiup^{\scriptscriptstyle{1}}$ and $\phiup^{\scriptscriptstyle{0}}$  should be compatible in a coherent way with the associativity, the commutativity ({\em i.e.}, the symmetries), and the unitary property of the tensor products  of both $\rcomod{\cH}$ and $\rcomod{\cK}$. Notice that, in this case, there also exists a symmetric monoidal equivalence between {\em left} comodules. 

The inverse natural transformation of $\phiup$ will be denoted by $\psiup$. It is known  that the functor $\mathcal{G}$ is also a symmetric monoidal functor; its associated natural isomorphism can be computed  from  that of $\mathcal{F}$ by using the natural transformation defining the equivalence.

Now, let $M \in \bicomod{T}{\cH}$, where $T$ is any commutative algebra, {\em i.e.}, $M$ is a $(T,A)$-bimodule and right $\cH$-comodule with left $T$-linear  coaction.   
Then,  we have an algebra map
$$
\lambda_{\Sf{l}}: T \to {\rcomod{\cH}}(M, M), \quad t \mapsto \{ m \mapsto  tm \},
$$
which is used to get a new algebra map
$$
\xymatrix{A \ar@{->}^-{\lambda_{\Sf{l}}}[r] &  {\rcomod{\cH}}(M, M) \ar@{->}^-{\mathcal{F}}[r] &   {\rcomod{\cK}}\big(\mathcal{F}(M), \mathcal{F}(M)\big), }
$$
from which we obtain that $\mathcal{F}(M)$ is a $(T,B)$-bimodule and that its right coaction $\rcoaction{\cK}{\mathcal{F}(M)}$ is left $T$-linear, that is, $\mathscr{F}(M) \in \bicomod{T}{\cK}$. Moreover,  $\mathscr{F}$ is restricted to the functor $$\mathscr{F}: \bicomod{T}{\cH} \to \bicomod{T}{\cK}.$$

Following \cite[\S23 \& \S39.3]{BrzWis:CAC}, since $\mathcal{F} $ is right exact and commutes with inductive limits, there is a natural isomorphism over (right) modules $\Sf{Mod}_{T}$
\begin{equation}\label{Eq:Upsilon}
\Upsilon_{-,M}: \mathcal{F}(-\tensor{T}M)  \longrightarrow -\tensor{T}\mathcal{F}(M), 
\end{equation}
which is natural on $M$ as well, and where the functor $-\tensor{T}M: \rmod{T} \to \rcomod{\cH}$ is defined as in  \eqref{Eq:RM}.  Furthermore, $\Upsilon$ defines morphisms of right $\cK$-comodules.  Notice that $\Upsilon_{\scriptscriptstyle{T,\, M}}: \mathcal{F}(M) \to T\tensor{T}\mathcal{F}(M)$ is just the canonical map sending $x \mapsto 1_{\scriptscriptstyle{T}}\tensor{T} x$.

For instance, in case $M:=\cH$ with left $A$-action given by the source $\Sf{s}$, we obtain  an algebra map 
$$
\lambda_{\Sf{s}}: A \to {\rcomod{\cH}}(\cH, \cH), \quad a \mapsto \{ u \mapsto \Sf{s}(a)u \}.
$$
The composition 
$$
\xymatrix{A \ar@{->}^-{\lambda_{\Sf{s}}}[r] &  {\rcomod{\cH}}(\cH, \cH) \ar@{->}^-{\mathcal{F}}[r] &   {\rcomod{\cK}}\big(\mathcal{F}(\cH), \mathcal{F}(\cH)\big) }
$$
induces on  $\mathcal{F}(\cH)$ an $(A,B)$-bimodule structure with a left $A$-linear  right coaction $\rcoaction{\cK}{\mathcal{F}(\cH)}$.  In fact, $\cF(\cH)$ becomes  an $(\cH,\cK)$-bicomodule with these actions as follows.
The structure of a left $\cH$-comodule  
 is given by
\begin{equation}\label{Eq:lcoact}
\xymatrix{ \lcoaction{\cH}{\mathcal{F}(\cH)}: \mathcal{F}(\cH) \ar@{->}^-{\mathcal{F}(\Delta)}[r] &  \mathcal{F}(\cH\tensor{A}\cH) \ar@{->}_-{\cong}^-{\Upsilon}[r] &  \cH \tensor{A}\mathcal{F}(\cH), }
\end{equation}
using the natural isomorphism of Eq.~\eqref{Eq:Upsilon}, which can be shown to be a morphism of right $\cK$-comodules. Similar arguments hold true for $\cG$. Furthermore, we have  natural isomorphisms
\begin{equation}\label{Eq:FG}
\cF\,\cong \, -\cotensor{\cH}\cF(\cH), \quad \cG \,\cong \, -\cotensor{\cK}\cG(\cK).
\end{equation}

Since $\cH$ is a monoid in $\Sf{Comod}_{\cH}$, it follows that $\mathcal{F}(\cH)$ is a monoid in $\Sf{Comod}_{\cK}$. Thus, $\mathcal{F}(\cH)$ is a right $\cK$-comodule algebra with respect to the underlying algebra map  
\begin{equation}\label{Eq:betaFH}
\beta:  B \overset{\phiup^{\scriptscriptstyle{0}}}{\cong} \mathcal{F}(A) \overset{\mathcal{F}(\Sf{t})}{\longrightarrow} \mathcal{F}(\cH).  
\end{equation}
Explicitly, the multiplication in $\mathcal{F}(\cH)$ is given by
\begin{equation}\label{Eq:mFH}
\xymatrix@C=30pt{ \textsl{m}_{\scriptscriptstyle{\mathcal{F}(\cH)}}: \mathcal{F}(\cH)\comdtensor{B}  \mathcal{F}(\cH) \ar@{->}^-{\psiup_{\scriptscriptstyle{\cH,\, \cH}}}[r] & \mathcal{F}(\cH \comdtensor{A}\cH) \ar@{->}^-{\mathcal{F}(\textsl{m}_{\scriptscriptstyle{\cH}})}[r] & \mathcal{F}(\cH).  }
\end{equation}
Note that $\mathcal{F}(\cH)$ is commutative since $\phiup$ is so (preserves the symmetries) as well as $\cH$. 

Next, we want to endow $\mathscr{F}(\cH)$ with the structure of a left $\cH$-comodule algebra using the left comodule structure of Eq.~\eqref{Eq:lcoact}.
The $A$-algebra structure on $\mathcal{F}(\cH)$ is given by  the linear map
\begin{equation}\label{Eq:alpha}
\alpha: A \to \mathcal{F}(\cH), \quad a \mapsto \mathcal{F}(\lambda_{\Sf{s}}(a)) (1_{\scriptscriptstyle{\mathcal{F}(\cH)}})=a. 1_{\scriptscriptstyle{\mathcal{F}(\cH)}},
\end{equation}
where $1_{\scriptscriptstyle{\mathcal{F}(\cH)}}$ is just the identity element of the right $\cK$-comodule algebra $\mathcal{F}(\cH)$,  which can be identified with $\mathcal{F}(\Sf{t}) \circ \phiup^{\scriptscriptstyle{0}}(1_{\scriptscriptstyle{B}})= \mathcal{F}(\Sf{t})(1_{\scriptscriptstyle{\mathcal{F}(A)}})$. 
We have:

\begin{lemma}\label{lemma:mAFH}
The map $\alpha$ of Eq.~\eqref{Eq:alpha} is an algebra map.  That is, there exists a map which makes the diagram 
$$
\xymatrix{ \scriptstyle{\mathcal{F}(\cH)\tensor{}\mathcal{F}(\cH)} \ar@{->}[r]  \ar@{->}[d] & \scriptstyle{\mathcal{F}(\cH) \tensor{A} \mathcal{F}(\cH)} \ar@{-->}[rr] & & \scriptstyle{\mathcal{F}(\cH)}  \\ \scriptstyle{\mathcal{F}(\cH)\comdtensor{B} \mathcal{F}(\cH)} \ar@{->}^-{\scriptstyle{\psiup}}[r] & \scriptstyle{\mathcal{F}(\cH \comdtensor{A} \cH)} \ar@{->}_-{\scriptstyle{\mathcal{F}(\textsl{m}_{\scriptscriptstyle{\cH}})}}[rru] &  &  }
$$
commutative.
\end{lemma}
\begin{proof}
It is clear that $\alpha(1_{\scriptscriptstyle{A}}) =1_{\scriptscriptstyle{\mathcal{F}(\cH)}}$ since $\mathcal{F}(\lambda_{\Sf{s}}(1_{\scriptscriptstyle{A}}))=id_{\scriptscriptstyle{\mathcal{F}(\cH)}}$. Now, for $a, a' \in A$ compute 
\begin{eqnarray*}
\textsl{m}_{\scriptscriptstyle{\mathcal{F}(\cH)}}\big( \alpha(a)\tensor{B}\alpha(a') \big) &=& \mathcal{F}(\textsl{m}_{\scriptscriptstyle{\cH}}) \circ \psiup_{\scriptscriptstyle{\cH,\, \cH}}\, \Big(   \mathcal{F}(\lambda_{\Sf{s}}(a)) (1_{\scriptscriptstyle{\mathcal{F}(\cH)}}) \tensor{B} \mathcal{F}(\lambda_{\Sf{s}}(a')) (1_{\scriptscriptstyle{\mathcal{F}(\cH)}}) \Big)
\\ &=& \mathcal{F}(\textsl{m}_{\scriptscriptstyle{\cH}}) \circ \psiup_{\scriptscriptstyle{\cH,\, \cH}}\circ \big( \mathcal{F}(\lambda_{\Sf{s}}(a)) \comdtensor{B} \mathcal{F}(\lambda_{\Sf{s}}(a')) \big) \, \big(1_{\scriptscriptstyle{\mathcal{F}(\cH)}}  \tensor{B} 1_{\scriptscriptstyle{\mathcal{F}(\cH)}} \big)
\\ &=& \mathcal{F}(\textsl{m}_{\scriptscriptstyle{\cH}}) \circ   \mathcal{F}\big(\lambda_{\Sf{s}}(a) \comdtensor{A} \lambda_{\Sf{s}}(a')\big) \circ  \psiup_{\scriptscriptstyle{\cH,\, \cH}} \, \big(1_{\scriptscriptstyle{\mathcal{F}(\cH)}}  \tensor{B} 1_{\scriptscriptstyle{\mathcal{F}(\cH)}} \big)
\\ &=& \mathcal{F}\big( \textsl{m}_{\scriptscriptstyle{\cH}} \circ   (\lambda_{\Sf{s}}(a) \comdtensor{A} \lambda_{\Sf{s}}(a')) \big) \circ  \psiup_{\scriptscriptstyle{\cH,\, \cH}}  \, \big(1_{\scriptscriptstyle{\mathcal{F}(\cH)}}  \tensor{B} 1_{\scriptscriptstyle{\mathcal{F}(\cH)}} \big)
\\ &=& \mathcal{F}\big(  \lambda_{\Sf{s}}(aa') \circ   \textsl{m}_{\scriptscriptstyle{\cH}} \big) \circ  \psiup_{\scriptscriptstyle{\cH,\, \cH}}  \, \big(1_{\scriptscriptstyle{\mathcal{F}(\cH)}}  \tensor{B} 1_{\scriptscriptstyle{\mathcal{F}(\cH)}} \big)
\\ &=& \mathcal{F}( \lambda_{\Sf{s}}(aa'))  \circ \mathcal{F}(   \textsl{m}_{\scriptscriptstyle{\cH}}) \circ  \psiup_{\scriptscriptstyle{\cH,\, \cH}}  \, \big(1_{\scriptscriptstyle{\mathcal{F}(\cH)}}  \tensor{B} 1_{\scriptscriptstyle{\mathcal{F}(\cH)}} \big)
\\ &=& \mathcal{F}( \lambda_{\Sf{s}}(aa'))  \circ    \textsl{m}_{\scriptscriptstyle{\mathcal{F}(\cH)}} \big(1_{\scriptscriptstyle{\mathcal{F}(\cH)}}  \tensor{B} 1_{\scriptscriptstyle{\mathcal{F}(\cH)}} \big)
\\ &=& \mathcal{F}( \lambda_{\Sf{s}}(aa'))   \big( 1_{\scriptscriptstyle{\mathcal{F}(\cH)}} \big) \,\,=\,\, \alpha(aa').
\end{eqnarray*}
As  the last statement is obvious, this finishes the proof.
\end{proof}

In order to show that the coaction \eqref{Eq:lcoact} is an algebra map with respect to $\alpha$, we  need to introduce the following natural transformations:
\begin{equation}\label{Eq:peleroja}
\begin{array}{rcl}
\B{\Omega}_{X,Y}: \big(X\tensor{A}\mathcal{F}(\cH)\big) \comdtensor{B}\big(Y\tensor{A}\mathcal{F}(\cH)\big) & \longrightarrow &   (X\tensor{A}Y)\tensor{A}\mathcal{F}(\cH),\quad \\  (x\tensor{A}p) \comdtensor{B}(y\tensor{A}q) &\longmapsto& (x\tensor{A}y) \tensor{A} pq,
\end{array}
\end{equation}  
\begin{equation}\label{Eq:pelerojasera}
\begin{array}{rcl}
\B{\nabla}_{X,Y}: \big(X\tensor{A}\cH\big) \comdtensor{A}\big(Y\tensor{A}\cH\big) & \longrightarrow&  (X\tensor{A}Y)\tensor{A}\cH, \quad \\ (x\tensor{A}u) \comdtensor{A}(y\tensor{A}v) &\longmapsto& (x\tensor{A}y) \tensor{A} uv,
\end{array}
\end{equation}
where $X$ and $Y$ are $A$-modules and where we used the multiplication in $\mathcal{F}(\cH)$. 
Using a functor similar to the one in \eqref{Eq:RM}, one sees that $\B{\Omega}$ defines  morphisms of right $\cK$-comodules
since the right $\cK$-coaction of $\mathcal{F}(\cH)$ is left $A$-linear (with respect to the $A$-action given by $\alpha$). Analogously, $\B{\nabla}$ defines morphisms of right $\cH$-comodules. These natural transformations are compatible in the following way:

\begin{proposition}\label{prop:Dstar}
The diagram 
\begin{small}
$$
\xymatrix@R=18pt@C=18pt{  & \scriptstyle{\mathcal{F}\big((X\tensor{A}\cH) \comdtensor{A} (Y\tensor{A}\cH) \big)} \ar@{->}^-{\scriptstyle{\mathcal{F}(\B{\nabla})}}[rr] 
 & &  \scriptstyle{\mathcal{F}((X\tensor{A}Y)\tensor{A}\cH)} \ar@{->}^-{\scriptstyle{\Upsilon}}[dd] \\  \scriptstyle{\mathcal{F}(X\tensor{A}\cH) \comdtensor{B} \mathcal{F}(Y\tensor{A}\cH)}  \ar@{->}_-{ \scriptstyle{\Upsilon\comdtensor{B} \Upsilon}}[dr]  \ar@{->}^-{\scriptstyle{\psiup}}[ru] & & & \\ & \scriptstyle{\big(X\tensor{A}\mathcal{F}(\cH)\big) \comdtensor{B}\big(Y\tensor{A}\mathcal{F}(\cH)\big)}   \ar@{->}^-{\scriptstyle{\B{\Omega}}}[rr] & & \scriptstyle{(X\tensor{A}Y)\tensor{A}\mathcal{F}(\cH)}   }
$$
\end{small}
of right $\cK$-comodules commutes.
\end{proposition}
\begin{proof}
First, notice that both $\Upsilon \circ \mathcal{F}(\B{\nabla})  \circ \psiup$ and $\B{\Omega} \circ  (\Upsilon\comdtensor{B} \Upsilon)$ are natural transformations on $(X,Y)$. Now, up to the canonical isomorphisms $A\tensor{A}\cH \cong \cH$ and $A\tensor{A}\mathcal{F}(\cH) \cong \mathcal{F}(\cH) $, we see that the diagram commutes for $X:=A$ and $Y:=A$ as this is just the definition of the multiplication $\textsl{m}_{\scriptscriptstyle{\mathcal{F}(\cH)}}$  defined in \eqref{Eq:mFH}. Using the naturality of both paths in the diagram, one can also show that the diagram  commutes when $X$ and $Y$ are free $A$-modules of finite rank. Since the
involved functors commute  with direct sums, the same holds true when $X$ and $Y$ are free $A$-modules. Lastly, since all involved functors 
are right exact, one can use  free representations of any $A$-module to complete the proof.  
\end{proof}

\begin{proposition}\label{prop:FHcomdalg}
The pair $(\mathcal{F}(\cH), \alpha)$ is a left $\cH$-comodule algebra with respect to the coaction \eqref{Eq:lcoact}.
\end{proposition}
\begin{proof}
We need to check that the map 
$\B{\lambda}= \Upsilon \circ \mathcal{F}(\Delta)$ 
in \eqref{Eq:lcoact}
is an algebra map. First, we prove unitality, that is,  $\B{\lambda}(1_{\scriptscriptstyle{\mathcal{F}(\cH)}})= \B{\lambda}(\alpha(1_{\scriptscriptstyle{A}})) =1_{\scriptscriptstyle{\cH}} \tensor{A} 1_{\scriptscriptstyle{\mathcal{F}(\cH)}}$: this follows from the commutative diagram 
\begin{small}
$$
\xymatrix{ \scriptstyle{\mathcal{F}(A)} \ar@{->}^-{\mathcal{F}(\Sf{t})}[rr]   \ar@{->}_-{\mathcal{F}(\Sf{t})}[d]  & &  \scriptstyle{\mathcal{F}(A)}  \ar@{->}^-{\mathcal{F}(\Delta)}[d] \\  \scriptstyle{\mathcal{F}(\cH)} \ar@{->}^-{\mathcal{F}(\Sf{t}\tensor{A}\cH)}[rr]  \ar@{->}_-{\Upsilon}[d] & & \scriptstyle{\mathcal{F}(\cH\tensor{A}\cH)}  \ar@{->}^-{\Upsilon}[d]  
 \\ \scriptstyle{A\tensor{A} \mathcal{F}(\cH)} \ar@{->}^-{\Sf{t}\tensor{A}\mathcal{F}(\cH)}[rr]  & & \scriptstyle{\cH\tensor{A}\mathcal{F}(\cH)}, } 
$$
\end{small}
where the left hand side $\Upsilon$ is just the canonical map $y \mapsto 1_{\scriptscriptstyle{A}}\tensor{A}y$. 

Now we proceed to check that $\B{\lambda}$ is multiplicative. To this end, we show that the diagram 
\begin{small}
$$
\xymatrix@C=45pt{  \scriptstyle{\mathcal{F}(\cH)\tensor{} \mathcal{F}(\cH)} \ar@{->}_-{\scriptstyle{}}[d] \ar@{->}^-{\scriptstyle{}}[r]  & \scriptstyle{\mathcal{F}(\cH)\comdtensor{A}\mathcal{F}(\cH)} \ar@{->}^-{\scriptscriptstyle{\textsl{m}}}[rr] & & \scriptstyle{\mathcal{F}(\cH)} \ar@{->}^-{\scriptscriptstyle{\mathcal{F}(\Delta)}}[dd] \\  \scriptstyle{\mathcal{F}(\cH)\comdtensor{B}\mathcal{F}(\cH)}   \ar@{->}|-{\scriptscriptstyle{\mathcal{F}(\Delta) \comdtensor{B}\mathcal{F}(\Delta) }}[dd] \ar@{->}^-{\scriptscriptstyle{\psiup}}[r] & \scriptstyle{\mathcal{F}(\cH\comdtensor{A}\cH)}  \ar@{->}|-{\scriptscriptstyle{\mathcal{F}(\Delta\comdtensor{A}\Delta) }}[dd]  \ar@{->}|-{\scriptscriptstyle{\mathcal{F}(\B{\rho})}}[rd]  \ar@{->}|-{\scriptscriptstyle{\mathcal{F}(\textsl{m})}}[rru]  &  & \\ & &  \scriptstyle{\mathcal{F}\big( (\cH\comdtensor{A}\cH) \tensor{A}\cH\big)} \ar@{->}^-{\scriptscriptstyle{\Upsilon}}[ddd] \ar@{->}^-{\scriptscriptstyle{\mathcal{F}(\textsl{m}\,\tensor{A}\cH) }}[r]   &  \scriptstyle{\mathcal{F}(\cH\comdtensor{A}\cH)}   \ar@{->}^-{\scriptscriptstyle{\Upsilon}}[ddd]  \\  \scriptstyle{\mathcal{F}(\cH\tensor{A}\cH) \comdtensor{B}  \mathcal{F}(\cH\tensor{A}\cH)}  \ar@{->}|-{\scriptscriptstyle{\Upsilon \comdtensor{B}\Upsilon }}[dd]  \ar@{->}^-{\scriptscriptstyle{\psiup}}[r]  & \scriptstyle{ \mathcal{F}\big( (\cH\tensor{A}\cH) \comdtensor{A} (\cH\tensor{A}\cH)\big) }  \ar@{->}|-{\scriptscriptstyle{\mathcal{F}(\B{\nabla})}}[ru]  & &  \\  & & & \\ \scriptstyle{\big(\cH\tensor{A}\mathcal{F}(\cH)\big) \comdtensor{B}\big(\cH\tensor{A}\mathcal{F}(\cH)\big) }  \ar@{->}^-{\scriptscriptstyle{\B{\Omega}}}[rr]   & &  \scriptstyle{ (\cH\comdtensor{A}\cH)\tensor{A}\mathcal{F}(\cH) } \ar@{->}^-{\scriptscriptstyle{\textsl{m}\,\tensor{A}\mathcal{F}(\cH) }}[r]  &  \scriptstyle{\cH\tensor{A}\mathcal{F}(\cH)}}
$$
\end{small}
is commutative, which  follows from Lemma \ref{lemma:mAFH}, Proposition \ref{prop:Dstar}, as well as from the very definitions of all involved maps and natural transformations.
\end{proof}

Our next aim is to show that $\mathcal{F}(\cH)$ is a principal left $(\cH, \cK)$-bundle with respect to $\alpha$ and $\beta$. As a start, the subsequent lemma concerns the faithfully flatness. 

\begin{lemma}\label{lemma:fplano}
Assume that there is a symmetric monoidal equivalence $$\mathcal{F}:  \Sf{Comod}_{\cH} \to \Sf{Comod}_{\cK}$$ with inverse $\mathcal{G}$. Then, for every right $\cH$-comodule $M$ whose underlying $A$-module is faithfully flat, 
$\mathcal{F}(M)$ is a faithfully flat $B$-module.
\end{lemma}
\begin{proof}
One can easily check that there is a natural isomorphism 
$$
\mathcal{O}_{\scriptscriptstyle{\cK}}(-) \tensor{B} \mathcal{F}(M) \overset{\cong}{\longrightarrow}  \mathcal{F}\big( \mathcal{G}(-)\otimes^\ahha M \big),
$$
where $\mathcal{O}_{\scriptscriptstyle{\cK}}: \rcomod{\cK} \to \rmod{B}$ denotes the forgetful functor.  
Hence, $\mathcal{O}_{\scriptscriptstyle{\cK}}(-) \tensor{B} \mathcal{F}(M)$ is a faithful and exact functor. Using the fact that $\mathcal{F}(M)$ carries the structure of a left $\cK$-comodule (in fact its opposite comodule), we see that $-\tensor{B}\mathcal{F}(M)$ is a faithful and exact functor.
\end{proof}

With the help of this lemma we can state:

\begin{proposition}\label{prop:FH}
The triple $(\mathcal{F}(\cH), \alpha,\beta)$ forms a left principal $(\cH,\cK)$-bundle. 
\end{proposition}
\begin{proof}
From Proposition \ref{prop:FHcomdalg} follows that $(\mathcal{F}(\cH), \alpha)$ is a left $\cH$-comodule algebra. Therefore,  $(\mathcal{F}(\cH), \alpha,\beta)$ is an $(\cH,\cK)$-bicomodule algebra since $(\mathcal{F}(\cH),\beta)$ is a right $\cK$-comodule algebra.

As $\thopf{\cH}$ is faithfully flat, $\mathcal{F}(\cH)_B$ is, using Lemma  \ref{lemma:fplano},  also faithfully flat and therefore $\beta$ is a faithfully flat extension. 
To complete the proof, we  need  to check that the canonical map 
$$ 
\xymatrix@C=40pt{\can{\cH}{\mathcal{F}(\cH)}: \mathcal{F}(\cH)\comdtensor{B}\mathcal{F}(\cH) \ar@{->}^-{\scriptstyle{\B{\lambda}\comdtensor{B}\mathcal{F}(\cH)}}[r] & \cH \tensor{A} \mathcal{F}(\cH) \comdtensor{B} \mathcal{F}(\cH) \ar@{->}^-{\scriptstyle{\cH \tensor{A} \, \textsl{m}}}[r] &  \cH \tensor{A} \mathcal{F}(\cH)  }
$$ 
is bijective. To this end, using Eqs.~ \eqref{Eq:lcoact} and \eqref{Eq:mFH} to express the coaction and the multiplication in $\cF(\cH)$, we write down the map $\can{\scriptscriptstyle{\cH}}{\scriptscriptstyle{\mathcal{F}(\cH)}}$  in  the diagram 
$$
\xymatrix@R=30pt@C=40pt{ \scriptstyle{\mathcal{F}(\cH)\comdtensor{B} \mathcal{F}(\cH)} \ar@{->}_-{\scriptscriptstyle{\psiup}}[d] \ar@{->}^-{\scriptscriptstyle{\mathcal{F}(\Delta)\comdtensor{B} \mathcal{F}(\cH)}}[rr]  &  & \scriptstyle{\mathcal{F}(\cH\tensor{A}\cH)\comdtensor{B} \mathcal{F}(\cH)} \ar@{->}^-{\scriptscriptstyle{\Upsilon\comdtensor{B} \mathcal{F}(\cH)}}[rr]   \ar@{->}_-{\scriptscriptstyle{\psiup}}[d]  & & \scriptstyle{\cH\tensor{A}\mathcal{F}(\cH)\comdtensor{B} \mathcal{F}(\cH)}  \ar@{->}_-{\scriptscriptstyle{\cH\tensor{A}\psiup}}[d]   \\ \scriptstyle{\mathcal{F}(\cH\comdtensor{A} \cH)}   \ar@{->}^-{\scriptscriptstyle{\mathcal{F}(\Delta\tensor{A}\cH)}}[rr] \ar@/_1pc/_-{\scriptscriptstyle{\mathcal{F}(\can{\cH}{\cH})}}[rrd]  & & \scriptstyle{\mathcal{F}\big( (\cH\tensor{A}\cH)\comdtensor{A} \cH\big) = \mathcal{F}\big( \cH\tensor{A}(\cH\comdtensor{A} \cH)\big) }  \ar@{->}^-{\scriptscriptstyle{\Upsilon}}[rr]  \ar@{->}_-{\scriptscriptstyle{\mathcal{F}(\cH\tensor{A}\,\textsl{m})}}[d] & & \scriptstyle{\cH\tensor{A}\mathcal{F}(\cH\comdtensor{A} \cH)}  \ar@{->}_-{\scriptscriptstyle{\cH\tensor{A}\mathcal{F}(\textsl{m})}}[d]  \\  & & \scriptstyle{\mathcal{F}(\cH\tensor{A} \cH)}  \ar@{->}^-{\scriptscriptstyle{\Upsilon}}[rr] & &  \scriptstyle{\cH\tensor{A} \mathcal{F}(\cH)}.  }
$$

Once shown that this diagram is commutative, it follows that the canonical map for $\cF(\cH)$  is bijective 
as $\can{\cH}{\mathcal{F}(\cH)} = \Upsilon \circ \mathcal{F}(\can{\cH}{\cH}) \circ \psiup$, where $\can{\cH}{\cH}$ is bijective being 
the canonical map of the unit bundle $\mathscr{U}(\cH)$.
To check that the above diagram is commutative, one only needs to show the commutativity of the  rectangle in the upper right.  This, in fact, forms  part of the well-known properties of the natural transformation $\Upsilon$; for the sake of completeness, we explain how this works: to start with, denote by $\mathcal{T}, \mathcal{S}: \rmod{A} \to \rcomod{\cK}$ the functors 
$$
\mathcal{T}(X)\,=\, \mathcal{F}(X\tensor{A}\cH)\comdtensor{B} \mathcal{F}(\cH),\quad \mathcal{S}(X) \,=\, \mathcal{F}\big( X\tensor{A}(\cH\comdtensor{A}\cH)\big).
$$   
Clearly, $\psiup_{(-\tensor{A}\cH),\,\cH}: \mathcal{T} \to \mathcal{S}$ is a natural transformation. Since $\mathcal{T}$ and $\mathcal{S}$ commute with direct limits, 
we have for every $A$-module $X$: 
$$ (X\tensor{A}\psiup_{(A\tensor{A}\cH),\,\cH}) \circ \Upsilon^{\scriptscriptstyle{\mathcal{T}}}_X \,\, =\,\, \Upsilon^{\scriptscriptstyle{\mathcal{S}}}_X \circ \psiup_{(X\tensor{A}\cH),\,\cH}.
$$
Using this equality for $X:=A$, we deduce the claim since $\Upsilon^{\scriptscriptstyle{\mathcal{T}}}_X=
\Upsilon^{\scriptscriptstyle{\mathcal{F}}}_X\comdtensor{B}\mathcal{F}(\cH)$ holds.
\end{proof}

\begin{corollary}
\label{coro:FHviaPo}
Let $(D, \cI)$ be another flat Hopf algebroid. Then the functor $\mathcal{F}$ restricts to a functor 
$$
\mathcal{F}: \lPB{\cI}{\cH} \longrightarrow \lPB{\cI}{\cK}.
$$
\end{corollary}
\begin{proof}
By Proposition \ref{prop:FH}, the triple $(\cF(\cH), \alpha,\beta)$
defines a principal left $(\cH, \cK)$-bundle; the cotensor product
$(R\cotensor{\cH}\mathcal{F}(\cH), \td{\delta}, \td{\beta})$, where $(R, \delta, \omega)$ is a principal left $(\cI, \cH)$-bundle, yields as in Lemma  \ref{lemma:cotensorpb} a principal left $(\cI, \cK)$-bundle. 
Then,  the first natural isomorphism of Eq.~\eqref{Eq:FG} leads to $R\cotensor{\cH}\mathcal{F}(\cH) \cong \mathcal{F}(R)$, which is an isomorphism of $(\cI, \cK)$-bicomodules, and this proves the claim.
\end{proof}

The following proposition (mentioned in Figure \ref{ahaaha}
in the Introduction) shows that two Morita equivalent Hopf algebroids are connected by a principal bibundle.

\begin{proposition}
\label{prop:llegaras}
Let $(A,\cH)$ and $(B,\cK)$ be two flat Hopf algebroids. 
Assume that there is a symmetric monoidal equivalence of categories $\mathcal{F}:  \Sf{Comod}_{\cH} \to \Sf{Comod}_{\cK}$ with inverse $\mathcal{G}$.  Then $(\mathcal{F}(\cH),\alpha,\beta)$ is a   principal  $(\cH,\cK)$-bibundle whose opposite bundle  is $\mathcal{G}(\cK)$.
\end{proposition}
\begin{proof}
Set $P:=\mathcal{F}(\cH)$ and $Q:=\mathcal{G}(\cK)$. From Proposition \ref{prop:FH} follows that $(P,\alpha, \beta)$ is a left principal $(\cH,\cK)$-bundle.  Interchanging $\mathcal{F}$ with $\mathcal{G}$, we also obtain that $(\mathcal{G}, \sigma, \theta)$ is  
a left principal $(\cK,\cH)$-bundle, where $\theta: A\cong \mathcal{G}(B) \to \mathcal{G}(\cK)$, and $\sigma$ is constructed in  the same way as was $\alpha$.  

On the other hand, using the equivalences $\mathcal{F}$ and $\mathcal{G}$ together with the natural transformations 
$$
\mathcal{F} \, \cong\, -\cotensor{\cH}P, \quad 
\mathcal{G} \, \cong\, -\cotensor{\cK}Q,
$$
of Eq.~\eqref{Eq:FG},  we obtain  the isomorphisms 
$$ 
P\cotensor{\cK} Q \cong \mathscr{U}(\cH),\qquad   Q\cotensor{\cH} P \cong \mathscr{U}(\cK)
$$
of $\cH$ and $\cK$-bicomodules, respectively, which fulfil the triangle properties \eqref{Eq:trianglI} and \eqref{Eq:trianglII}. 
This implies that $(P,\alpha,\beta)$ is an invertible $1$-cell in the category $\mathsf{PB}^\ell(\cH,\cK)$  of principal left bundles.
Now, conclude the proof by making use of Proposition \ref{prop:el parte} {\em (i)}.
\end{proof}

To sum up, we can state the main theorem of this article motivated by Theorem \ref{thm:AG} in the groupoid case:

\begin{theorem}
\label{thm:MAIN}
Let $(A,\cH)$ and $(B,\cK)$ be two flat Hopf algebroids. The following are equivalent:
\begin{enumerate}[(a)]
\item $(A,\cH)$ and $(B,\cK)$ are Morita equivalent.
\item There is a principal bibundle connecting $(A,\cH)$ and $(B,\cK)$.
\item $(A,\cH)$ and $(B,\cK)$ are weakly equivalent.
\end{enumerate}
\end{theorem}
\begin{proof}
The implication $(a) \Rightarrow (b)$ is Proposition \ref{prop:llegaras}, whereas the implication $(b) \Rightarrow (c)$ is contained in Proposition \ref{prop:PHPKP}. Finally, the step $(c) \Rightarrow (a)$ is obvious from the very definitions.
\end{proof}

\begin{rem}
As mentioned in Figure \ref{ahaaha}
in the Introduction, Theorem \ref{aromanflower} also states the implication $(b) \Rightarrow (a)$, whereas Proposition \ref{prop:PHPKP} moreover yields $(c) \Rightarrow (b)$. 
\end{rem}


\subsection{The categorical group of monoidal symmetric auto-equivalences}
\label{ssec:autoequi}
In this subsection,
we combine the results of Theorems \ref{thm:MAIN} and  \ref{aromanflower} by taking a single flat Hopf algebroid. More precisely, we show that all symmetric monoidal auto-equivalences of the category of right $\cH$-comodules form a categorical group with morphisms given by natural tensor transformations, 
and conclude that
this group
is equivalent to the categorical group of principal bibundles. 

Denote by $\Sf{Aut}^{\scriptscriptstyle{\otimes}}(A,\cH)$ the category of monoidal symmetric auto-equivalences of 
the category of (right) comodules $\rcomod{\cH}$ over a  flat Hopf algebroid $(A,\cH)$. 
Morphisms in this category are \emph{natural tensor transformations}, that is, natural transformations $\theta: \Ff \to \Ff'$ such that the diagrams
\begin{equation}
\label{Eq:BlackandWhite}
\xymatrix@C=45pt{ \Ff\big(X\comdtensor{A} Y\big)  \ar@{->}^-{\Theta_{\Sscript{X\comdtensor{A} Y}}}[r] \ar@{->}^-{\cong}_-{\phiup^{\Sscript{1,\, \Ff}}}[d] &  \Ff'\big(X\comdtensor{A}Y\big) \ar@{->}^-{\phiup^{\Sscript{1,\, \Ff'}}}_-{\cong}[d] \\ 
\Ff(X)\comdtensor{B} \Ff(Y)  \ar@{->}^-{\Theta_{\Sscript{X}} \comdtensor{B} \Theta_{\Sscript{Y}}}[r] &  \Ff'(X)\comdtensor{B}\Ff'(Y)  } \qquad  \xymatrix@C=25pt{ \Ff(A)  \ar@{->}^-{\Theta_{\Sscript{A}}}[rr]  &  &  \Ff'(A)  \\ 
& B \ar@{->}^-{\cong}_-{\phiup^{\Sscript{0,\, \Ff}}}[ur]  \ar@{->}^-{\phiup^{\Sscript{0,\,\Ff'}}}_-{\cong}[ul] &    }
\end{equation}
commute. 
Note that this gives a sets-category (in the sense that homomorphisms between two objects form a set) as $\rcomod{\cH}$ is a Grothendieck category and the involved functors preserve inductive limits.  The category $\Sf{Aut}^{\scriptscriptstyle{\otimes}}(A,\cH)$ is itself 
a monoidal category with multiplication given by 
the composition of functors and identity object given 
by the identity equivalence $\id_{\scriptscriptstyle{\rcomod{\cH}}}$. 
On the other hand, as in Subsection \ref{ssec:ppb}, we are interested in the monoidal category $\big(\bPB{\cH}{\cH},\cotensor{\cH}, \mathscr{U}(\cH) \big)$. Both categories are in fact \emph{categorical groups} (more precisely, a \emph{$2$-group} and a \emph{bigroup}) 
and are equivalent as such.

\begin{proposition}
\label{prop:pbmaps}
Let $(A,\cH)$ and $(B,\cK)$ be two flat Hopf algebroids, $\Ff,\Ff': \rcomod{\cH} \to \rcomod{\cK}$ two symmetric monoidal equivalences, and 
$\Theta: \Ff \to \Ff'$ a natural tensor transformation.  
Then $\Theta_{\cH}: \Ff(\cH) \to \Ff'(\cH)$ is a morphism of principal $(\cH,\cK)$-bibundles. In particular, $\Theta$ is a natural isomorphism and consequently $\big( \Sf{Aut}^{\scriptscriptstyle{\otimes}}(A,\cH), \B{\circ},\id_{\scriptscriptstyle{\rcomod{\cH}}}\big)$ is a categorical group.
\end{proposition}

\begin{proof}
By definition, $\Theta_{\cH}$ is a  morphism of right $\cK$-comodule algebras. Let us check that it is also a morphism of left $\cH$-comodule algebras. Recall that the respective left comodule algebra structure of both $\Ff(\cH)$ and $\Ff'(\cH)$ is given as in Proposition  \ref{prop:FHcomdalg}. 
That $\Theta_{\cH}$ is left $\cH$-colinear follows from the following diagram: 
\begin{small}
$$
\xymatrix@C=40pt@R=13pt{ &  \Ff\big(\cH\tensor{A} \cH\big)  \ar@{->}^-{\Ff(\Delta)}[rr] \ar@{->}^-{\Theta_{\Sscript{\cH\tensor{A}\cH}}}[ddd] & &  \Ff(\cH) \ar@{->}^-{\Theta_{\Sscript{\cH}}}[ddd] \\  \cH\tensor{A}\Ff(\cH) \ar@/_1pc/@{->}^-{\B{\lambda}}[rrru] \ar@{->}_-{\cong}^-{\Upsilon^{\Sscript{\Ff}}}[ru]  \ar@{->}_-{\cH\tensor{A}\Theta_{\Sscript{\cH}}}[ddd] & & & & \\  & & & &  \\ 
& \Ff'(\cH\tensor{A}\cH)  \ar@{->}^-{\Ff'(\Delta)}[rr] &  & \Ff'(\cH) \\  \cH\tensor{A}\Ff'(\cH) \ar@/_1pc/@{->}^-{\B{\lambda}}[rrru] \ar@{->}_-{\cong}^-{\Upsilon^{\Sscript{\Ff'}}}[ru] & & & &  } 
$$
\end{small}
where the left hand square is commutative by the universal property of the natural isomorphism $\Upsilon$.  The $A$-algebra structure of $\Ff(\cH)$ is given by the algebra map  $\alpha^{\Sscript{\Ff(\cH)}}: A \to \Ff(\cH)$, $a \mapsto \Ff(\lambda_{\Sscript{\Sf{s}}}(a))(1_{\Sscript{\Ff(\cH)}})$, and similarly for  $\Ff'(\cH)$, see Eq.~\eqref{Eq:alpha}. 
Thus, for any $a \in A$, we have  
$$ 
\Theta_{\Sscript{\cH}} \circ \alpha^{\Sscript{\Ff(\cH)}}(a)\,=\,  \Theta_{\Sscript{\cH}} \circ \Ff(\lambda_{\Sscript{\Sf{s}}}(a))(1_{\Sscript{\Ff(\cH)}}) \,=\, \Ff'(\lambda_{\Sscript{\Sf{s}}}(a)) \circ \Theta_{\Sscript{\cH}}(1_{\Sscript{\Ff(\cH)}}) \,=\, \Ff'(\lambda_{\Sscript{\Sf{s}}}(a))(1_{\Sscript{\Ff'(\cH)}})\,=\, \alpha^{\Sscript{\Ff'(\cH)}}(a) 
$$ 
since $\Theta_{\Sscript{\cH}}$ is a $B$-algebra map. 
Therefore, $\Theta_{\Sscript{\cH}}$ is an $A$-algebra map as it is multiplicative, and this finishes the proof of the first statement. 
Now, by Lemma \ref{lemma:ISO}, $\Theta_{\Sscript{\cH}}$ is an isomorphism and this suffices to show that $\Theta$ is a natural isomorphism: 
using the natural isomorphisms given in Eqs.~\eqref {Eq:Upsilon} and \eqref{Eq:FG}, one can see that the diagram 
$$
\xymatrix@C=40pt{ \Ff \ar@{->}^-{\Theta}[rr] \ar@{->}_-{\cong}[d] & & \Ff' \ar@{->}^-{\cong}[d]  \\ -\cotensor{\cH}\Ff(\cH) \ar@{->}_-{-\cotensor{\cH}\Theta_{\Sscript{\cH}}}[rr]  & &   -\cotensor{\cH}\Ff'(\cH) }
$$
of natural transformations
commutes, which means that $\Theta$ is a natural isomorphism.
\end{proof}

The following is Theorem \ref{thm:B} in the Introduction:

\begin{theorem}
\label{them:catggrp}
The functors 
\begin{equation*}
\begin{array}{rclrcl}
\big(\Sf{Aut}^{\scriptscriptstyle{\otimes}}(A,\cH), \B{\circ},\id_{\scriptscriptstyle{\rcomod{\cH}}}\big) 
&\longrightarrow& 
\big(\bPB{\cH}{\cH},\cotensor{\cH}, \mathscr{U}(\cH) \big), &
 \Ff &\longmapsto& \Ff(\cH) 
\\  
\big(\bPB{\cH}{\cH},\cotensor{\cH}, \mathscr{U}(\cH) \big)   &\longrightarrow& 
\big(\Sf{Aut}^{\scriptscriptstyle{\otimes}}(A,\cH), \B{\circ},\id_{\scriptscriptstyle{\rcomod{\cH}}}\big), & (P,\alpha,\beta) &\longmapsto& -\cotensor{\cH}P   
\end{array}
\end{equation*}
establish a monoidal equivalence of categorical groups.  
\end{theorem}
\begin{proof}
This essentially follows from Proposition \ref{prop:pbmaps}, Theorems  \ref{thm:MAIN} and  \ref{aromanflower}, in combination with Corollary \ref{coro:twoJapanese}.
\end{proof}


\appendix

\section{Some observations on coinvariant subalgebras}
As our guideline was to mimic the theory of principal bundles in the Lie groupoid context, we include for sake of completeness two results dealing with coinvariant subalgebras. They correspond to the statement that for any $\Gg$-equivariant submersion $Q \to P$,  where $P$ is a principal $\Gg$-bundle and $Q$ a $\Gg$-manifold, $Q/\Gg$ is a manifold as well and the canonical projection $Q \to Q/\Gg$ yields  a principal $\Gg$-bundle, see \cite[Lemma 2.8]{MoeMrc:LGSAC}.

\begin{proposition}
\label{prop:PQ}
Let $(Q, \sigma)$ be  a left $\cH$-comodule algebra, $F : P \to Q$ be an $\cH$-colinear injective map of $A$-rings, and $(P, \alpha, \beta)$ a trivial left principal $(\cH, \cK)$-bundle of the form $P:=\cH\tensor{\phi}B$. Consider the algebra $Q\tensor{P}B$, defined by using the splitting of $\beta$ in the second factor and $F$ in the first one. Then 
\begin{enumerate}
\item 
there is  an algebra isomorphism 
$$
T:= Q^{\scriptscriptstyle{\coinv_{\cH}}} \,\, \cong \,\, Q\tensor{P}B,
$$ 
and the canonical monomorphism $\tau: Q^{\scriptscriptstyle{\coinv_{\cH}}}\hookrightarrow Q$ splits as an algebra map; 
\item the triple  $(Q, \gs, \tau)$ is a left principal $(\cH,T)$-bundle.

\end{enumerate}
\end{proposition}
\begin{proof}
Denote by 
$$
\gamma: P := \cH\tensor{\phi}B \to B, \quad u\tensor{\phi}b \mapsto \phi_{\scriptscriptstyle{0}}(\varepsilon(u))b, 
$$
the splitting of $\beta$,
see Example \ref{exam:HopfMorph}.
To prove {\em (i)}, define first
$$
\omegaup:  Q\tensor{P}B \to Q, \quad q\tensor{P}b \mapsto q_{(0)} F\big(\mathscr{S}(q_{(-1)})\tensor{\phi}b \big),
$$ 
which via the map $(q \otimes_\behhe b, q') \mapsto \go(q \otimes_\behhe b)q'$ yields a left $(Q\tensor{P}B)$-action on $Q$. One can easily check that $\go$ is well-defined and has
$$
\kappaup:  Q \to Q\tensor{P}B, \quad q \mapsto q_{(0)} \tensor{P}\phi_{\scriptscriptstyle{0}}\big(\varepsilon(q_{(-1)})\big)
$$ 
as a splitting,
that is, $\kappaup \circ \omegaup = \id_{\scriptstyle{Q\tensor{P}B}}$. 
Since the image of $\omegaup$ lands in $Q^{\scriptscriptstyle{\coinv_{\cH}}}$, we can use this splitting to establish an isomorphism 
$Q^{\scriptscriptstyle{\coinv_{\cH}}} \cong Q\tensor{P}B$ of algebras, which also shows that $\tau$ is a  split monomorphism. 

To prove {\em (ii)}, we already know by part {\em (i)} that $\tauup$ splits, so in order to prove that $\tauup$ is faithfully flat, we only need to check that $\tauup$ is flat or, equivalently, that this is true for $\omegaup$. 
To this end, we will check that there is a  natural isomorphism
$$ 
-\tensor{Q\tensor{P}B}Q \ \to \ -\tensor{A}\thopf{\cH}, 
$$ 
where we consider $Q\tensor{P}B$ as an $A$-algebra via the map $\phi_{\scriptscriptstyle{0}}$ in the second factor. This will be sufficient since $\thopf{\cH}$ is flat. Let $X$ be a $(Q\tensor{P}B)$-module and consider the map 
$$
\vartheta: X\tensor{Q\tensor{P}B}Q \to X\tensor{A}\thopf{\cH}, \quad x \tensor{Q\tensor{P}B}q \mapsto \big( x(q_{(0)}\tensor{P}1_{\scriptscriptstyle{B}}) \big)\tensor{A}q_{(-1)},
$$ 
which is well-defined as the following consideration shows:
from one hand, we have 
\begin{eqnarray*}
\vartheta\big( (x(q\tensor{P}b)) \tensor{}q'\big) & =& (x(qq'_{(0)}\tensor{P}b)) \tensor{A}q'_{(-1)}.
\end{eqnarray*}
On the other hand, 
\begin{eqnarray*}
\vartheta\big( (x\tensor{}\omegaup(q\tensor{P}b)q'\big) &=& \vartheta\big(x \tensor{}\big(q_{(0)}F(\mathscr{S}(q_{(-1)})\tensor{\phi}b)\big)q' \big) \\ 
&=& x\Big(\big[q_{(0)}F\big( \mathscr{S}(q_{(-1)})\tensor{\phi}b\big)q'_{(0)}\big]\tensor{P}1_{\scriptscriptstyle{B}} \Big) \tensor{A}q'_{(-1)} \\ 
&=&  x\Big((q_{(0)}q'_{(0)})\tensor{P} 
\gamma(\mathscr{S}(q_{(-1)}) \otimes_\phi b)\Big) \tensor{A}q'_{(-1)} \\ 
&=&  x\Big((q_{(0)}q'_{(0)})\tensor{P} 
\phi_{\scriptscriptstyle{0}}\big( \varepsilon(q_{(-1)})\big)b \Big) \tensor{A}q'_{(-1)} \\ 
&=&   x\Big((q_{(0)}q'_{(0)})\tensor{P} \gamma\alpha\big( \varepsilon(q_{(-1)})\big)b \Big) \tensor{A}q'_{(-1)} \\ &=&   x\Big((q_{(0)}q'_{(0)} F\big( \Sf{s}(\varepsilon(q_{(-1)})) \tensor{\phi}1_{\scriptscriptstyle{B}}\big))\tensor{P}b \Big) \tensor{A}q'_{(-1)} \\ &=&   x\Big((q_{(0)}q'_{(0)} \sigma(\varepsilon(q_{(-1)})) )\tensor{P}b \Big) \tensor{A}q'_{(-1)}  \\ &=&   x\Big((qq'_{(0)} \tensor{P}b) \Big) \tensor{A}q'_{(-1)}  \\ &=& 
\vartheta\big( (x(q\tensor{P}b)) \tensor{}q'\big), 
\end{eqnarray*} 
which shows the well-definedness of $\gvt$.
The inverse of $\vartheta$ is now given by 
$$
\vartheta^{-1}: X\tensor{A}\thopf{\cH} \to X\tensor{Q\tensor{P}B}Q, 
\quad x\tensor{A}u \mapsto x \tensor{Q \tensor{P} B}F(u\tensor{\phi}1_{\scriptscriptstyle{B}}),
$$ 
and the fact that $\vartheta$ is a natural transformation is easily checked from the definition. 
Let us finally check that the canonical map $\mathsf{can}: Q \tensor{T} Q \to \cH \tensor{A} Q$ is bijective; define
$$
\mathsf{can}^{-1}: \cH\tensor{A}Q \to Q \tensor{T} Q, \quad u\tensor{A}q \mapsto F\big( u_{(1)}\tensor{\phi}1_{\scriptscriptstyle{B}}\big) \tensor{T} F \big(\mathscr{S}(u_{(2)}) \tensor{\phi}1_{\scriptscriptstyle{B}}\big)q,
$$
and we leave it to the reader to check that this is the desired inverse, indeed.
\end{proof}

In case that $P$ is no longer trivial, we can make the following statement:

\begin{proposition}
\label{coro:PQ}
Let $(Q, \sigma)$, and $F : P \to Q$ be as in Proposition  \ref{prop:PQ} and  $(P, \alpha, \beta)$ any left principal $(\cH,\cK)$-bundle. 
Then the canonical map
$$
\mathsf{can}: Q \otimes_{\tehhe} Q \to \cH \otimes_\ahha Q, \quad q \otimes_{\tehhe} q' \mapsto q_{(-1)} \otimes_\ahha q_{(0)}q'
$$
 is bijective, where $\tau: Q^{\coinv_\cakka}=:T  \to Q$  is the canonical monomorphism. 
\end{proposition}
\begin{proof}
Define a map
$$
\tilde{\mathsf{can}}: \cH \otimes_\ahha Q \to Q \otimes_\tehhe Q, \quad u \otimes_\ahha q \mapsto F(u_+) \otimes_\tehhe F(u_-)q,
$$
and we will explicitly compute that $\mathsf{can} \circ \tilde{\mathsf{can}} = \id_{\scriptscriptstyle{ \cH \otimes_\ahha Q}}$ along with $\tilde{\mathsf{can}} \circ \mathsf{can} = \id_{\scriptscriptstyle{Q \otimes_\tehhe Q}}$. Since $F$ is an $\cH$-colinear morphism of algebras, one sees that
\begin{equation*}
(\mathsf{can} \circ \tilde{\mathsf{can}})(u \otimes_\ahha q) =  (F(u_+))_{(-1)} \otimes_\ahha (F(u_+))_{(0)}F(u_-)q = u_{+(-1)} \otimes_\ahha F(u_{+(0)} u_-)q = u \otimes_\ahha q,
\end{equation*}
using \rmref{Eq:u-+}. On the other hand, 
\begin{equation*}
\begin{split}
(\tilde{\mathsf{can}} \circ {\mathsf{can}})(q \otimes_\tehhe q') 
&=   F(q_{(-1)+}) \otimes_\tehhe F(q_{(-1)-})q_{(0)}q' \\
&=  F(q_{(-1)+}) F(q_{(-1)-}) q_{(0)} \otimes_\tehhe q' \\
& = F(\ga(\varepsilon(q_{(-1)}))) q_{(0)} \otimes_\tehhe q'  = q \otimes_\tehhe q',
\end{split}
\end{equation*}
using \rmref{Eq:eps+-} in the third step, and where the second step is justified by the fact that an element of the form 
$q_{(-1)+} \otimes_\behhe F(q_{(-1)-}) q_{(0)} \in P \otimes_\behhe Q$ actually lies in $P \otimes_\behhe T = P \otimes_\behhe Q^{\coinv_\cH}$, which we show now: 
\begin{equation*}
\begin{split}
(\id_{\pehhe} \otimes_\behhe \B{\lambda})(q_{(-1)+} \otimes_\behhe F(q_{(-1)-}) q_{(0)}) & = q_{(-2)+} \otimes_\behhe  q_{(-2)-(-1)} q_{(-1)} \otimes_\ahha F(q_{(-2)-(0)}) q_{(0)} \\
& =  q_{(-3)+} \otimes_\behhe  \mathscr{S}(q_{(-2)}) q_{(-1)} \otimes_\ahha F(q_{(-3)-}) q_{(0)} \\
& =  q_{(-2)+} \otimes_\behhe  \mathsf{t}(\varepsilon(q_{(-1)})) \otimes_\ahha F(q_{(-2)-}) q_{(0)} \\
&= q_{(-1)+} \otimes_\behhe  1_\cH \otimes_\ahha F(q_{(-1)-}) q_{(0)}, 
\end{split}
\end{equation*}
where we used the $\cH$-colinearity of $F$ together with \rmref{Eq:S+-} and \rmref{Eq:eps+-}.
\end{proof}

\begin{rem}
If one were able to show that $\tauup$ is a faithfully flat extension, then the  triple $(Q, \gs, \tau)$ became a left principal $(\cH, T)$-bundle. 
\end{rem}

\providecommand{\bysame}{\leavevmode\hbox to3em{\hrulefill}\thinspace}
\providecommand{\MR}{\relax\ifhmode\unskip\space\fi MR }
\providecommand{\MRhref}[2]{%
  \href{http://www.ams.org/mathscinet-getitem?mr=#1}{#2}
}
\providecommand{\href}[2]{#2}

\end{document}